\documentclass[12pt]{article}

\usepackage[paperwidth=230mm, paperheight=310mm]{geometry}
\usepackage{graphicx}
\usepackage{amssymb}
\usepackage{amsmath}
\usepackage{enumerate}
\usepackage{amsthm}
\usepackage{amsmath}
\usepackage{float}
\usepackage{color}      
\usepackage{lineno}

\makeatletter
\renewenvironment{proof}[1][\proofname]{%
   \par\pushQED{\qed}\normalfont%
   \topsep6\p@\@plus6\p@\relax
   \trivlist\item[\hskip\labelsep\bfseries#1\@addpunct{.}]%
   \ignorespaces
}{%
   \popQED\endtrivlist\@endpefalse
}
\makeatother

\newtheorem{theorem}{Theorem}
\newtheorem{proposition}[theorem]{Proposition}
 \numberwithin{theorem}{section}
 \newtheorem{corollary}[theorem]{Corollary}
\newtheorem{lemma}[theorem]{Lemma}

\newtheorem{remark}[theorem]{Remark}

\numberwithin{equation}{section}

\renewcommand{\P}{\mathbb{P}}
\newcommand{\E}{\mathbb{E}}
\newcommand{\R}{\mathbb{R}}
\newcommand{\mR}{\mathcal{R}}
\newcommand{\mK}{\mathcal{K}}

\newcommand{\pmR}{\partial\mathcal{R}}

\newcommand{\Q}{\mathbb{Q}}
\newcommand{\N}{\mathbb{N}}

\newcommand{\cF}{\mathcal F}
\newcommand{\cE}{\mathcal E}
\newcommand{\cG}{\mathcal G}
\newcommand{\cW}{\mathcal W}

\newcommand{\veps}{\varepsilon}
 \newcommand{\nn}{\nonumber}
 \newcommand{\no}{\noindent}
 \newcommand{\lb}{\mathsf{L}}
\newcommand{\rb}{\mathsf{R}}
\newcommand{\pf}[1]{{\color{black}#1}}
\renewcommand{\pf}[1]{}
\newcommand{\SUBMIT}[1]{{\color{red}#1}}
\renewcommand{\SUBMIT}[1]{}
\newcommand{\ARXIV}[1]{{\color{black}#1}}

\begin{document}

\author{
Jieliang Hong\footnote{Department of Mathematics, The University of British Columbia. {\tt jlhong@math.ubc.ca}}
\and Leonid Mytnik\footnote{Faculty of Industrial Engineering and Management, Technion.
{\tt leonid@ie.technion.ac.il}}
\and Edwin Perkins \footnote{Department of Mathematics, The University of British Columbia.  {\tt perkins@math.ubc.ca}}
}

\title{On the topological boundary of the range of super-Brownian motion\ARXIV{--extended version}}

\maketitle

\begin{abstract}

We show that if $\pmR$ is the boundary of the range of super-Brownian motion and dim denotes Hausdorff dimension, then with probability one, for any open set $U$, $\pmR\cap U\neq\emptyset$ implies
$$\textnormal{dim}(U\cap\pmR)=\begin{cases}
4-2\sqrt2\approx1.17&\text{ if }d=2\\
\frac{9-\sqrt{17}}{2}\approx 2.44&\text{ if }d=3.
\end{cases}$$
This improves recent results of the last two authors by working with the actual topological boundary, rather than
the boundary of the zero set of the local time, and establishing a local result for the dimension.
\end{abstract}

\section{Introduction}

We consider a $d$-dimensional super-Brownian motion (SBM), $(X_t,t\ge0)$, starting at $X_0$ under $\P_{X_0}$ with $d\le 3$.  Here $X_0\in M_F(\R^d)$, the space of finite measures on $\R^d$ with the weak topology, $X$ is a continuous $M_F(\R^d)$-valued strong Markov process, and $\P_{X_0}$ denotes any probability under which $X$ is as above.  We write $X_t(\phi)$ for the integral of $\phi$ with respect to $X$, and take our branching rate to be one, so that for any non-negative bounded Borel functions $\phi,f$ on $\R^d$, 
\begin{equation}\label{LF}
\E_{X_0}\Bigl(\exp\Bigl(-X_t(\phi)-\int_0^t X_s(f)ds\Bigr)\Bigr) =\exp(-X_0(V_t(\phi,f)).
\end{equation}
Here $V_t(x)=V_t(\phi,f)(x)$ is the unique solution of the mild form of 
\begin{equation}\label{SLLE}
\frac{\partial V}{\partial t}=\frac{\Delta V_t}{2}-\frac{V_t^2}{2}+f,\quad V_0=\phi,
\end{equation}
that is,
\begin{equation*}
V_t=P_t(\phi)+\int_0^tP_s\Bigl(f-\frac{V_{t-s}^2}{2}\Bigr)\,ds.
\end{equation*}
In the above $(P_t)$ is the semigroup of standard $d$-dimensional Brownian motion. See Chapter II 
of \cite{Per02} for the above and further properties.  
Note that $X$ has an a.s. finite extinction time, and therefore we can define the so-called total occupation time measure of the super-Brownian motion  as a finite measure,
 \begin{eqnarray*}
I(A)=\int_0^\infty X_s(A)ds.
\end{eqnarray*}
Supp$(\mu)$ will denote the closed support of a measure $\mu$. We 
define the range, $\mR$, of $X$ to be 
\begin{equation*}
\mR=\textnormal{Supp}(I).
\end{equation*}
A slightly smaller set is often used in the literature (see \cite{DIP89} or Corollary 9 in Ch. IV of \cite{Leg99}) but the definitions agree under $\P_{\delta_x}$ or the canonical measures $\N_x$ defined below, and also give the same
outcomes for $\mR\cap\text{Supp}(X_0)^c$ and $\partial\mR\cap\text{Supp}(X_0)^c$. Therefore the two definitions will
be equivalent for our purposes. 
In dimensions $d\leq 3$, the occupation measure $I$ has a density, $L^x$, which is called (total) local time of $X$, that is,
$$I(f)=\int_0^\infty X_s(f)\,ds=\int_{\R^d} f(x)L^x\,dx\text{ for all  non-negative measurable }f.$$
Moreover, $x\mapsto L^x$ is lower semicontinuous, is continuous on $\text{Supp}(X_0)^c$, and  for $d=1$ is globally continuous (see Theorems~2 and 3 of \cite{Sug89}). 
From \eqref{LF} and \eqref{SLLE} it is easy to derive (see Lemma 2.2 in~\cite{MP17}) 
\begin{equation}\label{LVlambda}
\E_{X_0}(e^{-\lambda L^x})=\exp\left(-\int_{\R^d} V^\lambda(x-x_0)X_0(dx_0)\right),
\end{equation}
where $V^\lambda$ is the unique solution (see Section~2 of \cite{MP17} and the references given there) to 
\begin{equation}\label{vlequation}
\frac{\Delta V^\lambda}{2}=\frac{(V^\lambda)^2}{2}-\lambda\delta_0,\ \ V^\lambda>0\text{ on }\R^d.\end{equation}
Thus in dimensions $d\leq 3$ we have 
\begin{equation*}
\mR=\overline{\{x:L^x>0\}},
\end{equation*}
and $\mR$ is a closed set of positive Lebesgue measure. In dimensions $d\ge 4$, $\mR$ is a Lebesgue null
set of Hausdorff dimension $4$ (see Theorem~1.4 of \cite{DIP89}), which explains our restriction to $d\le 3$ in this work.

Our main goal in this paper is to study properties of $\pmR$ --- the topological boundary of $\mR$, and 
in particular to determine the local Hausdorff dimension of $\pmR$ outside the support of $X_0$. 
The related question of  the dimension of the boundary of the set where the local time is positive, that is the dimension of 
\begin{equation}\label{frontdef}
F=\partial\{x:L^x>0\},
\end{equation}
was studied in \cite{MP17}.
To describe this latter result we introduce:
\begin{equation}\label{5_07_0}
p=p(d)=\begin{cases}
3 &\text{ if }d=1\\
2\sqrt 2 &\text{ if }d=2\\
\frac{1+\sqrt {17}}{2} &\text{ if }d=3,
\end{cases}
\end{equation}
$d_f=d+2-p$, 
and
\begin{equation}
\label{5_07_1}
\alpha=\alpha(d)=\frac{p(d)-2}{4-d}=\begin{cases} 1/3&\text{ if }d=1\\
 \sqrt 2-1&\text{ if }d=2\\
\frac{\sqrt{17}-3}{2}&\text{ if }d=3.
\end{cases}
\end{equation}

\begin{theorem}[\cite{MP17}]\label{dimthm} With $\P_{\delta_0}$-probability one, 
$$\textnormal{dim}(F)= d_f=\begin{cases}0&\text{ if } d=1\\
4-2\sqrt2\approx1.17&\text{ if }d=2\\
\frac{9-\sqrt{17}}{2}\approx 2.44&\text{ if }d=3.
\end{cases}$$
\end{theorem}
\noindent There were also versions of the above in~\cite{MP17} for more general initial conditions $X_0$.  

I. Benjamini's observation that the boundary of the range exhibited interesting fractal
properties in simulations was one motivation for the above.
Although $F$ may be a natural object from a stochastic analyst's perspective, the topological boundary of $\mR$, $\pmR$, is the more natural geometric object and of course was the set Benjamini had in mind.
Clearly, $\pmR$ and $F$ are closely related; it is easy to check that
 \begin{equation}\label{boundaryofR} \pmR\subset F.
 \end{equation}
Thus,  Theorem~\ref{dimthm} gives an upper bound on dimension of $\pmR$. Whether or not $F=\pmR$ remains open for $d=2$ or $3$, but Theorem~1.7 in \cite{MP17} shows that,  if $d=1$, there exist random variables  $\lb,\rb$ such that 
\begin{equation}\label{d1}
F=\pmR=\{\lb,\rb\}\text{ where }\lb<0<\rb\quad \P_{\delta_0}-\text{a.s.}, 
\end{equation}
and so we will usually assume $d=2$ or $3$.  A point $x$ in $F$ will be in $\pmR$ iff there are open sets $U$ approaching $x$ s.t. $L=0$ on $U$. 
Note that,  for example,  any isolated zeros of $L$ will be in $F$ but not in $ \pmR$ but we do not even know if such points exist in $d=2,3$. 
It was conjectured in (1.10) of \cite{MP17}  that in $d=2,3$,
\begin{equation}\label{dimssame}
\textnormal{dim}( \pmR)=\textnormal{dim}(F)\quad\P_{\delta_0}-a.s..
\end{equation}
In this paper we verify this conjecture, and prove the following stronger local version. 

\begin{theorem}\label{localdim}
$\P_{X_0}$-a.s. for any open $U\subset \textnormal{Supp}(X_0)^c$
\[U\cap\pmR\neq\emptyset\Rightarrow \textnormal{dim}(U\cap\pmR)=d_f.\]
\end{theorem}
The following corollary is immediate.
\begin{corollary}\label{globaldim}
$\textnormal{Supp}(X_0)^c\cap\pmR\neq\emptyset\Rightarrow \textnormal{dim}(\textnormal{Supp}(X_0)^c\cap\pmR)=d_f\quad\P_{X_0}-\text{a.s.}$
\end{corollary}
\noindent The hypothesis in the above Corollary is needed--see Proposition~1.5 of \cite{MP17} for an example where
it fails with positive probability. 
\begin{corollary}\label{delta0} $\P_{\delta_0}$-a.s. for any open set $U$, 
\[U\cap\pmR\neq\emptyset\Rightarrow \textnormal{dim}(U\cap\pmR)=d_f.\]
 In particular, 
$\textnormal{dim}(\partial\mR)=d_f\quad \P_{\delta_0}-\text{a.s.}$
\end{corollary}
\begin{proof} 
It is enough to check the claim only for  $U\supset \{0\}$, since otherwise it follows immediately from  Theorem~\ref{localdim}. Thus, let $U$  be an open set such that $U\cap\pmR\neq\emptyset$ and  $U\supset \{0\}$. 
We claim that $(U\setminus\{0\})\cap \pmR\neq \emptyset$.  If not, we have for small enough $r>0$, $\{|x|<r\}\cap\pmR=\{0\}$. $\mR$ is $\P_{\delta_0}-\text{a.s.}$ a connected set (e.g., see Theorem~7 in Ch. IV of \cite{Leg99}) of positive Lebesgue measure.  
There are $r_n\downarrow 0$ s.t. $\partial B_{r_n}\cap \mR^c\neq\emptyset$, and by connectedness of $\mR \ni 0$ (and $\mR\neq \{0\}$),  $\partial B_{r_n}\cap \mR\neq\emptyset$.   
The connectedness of $\partial B_{r_n}$ now implies that $\partial B_{r_n}\cap\partial \mR$ is non-empty for all $n$, which contradicts our assumption that $0$ is isolated in $\partial \mR$.
Now we may apply Theorem~\ref{localdim} with $U\setminus\{0\}$ in place of $U$ to finish.
\end{proof}
\noindent Note that besides confirming \eqref{dimssame}, the above  shows that the dimension result holds locally on any open ball intersecting $\pmR$. 

We also consider $X$ and its local time under the canonical measures $\N_x$. Recall from Section II.7 of \cite{Per02} that $\N_x$ is a $\sigma$-finite measure on the space of continuous finite length $M_F(\R^d)$-valued excursion paths
such that 
\begin{equation}\label{PPPdecomp}
X_t=\int \nu_t\, \Xi(d\nu)\quad \text{ for all }t>0\ \ \text{under }\P_{X_0},
\end{equation}
where $\Xi$ is a Poisson point process with intensity $\N_{X_0}(\cdot)=\int \N_{x_0}(\cdot)\,X_0(dx_0)$.  In this way $\N_{x_0}$ governs the ``excursions" of $X$ from a single ancestor at $x_0$.  The existence of local time $L$ under $\N_x$ follows easily from the above, in fact it is even globally continuous (see \cite{H17}).
It should not be surprising that Corollary~\ref{delta0} continues to hold under the canonical measure, in fact, as we shall see,  the proof is a bit easier. 
\begin{theorem}\label{localdimcan} $\N_0$-a.e. for any open set $U$,
\[U\cap \pmR\neq \emptyset\Rightarrow \textnormal{dim}(U\cap\pmR)=d_f.\]
\end{theorem}

We first say a few words 
about the argument leading to the proof of Theorem~\ref{dimthm} in~\cite{MP17}. If a small ball $B$ intersects $F$, then $B$ contains a point $x$ such that $L^x$ is positive but small. Thus, to get the bounds on the Hausdorff dimension of $F$, it  is useful to  understand the asymptotics of $\P_{\delta_0}(0<L^x<\veps)$, as $\veps\downarrow 0$. Write $f(t)\sim g(t)$ as $t\downarrow 0$ iff $f(t)/g(t)$ is bounded and bounded away from zero for small positive $t$, and similarly for $f(t)\sim g(t)$ as $t\uparrow\infty$.  It was shown in Theorem~1.3 of~\cite{MP17} that for $p$ as in \eqref{5_07_0} and 
 $\alpha$ given by~\eqref{5_07_1},
\begin{equation}
\label{5_07_2}
\P_{\delta_0}(0<L^x<\veps)\sim |x|^{-p}\veps^\alpha,\;\;{\rm as}\; \veps\downarrow 0.
\end{equation}
Not very difficult {\em heuristics} involving regularity properties of local time and a covering argument explains the  upper bound on dimension of $F$:   $\text{dim}(F)\leq d_f$ (see the Introduction of~\cite{MP17}). \eqref{5_07_2} was derived in \cite{MP17} through a Tauberian theorem we now sketch.
Let $\lambda\uparrow\infty$ in~\eqref{LVlambda} and \eqref{vlequation} to see that $V^\lambda(x)\uparrow V^\infty(x)$ where
\begin{equation}\label{VinfL}
\P_{\delta_0}(L^x=0)=\exp(-V^\infty(x)).
\end{equation}
One important simplification available for the analysis of $F$ in \cite{MP17} is that $V^\infty$ is explicitly known (see e.g. (2.17) in \cite{MP17}):
\begin{equation}\label{Vinf}
V^\infty(x)=\frac{2(4-d)}{|x|^2}.
\end{equation}
In particular $V^\infty$ solves
\begin{equation}\label{Vinfeq}
\frac{\Delta V^\infty}{2}=\frac{(V^\infty)^2}{2}\text{ for }x\neq 0.
\end{equation}
 $V^\infty$ sometimes is called the  {\it  very singular solution} to \eqref{Vinfeq}, see, e.g.,~\cite{brez86}.  
Applying a Tauberian theorem one can see that \eqref{5_07_2} can be reduced to verifying 
\begin{equation}\label{6_07_1}
\E_{\delta_0}(e^{-\lambda L^x}1(L^x>0))\sim  |x|^{-p}
\lambda^{-\alpha}, \; \text{ as }\lambda\uparrow\infty.
\end{equation}
The left-hand side of the above behaves  like $d^\lambda(x):=V^\infty(x)-V^\lambda(x)$, and so
a substantial part of the argument in \cite{MP17} was devoted to finding a rate of convergence of $V^\lambda$ to $V^\infty$, and showing that it behaves like the right hand side of \eqref{6_07_1}.

The upper bound on dim$(F)$ in~\cite{MP17} also utilized
Dynkin's exit measures. We always assume 
\begin{align} \label{Gdef}
&G \text{ is an open set satisfying } d(G^c,\text{supp}(X_0))>0\text{ and a Brownian path }\\
\nn&\text{starting from any $x\in\partial G$ will exit $G$ immediately}.
\end{align}
The exit measure of $X$ from such a $G$ under $\P_{X_0}$ or $\N_{X_0}$ is denoted by $X_G$ (see Chapter V of \cite{Leg99} for a good introduction to exit measures).   $X_{G}$ is a random finite measure supported on $\partial G$,  which intuitively corresponds to the mass started at $X_0$ which is stopped at the instant it leaves $G$. 
The Laplace functional of $X_G$ is given by 
\begin{equation}\label{LFEM}
\E_{X_0}(\exp(-X_G(g))=\exp\Bigl(-\int 1-\exp(-X_G(g))d\N_{X_0}\Bigr)=\exp\Bigl(-\int U^g(x)X_0(dx)\Bigr),
\end{equation}
where $g:\partial G\to[0,\infty)$ is continuous and $U^g\ge 0$ is the unique continuous function on $\overline G$ which is $C^2$ on $G$ and solves
\begin{equation}\label{EMpde}
\Delta U^g=(U^g)^2\text{ on }G,\quad U^g=g\text{ on }\partial G.
\end{equation}
For this, see Theorem 6 in Chapter V of \cite{Leg99}, and the last exercise on p. 86 for uniqueness.
Let $G_\varepsilon^{x_0}=\{x:|x-x_0|>\varepsilon\}$ and set $G_\veps=G_\veps(0)$. Similarly
$B_\veps(x_0)$ is the open ball centered at $x_0$ and $B_\veps=B_\veps(0)$. 
Proposition~3.4 of \cite{MP17} gives an upper bound on $\P_{\delta_0}(0<X_{G_\veps^x}(1)<\veps)$ as $\veps\downarrow 0$ for $x\neq 0$. This bound is refined to precise asymptotics in Propositions~\ref{t1.2} and \ref{t2.4} in Section~\ref{lowerbndexitmeas} below. Intuitively these asymptotics are related to \eqref{5_07_2} since a small exit measure from $G^x_\veps$ suggests small values of the local time inside 
$B_\veps(x)$. 

Consider next the ideas underlying Theorem~\ref{localdim}, where exit measures play a more central role. 
To show that a point $x$ is near $\pmR$, it is not enough 
to show  that the local time at $x$ is small and positive, or that the exit measure from some $G^x_\veps$ is small. In addition, there should be balls $B$ near $x$ on which the  local time is zero, or equivalently $X_{{\bar B}^c}=0$. To this end we will study $\P_{\delta_0}(0<X_{G^x_\veps}(1)\leq K\veps^2, X_{G^x_{\veps/2}}(1)=0)$ 
and 
show (see Theorem~\ref{t2.3} and Proposition~\ref{t1.2})
\begin{equation}
\label{5_07_3}
\P_{\delta_0}(0<X_{G^x_\veps}(1)\leq K\veps^2, X_{G^x_{\veps/2}}(1)=0)\sim \veps^{p-2},\;{\rm as}\;\veps\downarrow 0. 
\end{equation}
The proof of \eqref{5_07_3} requires asymptotics for solutions to  \eqref{EMpde} with varying boundary conditions, rather than than solutions to \eqref{vlequation}. 
For $\veps,\lambda>0$ we let $U^{\lambda,\veps}$ denote the unique continuous function on $\{|x|\ge \veps\}$ such that (cf. \eqref{EMpde})
\begin{equation}\label{pde1}
\Delta U^{\lambda,\veps}=(U^{\lambda,\veps})^2\ \ \text{for }|x|>\veps,\ \ \text{ and }\  U^{\lambda,\veps}(x)=\lambda\ \ \text{for }|x|=\veps.
\end{equation}
Uniqueness of solutions implies the scaling property
\begin{equation}\label{scaling1}
U^{\lambda,\veps}(x)=\veps^{-2} U^{\lambda \veps^2,1}(x/\veps)\quad\text{for all }|x|\ge\veps,
\end{equation}
and also shows $U^{\lambda, \veps}$ is radially symmetric, thus allowing us to write $U^{\lambda,\veps}(|x|)$ for the value at $x\in\R^d$. 
 By \eqref{LFEM} we have for any finite initial measure satisfying $\text{Supp}(X_0)\subset G_\veps$,
\begin{equation}\label{LTexit}
\E_{X_0}(\exp(-\lambda X_{G_\veps}(1))=\exp(-X_0(U^{\lambda, \veps})).
\end{equation}
Let $\lambda\uparrow\infty$ in the above to see that $U^{\lambda,\veps}\uparrow U^{\infty,\veps}$ on $G_\veps$ and 
\begin{equation}\label{UinfLT}
\P_{X_0}(X_{G_\veps}(1)=0)=\exp(-X_0(U^{\infty,\veps})).
\end{equation}
Proposition 9(iii) of \cite{Leg99} readily implies (see (3.5) and (3.6) of \cite{MP17})
\begin{equation}\label{Uinftyprop}
U^{\infty,\veps}\text{ is }C^2\text{ and }\Delta U^{\infty,\veps}=(U^{\infty,\veps})^2\text{ on }G_\veps,\ \lim_{|x|\to\veps,|x|>\veps}U^{\infty,\veps}(x)=+\infty,\ \lim_{|x|\to\infty}U^{\infty,\veps}(x)=0.
\end{equation}
Clearly a key step in deriving \eqref{5_07_3} are asymptotics  for 
\begin{eqnarray*}
\P_{\delta_0}(0<X_{G^x_\veps}(1)\leq K\veps^2) &\sim& {U^{\infty, \veps}(x)}-U^{K^{-1} \varepsilon^{-2}, \varepsilon}(x)\ \text{ as }\veps\downarrow 0,
\end{eqnarray*}
where the above equivalence is by a Tauberian theorem. 
In Section~\ref{sec:3.1} we show (see Corollary~\ref{c1.4})
\begin{eqnarray*}
{U^{\infty, \veps}(x)}-U^{K^{-1} \varepsilon^{-2}, \varepsilon}(x)
&\sim& \veps^{p-2}.
\end{eqnarray*}
This and a special Markov property (Propositions~\ref{spmarkov} and \ref{p0.1}) then give  \eqref{5_07_3}.  To get a lower bound on $\pmR$, essentially by an inclusion-exclusion argument,
in addition to the lower bound in \eqref{5_07_3},
we will also need an upper bound on (see Proposition~\ref{p3.1})
\begin{equation}\label{Biv}P_{\delta_0}(0<X_{G^{x_1}_\veps}(1)\leq K\veps^2,\ 0<X_{G^{x_2}_\veps}(1)\leq K\veps^2).
\end{equation}
\SUBMIT{Although involved, this argument  is quite similar to the proof of Proposition 6.1 in \cite{MP17} and so is omitted (it can be found in the Appendix of \cite{HMP18A}).}
\ARXIV{Although involved, this argument  is  similar to the proof of Proposition 6.1 in \cite{MP17} and so is deferred to the Appendix ~\ref{sec:5}.}
The above estimates allow us to show 
 that the lower bound on the dimension of $\pmR$ holds with positive probability--see Proposition~\ref{prop:crudelbdim}.  To conclude the proof of Theorem~\ref{localdim}, we will  show that  the lower bound on local dimension, in fact  holds  {\it with probability one.} 
This will be a consequence of the following proposition:

\begin{proposition}\label{exitdim} Let $x_1\in\R^d$ and $0<r_1<r_0$, and assume $B_{2r_0}(x_1)\subset\textnormal{Supp}(X_0)^c$. Then $\P_{X_0}$-a.s.
\begin{equation}\label{exitcond}
X_{G_{r_1}^{x_1}}(1)=0\text{ and }X_{G_{r_0}^{x_1}}(1)>0\text{ imply }\textnormal{dim}(B_{r_0}(x_1)\cap\pmR)\ge d_f.
\end{equation}
\end{proposition}
\noindent The main ingredient in the proof of Proposition~\ref{exitdim} is a version under the canonical measure.
\begin{proposition}\label{candim}
 Let $x_1\in\R^d$ and $0<r_1<r_0$, and assume $B_{2r_0}(x_1)\subset\textnormal{Supp}(X_0)^c$. Then
$\N_{X_0}$-a.e.  
\begin{equation}\label{exitcond_a}
\left\{
\begin{array}{l}
X_{G_{r_1}^{x_1}}(1)=0\text{ and }X_{G_{r_0}^{x_1}}(1)>0\text{ imply }\\
 \textnormal{dim}(B_{r}(x_1)\cap\pmR)\ge d_f \text{ for every $r>r_1$ s.t. $X_{G_{r}^{x_1}}(1)>0$}.
\end{array}
\right.
\end{equation}
\end{proposition}

\medskip


The paper is organized as follows. In Section~\ref{sec:prel}, preliminary results on super-Brownian motion, Brownian snakes, exit measures and their special Markov property are presented. In Section~\ref{sec:pfthm} we prove  Theorems~\ref{localdim} and~\ref{localdimcan}, assuming Propositions~\ref{exitdim},~\ref{candim}. 

In Section~\ref{lowerbndexitmeas} left tail asymptotics of exit measures are given. First in Section~\ref{sec:3.1} we derive necessary bounds on solutions to the boundary value problems~\eqref{pde1} and~\eqref{Uinftyprop}, and then in Section~\ref{sec:3.2} we prove \eqref{5_07_3} (see Theorem~\ref{t2.3} and Proposition~\ref{t1.2}).
In Section~\ref{nonpolar}, we  show that the lower bound on the local dimension of $\pmR$ holds with positive probability; see  Proposition~\ref{prop:crudelbdim} and  Lemma~\ref{lem:spherelb}.  

In Section~\ref{sec:6},  we show that for any $r_0>0$, 
the process $Z_t=X_{G_{r_0e^{-t}}}(1)/(r_0e^{-t})^2,\ \  t\ge 0$  is a {\it time homogeneous} continuous state branching process (CSBP), see Proposition~\ref{prop:yzinfo}.  The absence of negative jumps in $Z$ is important in the proof of Proposition~\ref{candim}, but we also 
believe that $Z$ and its associated measure-valued process are of independent interest; see Remark~\ref{rem:CSBP}. The proof of Propositions~\ref{exitdim} and \ref{candim} are concluded in Section~\ref{sec:maintheorem}. 
For the proof of  Proposition~\ref{exitdim} one shows that for $r<r_0$ sufficiently small there is a single excursion of $X$  (see \eqref{PPPdecomp}) governed by $\N_{X_0}$ that enters $B_r$ and thus  by the monotonicity of dimension, Proposition~\ref{exitdim} follows from  Proposition~\ref{candim}. Proposition~\ref{candim} (with $x_1=0$ without loss of generaility) is proved by studying the martingale
\[M_r=\N_{x_0}(\text{dim}(B_{r_0}\cap \pmR)\ge d_f|\cE_r)\quad0\le r<r_0,\ |x_0|>2r_0,\]
where $\cE_r$ is the $\sigma$-field generated by the excursions of the Brownian snake in $G_{r_0-r}$ (see Section~\ref{sec:prel} for a careful definition). In particular we analyze 
$$M_r\text{ as }r\uparrow T_0=\inf\{r:X_{G_{r_0-r}}(1)=0\}\text{ on }\{0<T_0\le r_0-r_1\},$$
 where $r_0,r_1$ are as in Proposition~\ref{candim}. The special Markov property and results from Sections~\ref{nonpolar} and \ref{sec:6} will show $M_r\ge q>0$, for $r$ close enough to $T_0$, and on the above set. The last step is then to show that $\{\text{dim}(B_{r_0}\cap \pmR)\ge d_f\}\in\cE_{T_0-}$ (Lemma~\ref{Emeas}). Then letting $r\uparrow T_0$ shows that on $\{0<T_0\le r_0-r_1\}$, $1(\text{dim}(B_{r_0}\cap \pmR)\ge d_f)\ge q>0$, as required.
 
 Note that the methods of \cite{MP17} (see Theorem~1.4 and the ensuing discussion of that work) would
have required the stronger hypothesis $\text{Conv}(X_0)^c\cap\pmR\neq\emptyset$ in Corollary~\ref{globaldim}, where Conv$(X_0)$ is the closed convex hull of Supp$(X_0)$. This is because  exit measures from hyperplanes were used in \cite{MP17}, instead of  the process of exit measure from the complements of shrinking balls. This refinement also leads to the 
purely local result on dimension in Theorem~\ref{localdim}.

\paragraph{\bf Convention on Functions and Constants.}
Constants whose value is unimportant and may change from line to line are denoted $C, c, c_d, c_1,c_2,\dots$, while constants whose values will be referred to later and appear initially in say, Lemma~i.j are denoted $c_{i.j},$ or $ \underline c_{i.j}$ or $C_{i.j}$. 
\medskip

\section{Exit Measures and the Special Markov Property}\label{sec:prel}
{\bf Notation.} Let $\mK$ be the space of compact subsets of $\R^d$ equipped with the Hausdorff metric, where we add $\emptyset$ as a discrete point.  That is let $K^\veps=\{x:d(x,K)\le \veps\}$ and for $K_1,K_2$ non-empty compacts, set 
\[d(K_1,K_2)=\inf\{\veps>0: K_1\subset K_2^\veps\text{ and }K_2\subset K_1^\veps\}\wedge 1,\]
and set $d(\emptyset, K)=1$ for $K$ non-empty compact. $(\mK, d)$ is then a complete separable metric space.
If $U$ is an open set in $\R^d$ we let $C(U)$ be the space of continuous functions on $U$ with the compact-open topology.

We start with a measurability result requiring a bit of care; a proof is given in the Appendix.
\begin{lemma}\label{borel}
(a) For any $R>0$, $\psi_a:\mK\to\mK$ is a Borel map, where $\psi_a(K)=K\cap\overline{B_R}$.\\
(b) For any $\alpha,R>0$, $\psi_b:\mK\to\R$ is a universally measurable map, where\\ $\psi_b(K)=1(\text{dim}((\partial K)\cap B_R)< \alpha)$.
\end{lemma}

We will use Le Gall's Brownian snake construction of a SBM $X$, with initial state $X_0\in M_F(\R^d)$. Set $\cW=\cup_{t\ge 0} C([0,t],\R^d)$ with the natural metric (see page 54 of \cite{Leg99}), and let $\zeta(w)=t$ be the lifetime of $w\in C([0,t],\R^d)\subset\cW$. The Brownian snake $W=(W_t,t\ge0)$ is a $\cW$-valued continuous strong Markov process and, abusing notation slightly, let $\N_x$ denote its excursion measure starting from the path at $x\in\R^d$ with lifetime zero.  As usual we let $\hat W(t)=W_t(\zeta(W_t))$ denote the tip of the snake at time $t$, and $\sigma(W)>0$ denote the length of the excursion path. We refer the reader to Ch. IV of \cite{Leg99} for the precise definitions.  The construction of super-Brownian motion, $X=X(W)$ under $\N_x$ or $\P_{X_0}$, may be found in Ch. IV of \cite{Leg99}. The ``law" of $X(W)$ under $\N_x$ is the canonical measure of SBM starting at $x$ described in the last Section (and also denoted by $\N_x$). If $\Xi=\sum_{j\in J}\delta_{W_j}$ is a Poisson point process on $\cW$ with intensity $\N_{X_0}(dW)=\int\N_x(dW)X_0(dx)$, then by Theorem~4 of Ch. IV of \cite{Leg99} (cf. \eqref{PPPdecomp})
\begin{equation}\label{Xtdec}
X_t(W)=\sum_{j\in J}X_t(W_j)=\int X_t(W)\Xi(dW)\text{ for }t>0\end{equation}
defines a SBM with initial measure $X_0$. We will refer to this as the standard set-up for $X$ under $\P_{X_0}$. 

Recall $\mathcal{R}=\overline{\{x:L^x>0\}}$ is the range of the SBM $X$ under $\P_{X_0}$ or $\N_{X_0}$.    Under $\N_{X_0}$ we have (see (8) on p. 69 of \cite{Leg99})

\begin{equation}\label{snakerange}
\mR=\{\hat W(s):s\in[0,\sigma]\}.
\end{equation}

Let $G$ be as in \eqref{Gdef}.   Then 
\begin{equation}\label{exitsupport}
X_G \text{ is a finite random measure supported on }\mathcal{R}\cap \partial G\text{ a.s.}.
\end{equation}
Under $\N_{X_0}$ this follows from the definition of $X_G$ on p. 77 of \cite{Leg99} and 
the ensuing discussion, and \eqref{snakerange}.
  Although \cite{Leg99} works under $\N_x$ for $x\in G$ the above extends immediately to $\P_{X_0}$ because as in (2.23) of \cite{MP17}, 
  \begin{equation}\label{exitdecomp}
X_G=\sum_{j\in J} X_G(W_j)=\int X_G(W)d\Xi(W),
\end{equation}
where $\Xi$ is a Poisson point process on $\cW$ with intensity $\N_{X_0}$.  

Working under $\N_{X_0}$ and following \cite{Leg95}, we define
\begin{align*}
S_G(W_u)&=\inf\{t\le \zeta_u: W_u(t)\notin G\}\ \ (\inf\emptyset=\infty),\\
\eta_s^G(W)&=\inf\{t:\int_0^t1(\zeta_u\le S_G(W_u))\,du>s\},\\
\cE_G&=\sigma(W_{\eta_s^G},s\ge 0\}\vee\{\N_{X_0}-\text{null sets}\},
\end{align*}
where $s\to W_{\eta^G_s}$ is continuous (see p. 401 of \cite{Leg95}).  Write the open set $\{u:S_G(W_u)<\zeta_u\}$ as countable union of disjoint open intervals, $\cup_{i\in I}(a_i,b_i)$.
Clearly $S_G(W_u)=S^i_G<\infty$ for all $u\in [a_i,b_i]$ and we may define
\[W^i_s(t)=W_{(a_i+s)\wedge b_i}(S^i_G+t)\text{ for }0\le t\le \zeta_{(a_i+s)\wedge b_i}-S^i_G.\]
Therefore for $i\in I$, $W^i\in C(\R_+,\cW)$ are the excursions of $W$ outside $G$. Proposition 2.3 of \cite{Leg95} implies $X_G$ is $\cE_G$-measurable and Corollary~2.8 of the same reference implies
\begin{align}\label{SMP1}
\nn\text{Conditional on $\cE_G$, the point measure $\sum_{i\in I}\delta_{W^i}$ is a Poisson point measure}&\\
\text{with intensity $\N_{X_G}$}&. 
\end{align}
If $d(D^c,\bar G)>0$, then the definition (and existence) of $X_D(W)$ applies equally well to each $X_D(W^i)$
and it is easy to check that 
\begin{equation}\label{Widecomp}
X_D(W)=\sum_{i\in I}X_D(W^i).
\end{equation}

If $U$ is an open subset of $\text{Supp}(X_0)^c$, then $L_U$, the
restriction of the local time $L^x$ to $U$, is in $C(U)$.
Here are some simple consequences of \eqref{SMP1}. 
\begin{proposition}\label{spmarkov}(a) Let $G_1\subset G_2$ be open sets as in \eqref{Gdef} such that $d(G_2^c,\overline{G_1})>0$.\\

(i) If $\psi_1:C(\overline{G_1}^c)\to [0,\infty)$ is Borel measurable, then
\begin{equation*} 
\N_{X_0}(\psi_1(L_{\overline{G_1}^c})|\cE_{G_1})=\E_{X_{G_1}}(\psi_1(L_{\overline{G_1}^c})).
\end{equation*}

(ii) If $\psi_2:M_F(\R^d)\to [0,\infty)$ is Borel measurable then
\begin{equation*} 
\N_{X_0}(\psi_2(X_{G_2})|\cE_{G_1})=\E_{X_{G_1}}(\psi_2(X_{G_2})).
\end{equation*}
(b) If $0<R_2<R_1$, $d(\text{Supp}(X_0), \overline{B_{R_1}})>0$, and $\psi_3:\mK\to[0,\infty)$ is Borel measurable, then
\begin{equation*}\N_{X_0}(\psi_3(\mR\cap\overline{B_{R_2}})|\cE_{G_{R_1}})=\E_{X_{G_{R_1}}}(\psi_3(\mR\cap\overline{B_{R_2}})).
\end{equation*} 
\end{proposition}
\begin{proof} (a) (i) is Proposition 2.6(b) of \cite{MP17}. (a)(ii) follows in a similar manner from \eqref{SMP1}, \eqref{Widecomp} and \eqref{exitdecomp}.\\
(b) Define $S:C(B_{R_1})\to\mK$ by $S(f)=\text{Supp}(f):=\overline{\{x:f(x)>0\}}$, where the closure is taken in all of $\R^d$.  Then it is easy to see that $S$ is Borel measurable, for example by considering the inverse images of closed balls in $\mK$.  In addition the map $K\to \overline{B_{R_2}}\cap K$ is measurable on $\mK$ by Lemma~\ref{borel}(a). 
Now observe that $\mR\cap\overline{B_{R_2}}=S(L_{B_{R_1}})\cap\overline{B_{R_2}}$, and so by the above observations is a measurable function of $L_{B_{R_1}}$.  Therefore (b) now follows from (a)(i) with $G_i=G_{R_i}$.
\end{proof}

We will need a version of the above under $\P_{X_0}$ as well. 
\begin{proposition}\label{p0.1}
For an open set $G$ and $X_0 \in M_F(\R^d)$ as in \eqref{Gdef}, let $\Psi$ be a  bounded measurable function on $C({\overline{G}}^c)$ and $\Phi_i$, $i=0,1$ be bounded measurable
functions on $M_F(\R^d)$ and $M_F(\R^d)^n$, respectively. Then\\
(a) $\E_{X_0} (\Phi_0(X_G) \Psi(L))=\E_{X_0}(\Phi_0(X_G)\E_{X_G}(\Psi(L)))$.\\
(b) (i) Let $D_i$ be open sets as in \eqref{Gdef} such that $d(D_i^c,\bar G)>0$, $\forall 1\leq i\leq n$. Then \[\E_{X_0} \Big(\Phi_0(X_G)  \Phi_1(X_{D_1},\dots, X_{D_n})\Big)=\E_{X_0}\Big(\Phi_0(X_G) \E_{X_G}\Big( \Phi_1(X_{D_1},\dots,X_{D_n})\Big)\Big).\]
\quad(ii)  If $0<R_2<R_1$ and $d(\text{Supp}(X_0)^c, \overline{B_{R_1}})>0$, then 
$$\E_{X_0}(\Phi_0(X_{G_{R_1}})1(\mR\cap\overline{B_{R_2}}\neq\emptyset)=\E_{X_0}(\Phi_0(X_{G_{R_1}})\P_{X_{G_{R_1}}}(\mR\cap\overline{B_{R_2}}\neq\emptyset)).$$
\end{proposition}
\begin{proof}
 (a) is Proposition 2.6(c) of \cite{MP17}.  (b)(i) follows by the same reasoning there, using \eqref{SMP1}, \eqref{exitdecomp} (the latter for each $D_i$, as well as $G$), and Proposition~\ref{spmarkov}(a)(ii), trivially extended to accommodate $(X_{D_1},\dots,X_{D_n})$ in place of $X_{G_2}$.  (b)(ii) follows from (a), as in the proof of Proposition~\ref{spmarkov}(b).
\end{proof}

\section{Proof of Theorems~\ref{localdim} and~\ref{localdimcan}}
\label{sec:pfthm}

We will see in this section that (using Theorem~\ref{dimthm}) Theorem~\ref{localdim} is a  simple consequence of Proposition~\ref{exitdim},
and similarly  Theorem~\ref{localdimcan} can be derived from Proposition~\ref{candim}.\\
\noindent{\bf Proof of Theorem~\ref{localdim}.} Let $x_1\in\R^d$ and $0<r_1<r_0\le 1$ satisfy $B_{2r_0}(x_1)\subset \text{Supp}(X_0)^c$. From \eqref{exitsupport} we have $\P_{X_0}-\text{a.s.}$,
\begin{equation}\label{fact1}
\partial G_{r_1}^{x_1}\cap \mR=\emptyset\Rightarrow X_{G_{r_1}^{x_1}}=0.
\end{equation}
Proposition~\ref{p0.1}(b)(ii) and translation invariance imply
\[\P_{X_0}(X_{G_{r_0}^{x_1}}=0,\ \mR\cap B_{r_0/2}(x_1)\neq\emptyset)=\E_{X_0}(1(X_{G_{r_0}^{x_1}}=0)\P_{X_{G_{r_0}^{x_1}}}(\mR\cap B_{r_0/2}(x_1)\neq\emptyset))=0.\]
It follows that $\P_{X_0}-\text{a.s.}$,
\begin{equation}\label{fact2}
\mR\cap B_{r_0/2}(x_1)\neq\emptyset\Rightarrow X_{G_{r_0}^{x_1}}(1)>0.
\end{equation}

Fix $\omega$ outside a $\P_{X_0}$-null set so that \eqref{exitcond} of Proposition~\ref{exitdim}, \eqref{fact1}, and \eqref{fact2} all hold for all $x_1\in\Q^d$ and  all rational numbers $0<r_1<r_0\le 1$ satisfying $B_{2r_0}(x_1)\subset\text{Supp}(X_0)^c$.  Assume $U$ is an open set in $\text{Supp}(X_0)^c$ which intersects $\pmR$ and choose $x_0\in U\cap\pmR$. Pick a rational $r_0$ in $(0,1]$ so that 
\[B_{3r_0}(x_0)\subset U\subset \text{Supp}(X_0)^c,\]
then choose  $x_1\in\Q^d\cap B_{r_0/2}(x_0)\cap\mR^c$, and finally select a rational $r_1\in(0,r_0)$ such that
\begin{equation}\label{emptyr1}
B_{2r_1}(x_1)\subset \mR^c\text{ and so }\partial G_{r_1}^{x_1}\cap\mR=\emptyset.
\end{equation}
Clearly we have 
\begin{equation}\label{outS}
B_{2r_0}(x_1)\subset B_{3r_0}(x_0)\subset U\subset \text{Supp}(X_0)^c,
\end{equation}
and
\begin{equation}\label{r0exitp}
x_0\in B_{r_0/2}(x_1)\cap\pmR\text{ and so }B_{r_0/2}(x_1)\cap\mR\neq\emptyset.
\end{equation}
Our choice of $\omega$ and \eqref{outS} allow us to conclude from \eqref{emptyr1} and \eqref{r0exitp}, respectively, that 
\[ X_{G_{r_1}^{x_1}}(1)=0\text{ and }X_{G_{r_0}^{x_1}}(1)>0,\text{respectively.}\]
By \eqref{outS} and our choice of $\omega$ we may also apply Proposition~\ref{exitdim} and conclude that
\[\text{dim}(U\cap\pmR)\ge \text{dim}(B_{r_0}(x_1)\cap\pmR)\ge d_f,
\]
where we have used \eqref{outS} in the first inequality. On the other hand we know from Theorem~1.4(a) of \cite{MP17} and $\pmR\subset\partial\{x:L^x>0\}$ that
\[\text{dim}(U\cap\pmR)\le \text{dim}(S(X_0)^c\cap\partial\mR)\le d_f,\]
and the proof is complete.\qed\\

\noindent{\bf Proof of Theorem~\ref{localdimcan}.} The derivation of Theorem~\ref{localdimcan} directly from Proposition~\ref{candim} is very similar to the above proof of Theorem~\ref{localdim} from Proposition~\ref{exitdim} and
so is omitted. Note here that the derivation of Proposition~\ref{exitdim} under $\N_0$  from Proposition~\ref{candim}  with $X_0=\delta_0$ is   almost immediate as both are under $\N_0$, and that the proof of Corollary~\ref{delta0} then holds under $\N_0$.  
\qed

\section{Lower Bound on the Exit Measure Probability}\label{lowerbndexitmeas}
\noindent $\mathbf{Throughout\ this\ Section\ we\ fix }$ $\varepsilon_0 \in (0,1)$.
As noted in the Introduction, the goal of this section, stated below, is a key estimate for the lower bound 
on the dimension of $\pmR$. Although we are interested in $d=2,3$, we assume $d\le 3$ throughout this 
section as the arguments remain valid. 
\begin{theorem}\label{t2.3}
There are positive constants $R_{\ref{t2.3}}$, $K_1(\varepsilon_0)<K_2(\varepsilon_0)<\infty$ and $c_{\ref{t2.3}}(\varepsilon_0)$  such that, for all $\varepsilon_0 \leq |x| \leq \varepsilon_0^{-1}$,
\[\P_{\delta_x}(K_1\leq \frac{X_{G_{\varepsilon}}(1)}{\varepsilon^2} \leq K_2, X_{G_{\varepsilon/2}(1)}=0) \geq c_{\ref{t2.3}}(\varepsilon_0) \varepsilon^{p-2}, \ \forall\, 0<\varepsilon <\varepsilon_0/R_{\ref{t2.3}}.\]
\end{theorem}
The next subsection is devoted to proving necessary bounds on solutions to the boundary value problems~\eqref{pde1}, \eqref{Uinftyprop}. These bounds will be used for proving Theorem~\ref{t2.3} in Section~\ref{sec:3.2}.

\subsection{Bounds on Solutions to some Boundary Value Problems}
\label{sec:3.1}
Recall $U^{\lambda,R}$ and $U^{\infty,R}$  from~\eqref{pde1} and \eqref{Uinftyprop}, respectively. 
A simple application of \eqref{Uinftyprop}, \eqref{Vinfeq} and the maximum principle implies
\begin{equation}\label{Uinflb}
V^\infty(x)\le U^{\infty,1}(x)\quad\forall |x|>1.
\end{equation}
We will need an upper bound on $U^{\infty,1}$ which shows this bound is asymptotically sharp for large $|x|$.
We briefly include $d=1$ in our analysis. 
\begin{proposition}\label{p2.1}  There exist constants $C_{\ref{p2.1}}>1$
and  $c_{\ref{p2.1}}\ge 0$ such that 
$$U^{\infty,1}(x)\le V^\infty(x)(1+    c_{\ref{p2.1}}|x|^{2-p})\quad\forall |x|\ge C_{\ref{p2.1}}.$$
\end{proposition}
\begin{proof}
We will write $u(r)$ for $U^{\infty,1}(r)$ and $v(r)$ for $V^\infty(r)$. 

For $t\ge 0$, let 
\begin{eqnarray*}
q(t)&=&\frac{u(e^{t/4})}{v(e^{t/4})}=\frac{1}{4}u(e^{t/4})e^{t/2}, \;{\rm in}\; d=2, \\
q(t)&=&\frac{u(12^{1/3}e^{t/3})}{v(12^{1/3}e^{t/3})}=\frac{12^{2/3}}{2}u(12^{1/3}e^{t/3})e^{2t/3}, \;{\rm in}\; d=3, \\
q(t)&=&\frac{u(180^{1/5}e^{t/5})}{v(180^{1/5}e^{t/5})}=
\frac{180^{2/5}}{2}u(180^{1/5}e^{t/3})e^{2t/5}, \;{\rm in}\; d=1.
\end{eqnarray*}
A simple calculation 
gives
\[ \frac{1}{2}q''-\frac{1}{2}q'+\beta(q-q^2)=0 \text{ on }(0,\infty),\ q(0)ֿ\in(0,\infty],\ \lim_{t\to\infty}q(t)=1,\]
where  $\beta=\frac{3}{25}$ in $d=1$, 
$\beta=\frac{1}{8}$ in $d=2$, and $\beta=\frac{1}{9}$ in $d=3$. Note that 
 the last limit  is derived the same  way as (3.10) in~\cite{MP17} by taking $U^{\infty,1}$ instead of  $U^{\delta_0,1}$, $\tilde y(t)=y(t+2)$ instead of $y(t)$ and $\tilde z(t)=z(t+2)$ instead of $z(t)$ there. 

Note that $q(x)\ge 0$  for all $x\geq 0$ by \eqref{Uinflb}. Define 
\[ z=q-1,\]
then $z$ satisfies the following equation: 
\begin{equation}
\frac{1}{2}z''-\frac{1}{2}z'-\beta z(z+1)=0\ \ \text{on }(0,\infty),\ z(0)=\infty,\ \lim_{t\to\infty}z(t)=0.
\end{equation} 
Let $w$ be the unique solution to 
\begin{equation}
 \frac{1}{2}w''(t)-\frac{1}{2}w'(t)-\beta w(t) =0,\ t>1,  w(1)=z(1),\ \lim_{t\to\infty}w(t)=0.
\end{equation} 
By the  comparison principle we get 
\begin{equation}
\label{6_1}
w(t)\geq z(t),\; t\geq 1. 
\end{equation}
We leave it for the reader to check that 
\begin{eqnarray*}
w(t)=  z(1)e^{-\lambda } e^{\lambda t},\;t\geq 1,
\end{eqnarray*}
with 
\begin{equation*}
\lambda=\frac{1}{2}- \sqrt{\frac{1}{4}+2\beta}<0.
\end{equation*}
By the definition of $\beta$ we have $\lambda=-0.2$ for $d=1$, $\lambda=1/2-\frac{1}{\sqrt 2}$ for $d=2$, and $\lambda=1/2-\frac{\sqrt{17}}{6}$ for $d=3$.
This and~\eqref{6_1} imply   that for  $C= z(1)e^{-\lambda }\ge 0$ we have 
\begin{equation*}
z(x)\leq C e^{\lambda x}, \;x\geq 1,
\end{equation*}
and since $\lambda <0$ we get that $z$ decreases to zero exponentially fast. Recall the definition of $q$ 
to get 
\begin{equation*}
q(x)\leq 1+C e^{\lambda x},\;x\geq 1.
\end{equation*}

Then in $d=2$ we have,
\[ 
u(e^{t/4})\leq v(e^{t/4}) (1+ C e^{\lambda t}),\;t\geq 1.
\]
and so
\[ 
u(s)\leq v(s) (1+ C s^{4\lambda})=  v(s) (1+ C s^{2-p}) ,\;s\geq e^{1/4}.
\]
Similar algebra shows the result in $d=1,3$. 
\end{proof}

Recalling \eqref{Vinf}, we may immediately conclude:
\begin{corollary}\label{c2.2}
 There are constants $C_{\ref{c2.2}}, c_{\ref{c2.2}}>0$ such that for all $x\in \R^d$, with $|x|\geq C_{\ref{c2.2}}$, we have
\begin{align}\label{e1.3}
U^{\infty,1}(x)\leq V^{\infty}(x)+\frac{c_{\ref{c2.2}}}{|x|^{p}}.
\end{align}
In particular, there is some constant $K_{\ref{c2.2}}>2$ such that 
\begin{align}\label{e1.4}
U^{\infty,1}(x)\leq 3(4-d) |x|^{-2},\ \forall |x|\geq K_{\ref{c2.2}}.
\end{align}
\end{corollary}

If  $D^\lambda=U^{\infty,1}-U^{\lambda,1}\geq 0$ for $\lambda>0$, then the Feynmann-Kac formula
 (as in (3.8) in \cite{MP17}) easily gives 
\begin{align}\label{e1.2}
D^\lambda(x)=D^\lambda(R)E_x\Big(1_{(\tau_R<\infty)} \exp\Bigl(-\int_{0}^{\tau_R} \Bigl(\frac{U^{\infty,1}+U^{\lambda,1}}{2}\Bigr)(B_s)ds\Bigr)\Big), \ |x|\geq R>1,
\end{align}
where $B$ denote a $d$-dimensional Brownian motion starting at $x$ under $P_x$ and $\tau_R=\inf \{t\geq 0: |B_t|\leq R\}$ for $|x|\geq R>1$. 



We will frequently use the following lemmas.
For $\gamma\in\R$, let $(\rho_t)$ denote a $\gamma$-dimensional Bessel process starting from $r>0$ under $P_{r}^{(\gamma)}$, and $(\cF_t)$ be the filtration generated by $\rho$. Define $\tau_R=\inf\{t\geq 0: \rho_t\leq R\}$ for $R>0$. The following result is from Lemma 5.3 of \cite{MP17}.
\begin{lemma} \label{expbound2} 
Assume $0<2\gamma \le \nu^2$ and $q>2$. Then
\[\sup_{r\ge 1}E_r^{(2+2\nu)}\Bigl(\exp\Bigl(\int_0^{\tau_1}\frac{\gamma}{\rho_s^q}\,ds\Bigr)\Bigr|\tau_1<\infty\Bigr)\le c_{\ref{expbound2}}(q,\nu)<\infty.\]
\end{lemma}
\begin{lemma}
\label{lem:22_7_1}
Let $q>2$, $a\in \R$, 
$\zeta\in [0,2(4-d))$, $\nu_{\zeta}=\sqrt{\nu^2-\zeta}$ and $p_{\zeta}=\nu_\zeta+\mu$, where 
\begin{align}\label{e1.5}
\mu=
\begin{cases}
-1/2\ &\text{if}\ d=1,\\
0\ &\text{if}\ d=2,\\
1/2\ &\text{if}\ d=3,
\end{cases}
\text{ and } \nu=\sqrt{\mu^2+4(4-d)},
\end{align}
 Then for all $R<|x|$, 
\begin{align}
\label{eq:23_7_1}
E_x\Big(1(\tau_R<\infty)&\exp\Bigl(\int_{0}^{\tau_R} \frac{a}{|B_s|^q}ds\Bigr)  \exp(-\int_{0}^{\tau_R} \frac{2(4-d)-\zeta/2}{|B_s|^2}ds) \Big) \\
\nonumber
&= E_{|x|}^{(2+2\nu_\zeta)}\Big(\exp\Bigl(\int_{0}^{\tau_R} \frac{a}{\rho_s^q}ds\Bigr)\Big|\tau_R<\infty \Big) (R/|x|)^{p_\zeta}.
\end{align}
\end{lemma}
\begin{proof} The proof is based on arguments from~\cite{MP17} (see the proof of Lemma~5.4 there) and is deferred to 
Appendix~\ref{sec22_7_1}.
\end{proof}

\begin{proposition}\label{p1.1}
There are positive universal constants $C_{\ref{p1.1}}, c_{\ref{p1.1}}>0$, $K_{\ref{p1.1}}>K_{\ref{c2.2}}$, and  
$R_{\ref{p1.1}}>2$ such that
\begin{enumerate}[(a)]
\item \[D^\lambda(x) \leq \frac{R^p}{|x|^p}D^\lambda(R),\ \forall |x|\geq R>1,\ \lambda\geq 6.\]
\item \[D^\lambda(x) \leq C_{\ref{p1.1}} \frac{R^p}{|x|^p}D^\lambda(R),\ \forall |x|\geq R\geq \frac{K_{\ref{p1.1}}}{\lambda}, 0<\lambda<1.\]
\item \[D^\lambda(x) \geq c_{\ref{p1.1}} \frac{R^p}{|x|^p}D^\lambda(R)>0,\ \forall |x|\geq R\geq R_{\ref{p1.1}},\ \lambda>0.\]
\end{enumerate}
\end{proposition}

\begin{proof}

Recall $\mu$, $\nu$ introduced in~\eqref{e1.5}
so that for $p=p(d)$ defined as in \eqref{5_07_0}, we have
\begin{equation}\label{pmunu}
p=\mu+\nu.
\end{equation}
(a) For $\lambda\geq 6$, clearly we have $U^{\lambda,1}(1)=\lambda \geq V^{\infty}(1)$. As $$\lim_{|x|\to\infty} U^{\lambda,1}(x)\leq \lim_{|x|\to\infty} U^{\infty,1}(x) =0$$ by \eqref{Uinftyprop}, we may apply the maximum principle to get 
\begin{align}\label{e1.0}
 U^{\lambda,1}(x) \geq  V^{\infty}(x)=\frac{2(4-d)}{|x|^2},\quad \forall |x|>1.
\end{align}
Use \eqref{Uinflb} and the above to see that \eqref{e1.2} becomes
\begin{align*}
D^\lambda(x)&\leq D^\lambda(R)E_x\Big(1_{(\tau_R<\infty)} \exp\Bigl(-\int_{0}^{\tau_R} \frac{2(4-d)}{|B_s|^2}ds\Bigr)\Big)=D^\lambda(R) (R/|x|)^{p},
\end{align*}
the last by Lemma~\ref{lem:22_7_1}.

\noindent(b) Assume $\lambda\in(0,1)$. Recall Proposition 3.3(b) in \cite{MP17}: 
\begin{align}\label{e1.6}
\forall \delta\in (0,1),\ \exists C_\delta>2,  \text{ so that } U^{\lambda,1}(x)\geq (1-\delta) V^{\infty}(x) \text{ for all } |x|\geq C_\delta/\lambda. 
\end{align}
For any $\delta \in (0,1)$, let $\zeta=2(4-d)\delta \in (0,2(4-d))$. Let $\mu$ and $\nu$ be as in \eqref{e1.5}.  Define $\nu_{\zeta}=\sqrt{\mu^2+4(4-d)-\zeta}$ and $p_{\zeta}=\nu_\zeta+\mu \to p>2$ as $\zeta \downarrow 0$. Choose  $\delta\in(0,1)$ small enough so that $p_{\zeta}>2$. Let $ K_{\ref{p1.1}}\equiv C_\delta+K_{\ref{c2.2}}$. Now use \eqref{Uinflb},
 \eqref{e1.6} and  Lemma~\ref{lem:22_7_1} to see that for $|x|\geq R\geq K_{\ref{p1.1}}/\lambda>C_\delta/\lambda$, \eqref{e1.2} implies
\begin{align*}
D^\lambda(x)&\leq D^\lambda(R)E_x\Big(1(\tau_R<\infty) \exp\Bigl(-\int_{0}^{\tau_R} \frac{2(4-d)-(\zeta/2)}{|B_s|^2}ds\Bigr) \Big)
=D^\lambda(R) (R/|x|)^{p_\zeta},
\end{align*}
Let $\xi(R)=D^\lambda(R) R^{p_\zeta}/2.$ Then the above gives
\[\left(\frac{U^{\lambda,1}+U^{\infty,1}}{2}\right)(x) \geq U^{\infty,1}(x)-\frac{\xi(R)}{|x|^{p_\zeta}}\geq V^{\infty}(x)-\frac{\xi(R)}{|x|^{p_\zeta}} \text{ for } |x|\geq R.\] 
Use this in \eqref{e1.2} and then  Lemma~\ref{lem:22_7_1} to see that for $|x|\geq R$,
\begin{align*}
D^\lambda(x)&\leq D^\lambda(R)E_x\Big(1(\tau_R<\infty)\exp\Bigl(\int_{0}^{\tau_R} \frac{\xi(R)}{|B_s|^{p_\zeta}}ds\Bigr)  \exp(-\int_{0}^{\tau_R} \frac{2(4-d)}{|B_s|^2}ds) \Big)\\
&= D^\lambda(R) E_{|x|}^{(2+2\nu)}\Big(\exp\Bigl(\int_{0}^{\tau_R} \frac{\xi(R)}{\rho_s^{p_\zeta}}ds\Bigr)\Big|\tau_R<\infty \Big) (R/|x|)^{p}.
\end{align*}
 A scaling argument shows that the above equals
\[D^\lambda(R)(R/|x|)^{p} E_{|x|/R}^{(2+2\nu)}\Big(\exp\Bigl(\int_{0}^{\tau_1} \frac{\xi(R)R^{2-p_\zeta}}{\rho_s^{p_\zeta}}ds\Bigr)|\tau_1<\infty \Big) . \]
To apply Lemma \ref{expbound2} we note that by \eqref{e1.4}, for $R\geq K_{\ref{p1.1}}/\lambda>K_{\ref{c2.2}}$ we have
\[2\gamma\equiv 2\xi(R)R^{2-p_\zeta} \leq U^{\infty,1}(R) R^2\leq 3(4-d)<\nu^2.\] So Lemma \ref{expbound2} and the above bound  show that
\begin{align}\label{e1.9}
D^\lambda(x)&\leq D^\lambda(R) (R/|x|)^{p} c_{\ref{expbound2}}(p_\zeta,\nu).
\end{align}

\noindent(c) Use \eqref{e1.3} in Corollary \ref{c2.2} to see that for $|x|\geq R> C_{\ref{c2.2}}$, we have
\[\frac{U^{\infty,1}+U^{\lambda,1}}{2}(x)\leq U^{\infty,1}(x)\leq  \frac{2(4-d)}{|x|^2}+\frac{c_{\ref{c2.2}}}{|x|^{p}}.\] 
So \eqref{e1.2} and Lemma~\ref{lem:22_7_1} imply
\begin{align*}
D^\lambda(x)&\geq D^\lambda(R) E_x\Big(1_{(\tau_R<\infty)} \exp\Bigl(-\int_{0}^{\tau_R} \frac{c_{\ref{c2.2}}}{|B_s|^p}ds\Bigr)\exp\Bigl(-\int_{0}^{\tau_R} \frac{2(4-d)}{|B_s|^2}ds\Bigr) \Big) \\
&=D^\lambda(R) E_{|x|}^{(2+2\nu)}\Big(\exp\Bigl(-\int_{0}^{\tau_R} \frac{c_{\ref{c2.2}}}{\rho_s^p}ds\Bigr)\big|\tau_R<\infty \Big) (R/|x|)^{p}, 
\end{align*}
with $p=\mu+\nu$.
 A scaling argument shows that the above equals
\begin{equation}\label{Dllb}D^\lambda(R)(R/|x|)^{p} E_{|x|/R}^{(2+2\nu)}\Big(\exp\Big(-\int_{0}^{\tau_1} \frac{c_{\ref{c2.2}}R^{2-p}}{\rho_s^p}ds\Big)|\tau_1<\infty \Big) . 
\end{equation}
To apply Lemma \ref{expbound2} note that if $R\geq R_{\ref{p1.1}}$ for some constant $R_{\ref{p1.1}}>2$,
\[2\gamma\equiv 2c_{\ref{c2.2}}R^{2-p}<2(4-d)<\nu^2.\]  By Cauchy-Schwartz, we have
\begin{align*}
1&=\Bigg(E_{|x|/R}^{(2+2\nu)}\Big(\exp\Bigl(-\int_{0}^{\tau_1} \frac{c_{\ref{c2.2}}R^{2-p}}{2\rho_s^p}ds\Bigr) \exp\Bigl(\int_{0}^{\tau_1} \frac{c_{\ref{c2.2}}R^{2-p}}{2\rho_s^p}ds\Bigr)\big|\tau_1<\infty \Big) \Bigg)^2\\
&\leq E_{|x|/R}^{(2+2\nu)}\Big(\exp\Bigl(-\int_{0}^{\tau_1} \frac{c_{\ref{c2.2}}R^{2-p}}{\rho_s^p}ds\Bigr) \big|\tau_1<\infty \Big) E_{|x|/R}^{(2+2\nu)}\Big( \exp\Bigl(\int_{0}^{\tau_1} \frac{c_{\ref{c2.2}}R^{2-p}}{\rho_s^p}ds\Bigr)|\tau_1<\infty \Big)\\
&\leq E_{|x|/R}^{(2+2\nu)}\Big(\exp\Bigl(-\int_{0}^{\tau_1} \frac{c_{\ref{c2.2}}R^{2-p}}{\rho_s^p}ds\Bigr) \big|\tau_1<\infty \Big) c_{\ref{expbound2}}(p,\nu),
\end{align*}
the last by Lemma \ref{expbound2}.
Hence \[\inf_{|x|\geq R} E_{|x|/R}^{(2+2\nu)}\Big(\exp\Bigl(-\int_{0}^{\tau_1} \frac{c_{\ref{c2.2}}R^{2-p}}{\rho_s^p}ds\Bigr) \big|\tau_1<\infty \Big) \geq c_{\ref{expbound2}}(p,\nu)^{-1}>0, \]
and by \eqref{Dllb} we are done. 
\end{proof}

By using the scaling relations of $U^{\infty, \varepsilon}$ and $U^{\lambda \varepsilon^{-2},\varepsilon}$ from \eqref{scaling1}, the following is immediate from 
the above.
\begin{corollary}\label{c1.4}
For all $\varepsilon>0$, we have
\begin{enumerate}[(a)]
\item \[U^{\infty, \varepsilon}(x)-U^{\lambda \varepsilon^{-2},\varepsilon}(x) \leq \frac{R^p}{|x|^p}D^{\lambda}(R)\varepsilon^{p-2},\ \forall |x|/\varepsilon \geq R>1,\ \lambda\geq 6.\]
\item \[U^{\infty, \varepsilon}(x)-U^{\lambda\varepsilon^{-2},\varepsilon}(x) \leq C_{\ref{p1.1}} \frac{R^p}{|x|^p}D^{\lambda}(R)\varepsilon^{p-2},\ \forall |x|/\varepsilon \geq R\geq \frac{K_{\ref{p1.1}}}{\lambda}, 0<\lambda<1.\]
\item \[U^{\infty, \varepsilon}(x)-U^{\lambda\varepsilon^{-2},\varepsilon}(x) \geq c_{\ref{p1.1}} \frac{R^p}{|x|^p}D^{\lambda}(R)\varepsilon^{p-2}>0,\ \forall |x|/\varepsilon\geq R\geq R_{\ref{p1.1}},\ \lambda>0.\]
\end{enumerate}
\end{corollary}

\subsection{The left tail of the total exit measure and Proof of Theorem~\ref{t2.3}}
\label{sec:3.2}
\begin{proposition}\label{t1.1}
For any $|x|\geq \varepsilon_0$ and $\varepsilon \in (0,\varepsilon_0/2)$, we have
\[ \P_{\delta_x}\Big(0<\frac{X_{G_\varepsilon}(1)}{\varepsilon^2}\leq \frac{1}{\lambda}\Big)\leq e\frac{2^p}{|x|^p}D^\lambda(2) \varepsilon^{p-2}, \ \forall  \lambda\geq 6.\]
\end{proposition}

\begin{proof}

Apply Markov's inequality to get
\begin{align}\label{e1.7}
\P_{\delta_x}\Big(0<\frac{X_{G_\varepsilon}(1)}{\varepsilon^2}\leq \frac{1}{\lambda}\Big)& \leq e\E_{\delta_x}\Big(\exp(-\lambda \varepsilon^{-2} X_{G_\varepsilon}(1)) 1(X_{G_\varepsilon}(1)>0)\Big)\nonumber\\
&=e\Big(\exp ({-U^{\lambda \varepsilon^{-2}, \varepsilon}(x)})-\exp ({-U^{\infty, \varepsilon}(x)} )\Big)\nonumber\\
&\leq e\Big({U^{\infty, \varepsilon}(x)}-U^{\lambda \varepsilon^{-2}, \varepsilon}(x)\Big),
\end{align}
the equality by \eqref{LTexit} and \eqref{UinfLT}. Note we've assumed $|x|/\varepsilon \geq \varepsilon_0/\varepsilon>2$ and $\lambda\ge 6$ so that we can use Corollary \ref{c1.4}(a) with $R=2$ to bound the right-hand side of \eqref{e1.7} by $e (2/|x|)^p D^{\lambda}(2) \varepsilon^{p-2}$, as required.
\end{proof}

\begin{proposition}\label{t1.2}
There is some $c_{\ref{t1.2}}(\varepsilon_0)>0$ such that for all $|x|\geq \varepsilon_0$ and $\varepsilon \in (0,\varepsilon_0)$,
\[ \P_{\delta_x}(0<\frac{X_{G_\varepsilon}(1)}{\varepsilon^2}\leq\frac{1}{\lambda})\leq c_{\ref{t1.2}}(\varepsilon_0) \lambda^{-(p-2)} \varepsilon^{p-2},\ \forall\  0<\lambda<1.\]
\end{proposition}

\begin{proof}

For $\lambda \in (0,1)$ such that \[|x|/\varepsilon \geq \varepsilon_0/\varepsilon \geq K_{\ref{p1.1}}/\lambda,\] 
we apply Markov's inequality as in \eqref{e1.7} and use Corollary \ref{c1.4}(b) with $R=K_{\ref{p1.1}}/\lambda$  to get
\begin{align*}
\P_{\delta_x}(0<\frac{X_{G_\varepsilon}(1)}{\varepsilon^2}\leq \frac{1}{\lambda})
& \leq e\Big({U^{\infty, \varepsilon}(x)}-U^{\lambda \varepsilon^{-2}, \varepsilon}(x)\Big) \leq e C_{\ref{p1.1}} U^{\infty,1}(R) (R/|x|)^{p} \varepsilon^{p-2}\\
& \leq e C_{\ref{p1.1}} (3(4-d)/R^2) (R/\varepsilon_0)^{p} \varepsilon^{p-2} \leq 9eC_{\ref{p1.1}}   \varepsilon_0^{-p} K_{\ref{p1.1}}^{p-2}  \lambda^{-(p-2)} \varepsilon^{p-2} ,
\end{align*}
the second line by  \eqref{e1.4} with $R>K_{\ref{p1.1}}>K_{\ref{c2.2}}$.

For $\lambda \in (0,1)$ such that \[ \varepsilon_0/\varepsilon \leq K_{\ref{p1.1}}/\lambda,\] we have
\[\P_{\delta_x}(0<\frac{X_{G_\varepsilon}(1)}{\varepsilon^2}\leq \frac{1}{\lambda}) \leq 1\leq (K_{\ref{p1.1}}/\varepsilon_0 )^{p-2} \lambda^{-(p-2)} \varepsilon^{p-2}.\] 

The result follows by letting $c_{\ref{t1.2}}(\varepsilon_0)=9eC_{\ref{p1.1}}   \varepsilon_0^{-p} K_{\ref{p1.1}}^{p-2}+(K_{\ref{p1.1}}/\varepsilon_0 )^{p-2}.$
\end{proof}

For $|x|\geq \varepsilon_0$ and $\varepsilon \in (0,\varepsilon_0)$, we define
\begin{align}\label{e1.8.1}
F_{\varepsilon,x}(K)\equiv \P_{\delta_x}(0<\frac{X_{G_\varepsilon}(1)}{\varepsilon^2}\leq K), \ \forall K>0,
\end{align}
and
\begin{align}\label{e1.8}
\hat{F}_{\varepsilon,x}(\lambda)\equiv \E_{\delta_x}\Big(\exp\Bigl(-\lambda \frac{X_{G_{\varepsilon}}(1)}{\varepsilon^2}\Bigr)1{(X_{G_{\varepsilon}}(1)>0)}\Big), \ \forall \lambda>0.
\end{align}
The dependence on $\varepsilon$ and $x$ will at times be suppressed.

\begin{lemma}\label{t2.2}
There are constants $R_{\ref{t2.2}}>2$ and $c_{\ref{t2.2}}(\varepsilon_0)>0$ so that for any $\varepsilon_0\leq |x|\leq \varepsilon_0^{-1}$ and $\varepsilon \in (0,\varepsilon_0/R_{\ref{t2.2}})$, \[ \E_{\delta_x}\Bigl(\exp\Bigl(-\lambda \frac{X_{G_{\varepsilon}}(1)}{\varepsilon^2}\Bigr)1{(X_{G_{\varepsilon}}(1)>0)}\Bigr) \geq c_{\ref{t2.2}}(\varepsilon_0) D^{\lambda}(R_{\ref{t2.2}}) \varepsilon^{p-2},\ \forall \lambda>0. \]
\end{lemma}
\begin{proof}
By  \eqref{LTexit} and \eqref{UinfLT}, for $\hat{F}=\hat{F}_{\varepsilon,x}$ defined as in \eqref{e1.8} we have
\begin{align}\label{hatFlb}
\hat{F}(\lambda)&= \exp ({-U^{\lambda \varepsilon^{-2}, \varepsilon}(x)})-\exp ({-U^{\infty, \varepsilon}(x)} )
\geq  ({U^{\infty, \varepsilon}(x)}-U^{\lambda \varepsilon^{-2}, \varepsilon}(x)) \exp ({-U^{\infty, \varepsilon}(x)} ).
\end{align}
Let $R_{\ref{t2.2}} =K_{\ref{c2.2}}+ R_{\ref{p1.1}}>R_{\ref{p1.1}}$.
 Then for $\varepsilon \in (0,\varepsilon_0/R_{\ref{t2.2}})$ we have
$|x|/\varepsilon \geq \varepsilon_0/\varepsilon>R_{\ref{t2.2}}$. Use Corollary \ref{c1.4}(c) with $R=R_{\ref{t2.2}}$ to get
\[{U^{\infty, \varepsilon}(x)}-U^{\lambda \varepsilon^{-2}, \varepsilon}(x) \geq  c_{\ref{p1.1}} D^\lambda(R_{\ref{t2.2}}) (R_{\ref{t2.2}}/|x|)^p  \varepsilon^{p-2} \geq  c_{\ref{p1.1}} D^\lambda(R_{\ref{t2.2}}) R_{\ref{t2.2}}^p \varepsilon_0^{-p}  \varepsilon^{p-2}.\]
Next since $|x|/\varepsilon \geq \varepsilon_0/\varepsilon>R_{\ref{t2.2}}>K_{\ref{c2.2}}$, we may apply \eqref{e1.4} to get 
\[\exp ({-U^{\infty, \varepsilon}(x)} )=\exp({-\varepsilon^{-2} U^{\infty, 1}(x/\varepsilon)}) \geq \exp(-\varepsilon^{-2} 3(4-d) (|x|/\varepsilon)^{-2}) \geq \exp({-9\varepsilon_0^{-2} }).\] So the lemma follows from \eqref{hatFlb} and the above inequalities.
\end{proof}

\begin{proposition}\label{t2.4}
There are positive constants $K_{\ref{t2.4}}(\varepsilon_0)$ and $c_{\ref{t2.4}}(\varepsilon_0)$  such that, for all $\varepsilon_0 \leq |x| \leq \varepsilon_0^{-1}$ we have
\[\P_{\delta_x}(0< \frac{X_{G_{\varepsilon}}(1)}{\varepsilon^2} \leq K_{\ref{t2.4}}(\veps_0)) \geq c_{\ref{t2.4}}(\varepsilon_0) \varepsilon^{p-2},\ \forall 0<\varepsilon<\varepsilon_0/R_{\ref{t2.2}}.\]
\end{proposition}
\begin{proof}
Recall $F={F}_{\varepsilon,x}$ and $\hat{F}=\hat{F}_{\varepsilon,x}$ from \eqref{e1.8.1} and \eqref{e1.8}, respectively. We have
\begin{align*}
\hat{F}(\lambda)=\int_0^\infty e^{-\lambda y} dF(y).
\end{align*}
Let $\lambda=1$ and $K>1$. Use integration by parts and Proposition \ref{t1.2} to see that
\begin{align*}
\hat{F}(1)=\int_0^\infty e^{-y} F(y) dy &\leq F(K)+\int_K^{\infty} e^{-y} F(y) dy \leq  F(K)+\int_K^{\infty} e^{-y}  c_{\ref{t1.2}}(\varepsilon_0) y^{p-2} \varepsilon^{p-2} dy \\
& \leq F(K)+\frac{1}{2}  c_{\ref{t2.2}}(\varepsilon_0) D^1(R_{\ref{t2.2}})  \varepsilon^{p-2},
\end{align*}
where $K=K(\varepsilon_0)>1$ is large enough. Lemma \ref{t2.2}, with $\lambda=1$ and $\veps$, $x$ as in the Proposition, implies
\begin{align*}
F(K)&\geq  c_{\ref{t2.2}}(\varepsilon_0) D^1(R_{\ref{t2.2}}) \varepsilon^{p-2}-\frac{1}{2}  c_{\ref{t2.2}}(\varepsilon_0) D^1(R_{\ref{t2.2}})  \varepsilon^{p-2}= \frac{1}{2} c_{\ref{t2.2}}(\varepsilon_0) D^1(R_{\ref{t2.2}}) \varepsilon^{p-2}.
\end{align*}\end{proof}

\begin{proof}[Proof of Theorem~\ref{t2.3}.]
Pick $\lambda=\lambda(\varepsilon_0)\ge 6$ large enough so that \[e (2^p/|x|^p) D^\lambda(2)\leq e 2^p \varepsilon_0^{-p} D^\lambda(2) \leq \frac{1}{2} c_{\ref{t2.4}}(\varepsilon_0).\]  So for $K_1(\veps_0)\in(0,1/\lambda(\veps_0))$, Proposition~\ref{t1.1} gives \[\P_{\delta_x}(0<\frac{X_{G_{\varepsilon}}(1)}{\varepsilon^2} \leq K_1)\leq \frac{1}{2} c_{\ref{t2.4}}(\varepsilon_0)  \varepsilon^{p-2}.\] Let $K_2=K_{\ref{t2.4}}$ in Proposition~\ref{t2.4} to see that for $x,\veps$ as in the Theorem and $R_{\ref{t2.3}}=R_{\ref{t2.2}}$,
\begin{align}\label{e1.11}
\P_{\delta_x}(K_1\leq \frac{X_{G_{\varepsilon}(1)}}{\varepsilon^2} \leq K_2)= \P_{\delta_x}(0<\frac{X_{G_{\varepsilon}}(1)}{\varepsilon^2} \leq K_2)-\P_{\delta_x}(0<\frac{X_{G_{\varepsilon}}(1)}{\varepsilon^2} \leq K_1) \geq \frac{1}{2} c_{\ref{t2.4}}(\varepsilon_0)  \varepsilon^{p-2}.
\end{align}
Use Proposition \ref{p0.1}(b)(i) with $G=G_\varepsilon$ and $D_1=G_{\varepsilon/2}$ to see that for $x,\veps$ as above,
\begin{align*}
&\P_{\delta_x}(K_1\leq \frac{X_{G_{\varepsilon}}(1)}{\varepsilon^2} \leq K_2, X_{G_{\varepsilon/2}}(1)=0)=\E_{\delta_x}( 1(K_1\leq \frac{X_{G_{\varepsilon}}(1)}{\varepsilon^2} \leq K_2) \P_{X_{G_{\varepsilon}}}(X_{G_{\varepsilon/2}}(1)=0))\\
=& \E_{\delta_x}( 1(K_1\leq \frac{X_{G_{\varepsilon}}(1)}{\varepsilon^2} \leq K_2) \exp(-\int U^{\infty,\varepsilon/2}(y) X_{G_{\varepsilon}}(dy)))\ \text{ (by } \eqref{UinfLT} \text{ with } X_0=X_{G_{\varepsilon}} )\\
=& \E_{\delta_x}( 1(K_1\leq \frac{X_{G_{\varepsilon}}(1)}{\varepsilon^2} \leq K_2) \exp(-  4\varepsilon^{-2} U^{\infty,1}(2) X_{G_{\varepsilon}}(1)))\
\text{ (by }\eqref{scaling1} )\\
\geq & \E_{\delta_x}( 1(K_1\leq \frac{X_{G_{\varepsilon}}(1)}{\varepsilon^2} \leq K_2) \exp(-  4K_2 U^{\infty,1}(2) ))\geq    \frac{1}{2} c_{\ref{t2.4}}(\varepsilon_0)  \varepsilon^{p-2} \exp(-  4K_2 U^{\infty,1}(2) ),
\end{align*}
the last by \eqref{e1.11}. So the theorem follows.
\end{proof}

\section{
Preliminaries for the Lower Bound on the Dimension
}\label{nonpolar}
\label{sec:4}
In this section, we will show that the lower bound on the local dimension of $\pmR$ holds with positive probability (see  Proposition~\ref{prop:crudelbdim}). The refined version of this result, which is crucial for the later proof of Proposition~\ref{candim}, is given in Lemma~\ref{lem:spherelb}. 
The next result is important for implementing our program: it plays a role analogous to that of Proposition~6.1 in~\cite{MP17}.

\begin{proposition}\label{p3.1}
There is a $\lambda>0$ such that for all $\varepsilon_0>0$, there is some $c_{\ref{p3.1}}(\varepsilon_0)>0$ so that for all $|x_i|\geq \varepsilon_0$ and $\varepsilon \in (0,\varepsilon_0)$,
\[\E_{\delta_0}\Big(\prod_{i=1}^2 \lambda \frac{X_{G_\varepsilon^{x_i}}(1)}{\varepsilon^2} \exp\big(-\lambda \frac{X_{G_\varepsilon^{x_i}}(1)}{\varepsilon^2}\big) \Big)\leq c_{\ref{p3.1}} (1+|x_1-x_2|^{2-p}) \varepsilon^{2(p-2)}.\]
\end{proposition}
\noindent \SUBMIT{Given the results in Section~\ref{lowerbndexitmeas}, the proof then follows closely to that of Proposition~6.1 in \cite{MP17}, and so is omitted. It is included in the Appendix
of a longer version of this work on the ArXiv \cite{HMP18A}.}\ARXIV{Given the results in Section~\ref{lowerbndexitmeas}, the proof then follows that of Proposition~6.1 in \cite{MP17}, and so is deferred to Appendix~\ref{sec:5}.} The reader should note that the role of $\lambda$ in \cite{MP17} is now played by $\lambda \veps^{-2}$, where $\lambda$ is chosen to be a fixed large constant.

Recall that we are in the case $d=2$ or $3$. Let $\beta>0$ and $g_\beta(r)=r^{-\beta}$. For a finite measure $\mu$ on $\R^d$ and Borel subset $A$ of $\R^d$, let \[\langle \mu \rangle_{g_\beta}  =\int \int g_\beta(|x-y|) d\mu(x) d\mu(y),\] and \[I(g_\beta)(A)=\inf \{\langle \mu \rangle_{g_\beta}: \mu \text{ is a probability supported by } A\}.\]
The $g_\beta$-capacity of $A$ is $C(g_\beta)(A)=(I(g_\beta)(A))^{-1}$. Set 
\[\beta=p-2=
\begin{cases}
2\sqrt{2}-2,\ &\text{ if } d=2,\\
(\sqrt{17}-3)/2, \ &\text{ if } d=3,
\end{cases}
\]
 and note $\beta \in (1/2,1)$. Now we may use Theorem \ref{t2.3} and Proposition \ref{p3.1}  to get the following theorem. Although similar to the omitted proof of Theorem 6.2 in \cite{MP17}, there are some important adjustments, and so this time the argument is included.
\begin{theorem}\label{t3.2}
Assume $d=2$ or $3$. For every $\varepsilon_0 \in (0,1)$, there is a $c_{\ref{t3.2}}(\varepsilon_0)>0$ such that for any Borel set, $A$, of $\{x\in \R^d: \varepsilon_0\leq |x|\leq \varepsilon_0^{-1} \}$,
\[\P_{\delta_0}(\partial \mR\cap A\neq \emptyset) \geq c_{\ref{t3.2}}(\varepsilon_0) C(g_\beta)(A).\] In particular for any Borel subset $A$ of $\R^d$, $C(g_\beta)(A)>0$ implies that $\P_{\delta_0}(\partial \mR\cap A\neq \emptyset) >0$.
\end{theorem}
\begin{proof}
Fix $\varepsilon_0 \in (0,1)$. We approximate $\partial \mR$ by 
\[\pmR(\varepsilon):=\{x: K_1 \varepsilon^2 \leq  X_{G_\varepsilon^{x}}(1) \leq K_2 \varepsilon^2, X_{G_{\varepsilon/2}^{x}}(1)=0 \}. \]
where $0<K_1(\varepsilon_0)<K_2(\varepsilon_0)<\infty$ are as in Theorem \ref{t2.3}. Then for $\lambda >0$ as in Proposition \ref{p3.1}, there is some $\delta=\delta(\lambda, \varepsilon_0) \in (0,e^{-1})$ such that 
\begin{align}\label{e3.0}
\lambda \frac{X_{G_\varepsilon^{x}}(1)}{\varepsilon^2} \exp\big(-\lambda \frac{X_{G_\varepsilon^{x}}(1)}{\varepsilon^2}\big) \geq \delta,  \ \forall x \in \pmR(\varepsilon).
\end{align}
Let $\Gamma$ be a compact subset of $\{x\in \R^d: \varepsilon_0\leq |x|\leq \varepsilon_0^{-1} \}$ such that (without loss of generality) $C(\Gamma)=C(g_\beta)(\Gamma)>0$. If $I(\Gamma)=I(g_\beta)(\Gamma)$, we may choose $\{x^N_i: 1\leq i \leq N\} \subset \Gamma$ so that (suppressing the superscript $N$) as $N\to \infty$ (see \cite{Tay61}),
\begin{align}\label{e3.1}
I_N \equiv \frac{1}{N(N-1)} \sum_i \sum_{j\neq i} |x_i-x_j|^{-\beta} \to I(\Gamma)=1/C(\Gamma).
\end{align}
Therefore by translation invariance, inclusion-exclusion, Theorem \ref{t2.3}, \eqref{e3.0}, and Proposition \ref{p3.1}, for $\varepsilon \in (0,\varepsilon_0/R_{\ref{t2.3}})$, 
\begin{align*}
\P_{\delta_0} (\Gamma\cap \pmR(\varepsilon) \neq \emptyset)&\geq \sum_{j=1}^N \P_{\delta_0} (x_j \in   \pmR(\varepsilon))-\sum_i \sum_{j\neq i} \P_{\delta_0} (x_i, x_j \in   \pmR(\varepsilon))\\
&\geq N c_{\ref{t2.3}}(\varepsilon_0) \varepsilon^{p-2} -\sum_i \sum_{j\neq i} \delta^{-2} \E_{\delta_0}\Big(\prod_{k=i,j} \lambda \frac{X_{G_\varepsilon^{x_k}}(1)}{\varepsilon^2} \exp\big(-\lambda \frac{X_{G_\varepsilon^{x_k}}(1)}{\varepsilon^2}\big) \Big)\\
&\geq N c_{\ref{t2.3}} \varepsilon^{p-2} - c_{\ref{p3.1}} \delta^{-2} \varepsilon^{2(p-2)} \sum_i \sum_{j\neq i} (1+|x_i-x_j|^{2-p})\\
&\geq c_{\ref{t2.3}} N \varepsilon^{p-2} - C(\varepsilon_0) (N \varepsilon^{p-2})^2 I_N.
\end{align*}
Now choose $\varepsilon_N \to 0$ such that $N \varepsilon_N^{p-2}=c_{\ref{t2.3}}/(2C(\varepsilon_0) I_N)$. Therefore, for some $c(\veps_0)>0$,
\[\P_{\delta_0} (\Gamma\cap \pmR(\varepsilon_N) \neq \emptyset)\geq  \frac{c_{\ref{t2.3}}}{2C(\varepsilon_0) I_N} \frac{c_{\ref{t2.3}} }{2} \to c(\varepsilon_0) C(\Gamma), \text{ as } N\to \infty.\]
This implies
\[\P_{\delta_0} (\Gamma\cap \pmR(\varepsilon_N) \neq \emptyset, \text{ infinitely often})\geq  c(\varepsilon_0) C(\Gamma).\] 

Assume now that 
\[\omega\in\{\Gamma\cap \pmR(\varepsilon_N) \neq \emptyset, \text{ infinitely often}\}.\]
So we may choose 
$\{x_N\} \subset \Gamma$ such that $x_N \in \pmR(\varepsilon_N)$, where we have suppressed the further subsequence of $\varepsilon_N$ in our notation.  The definition of $ \pmR(\varepsilon_N)$ gives $X_{G_{\varepsilon_N}^{x_N}}(1)>0$ and  $X_{G_{\varepsilon_N/2}^{x_N}}(1)=0$. By Proposition \ref{p0.1}(b)(ii) and translation invariance, we have $\P_{\delta_0}$-a.s.
\begin{equation}\label{hitsb}\mR \cap B(x_N, \varepsilon_N/3)=\emptyset.
\end{equation}
By \eqref{exitsupport}, $X_{G_{\varepsilon_N}^{x_N}}(1)>0$ implies 
\begin{equation}\label{hitbb}
\mathcal{R}\cap \partial G_{\varepsilon_N}^{x_N}\text{ is non-empty. }
\end{equation}
Combining \eqref{hitsb} and \eqref{hitbb} with an elementary argument in point set
topology
 we can choose $y_N \in \partial \mR$ such that $\varepsilon_N/4\leq |y_N-x_N|\leq \varepsilon_N$. The compactness of $\Gamma$ implies there is some $x \in \Gamma$ such that $x_{N_k} \to x$ as $N_k \to \infty.$ Therefore $y_{N_k}\to x$ and $x \in \partial \mR$ since $\partial \mR$ is closed, which gives $x\in \Gamma\cap \partial \mR \neq \emptyset$, and so the proof is complete for $A = \Gamma$ compact. Use the
inner regularity of capacity to extend the result to any Borel subset of $\{x\in \R^d: \varepsilon_0\leq |x|\leq \varepsilon_0^{-1} \}$.
\end{proof}
\begin{proposition}\label{prop:crudelbdim}
For each non-empty open set $B$ in $\R^d$,\\
(a) $\P_{\delta_0}(\textnormal{dim}(\partial \mR\cap B)\ge d_f)>0$\\
(b) $\N_0(\textnormal{dim}(\partial \mR\cap B)\ge d_f):=p_{\ref{prop:crudelbdim}}(B)>0$
\end{proposition}
\begin{proof} (a) is derived from Theorem~\ref{t3.2} by taking $A$ to be the range of an appropriate independent L\'evy process, exactly as in the proof of Proposition 6.5 of \cite{MP17}.  (b) then follows easily from (a) by making trivial changes to the cluster decomposition proof of Corollary~6.6 in \cite{MP17}.
\end{proof}

To help upgrade the lower bound in part (a) of the above to probability one, we need to
extend (a) to more general initial conditions through a scaling argument.

\begin{lemma}\label{lem:spherelb}
There is a constant $q_{\ref{lem:spherelb}}>0$ so that if $X_0'\in M_F(\R^d)$ is supported on \\
$\{|x|=r\}$ and $\delta=X_0'(1)$ satisfies $0<\delta\le r^2$, then
\begin{equation*}
\P_{X_0'}\Bigl(\textnormal{dim}\Bigl(\pmR\cap B\Bigl(0,r-\frac{\sqrt\delta}{2}\Bigr)\Bigr)\ge d_f\Bigr)\ge q_{\ref{lem:spherelb}}.
\end{equation*}
\end{lemma}
\begin{proof} Define $X_0^{(\delta)}(A)=\delta^{-1}X'_0(\sqrt\delta A)$,
so that $X_0^{(\delta)}$
 is supported on $\{|x|=r/\sqrt\delta\}$ and has total mass one.  
By scaling properties of SBM (see, e.g., Ex. II.5.5 in \cite{Per02}) and scale invariance of Hausdorff dimension, we may
conclude that 
\begin{align}\label{scaledim}
&\P_{X_0'}(\text{dim}(\pmR\cap B\Bigl(0,r-\frac{\sqrt\delta}{2}\Bigr)\ge d_f)\\
\nn&\qquad=\P_{X_0^{(\delta)}}(\text{dim}(\pmR\cap B\Bigl(0,\frac{r}{\sqrt\delta}-\frac{1}{2}\Bigr)\ge d_f).
\end{align}
Now work in our standard set-up for SBM with initial law $X_0^{(\delta)}$ so that (by \eqref{Xtdec}), \break
$X_t=\sum_{j\in J}X_t(W_j)=\int X_t(W)\Xi(dW)$ for all $t>0$, where $\Xi$ is a Poisson point process with intensity $\N_{X_0^{(\delta)}}$.
For $r\ge \sqrt\delta$ define
\begin{align*}
&\tau_\rho(W_j)=\inf\{t\ge 0:|\hat W_j(t)|\le \rho\},\\
&U_\rho(W_j)=\inf\{t\ge 0:|\hat W_j(t)-\hat W_j(0)|\ge\rho\},\\
&\text{and } N_1=\sum_{j\in J}1(\tau_{(r/\sqrt\delta)-(1/2)}(W_j)<\infty):=\#(I_{r,\delta}).
\end{align*}
Here as usual $\inf\emptyset=\infty$.  Then $N_1$ is Poisson with mean
\begin{align}\label{Nonebnd}
m_{r,\delta}:=\N_{X_0^{(\delta)}}(\tau_{(r/\sqrt\delta)-(1/2)}<\infty)
&\le \N_{X_0^{(\delta)}}(U_{1/2}(W)<\infty)\\
\nn&=\N_{0}(U_{1/2}(W)<\infty):=\overline m<\infty,
\end{align}
where $X_0^{(\delta)}(1)=1$ and translation invariance are used in the equality, and the finiteness of $\bar m$ follows from Theorem~1 of \cite{Isc88}. Let $\mR(W_j)=\{\hat W_j(t):t\le\sigma(W_j)\}$ (recall \eqref{snakerange}) be the range of the $j$th excursion, so that 
\[\mR\cap B\Bigl(0,\frac{r}{\sqrt\delta}-\frac{1}{2}\Bigr)=\cup_{j\in J, \tau_{(r/\sqrt\delta)-(1/2)}(W_j)<\infty} \Bigl(\mR(W_j)\cap B\Bigl(0,\frac{r}{\sqrt\delta}-\frac{1}{2}\Bigr)\Bigr).\]
We may assume (by additional randomization) that conditional on $I_{r,\delta}$, $\{W_j:j\in I_{r,\delta}\}$ are iid with law $\N_{X_0^{(\delta)}}(W\in\cdot|\tau_{(r/\sqrt\delta)-(1/2)}<\infty)$. Therefore the right-hand side of \eqref{scaledim} is at least
\begin{align}\label{Nonebndb}
\P_{X_0^{(\delta)}}&(N_1=1)\N_{X_0^{(\delta)}}\Bigl(\text{dim}\Bigl(\pmR\cap B(0,\frac{r}{\sqrt\delta}-\frac{1}{2})\Bigr)\ge d_f\Bigl|\tau_{\frac{r}{\sqrt\delta}-\frac{1}{2}}<\infty\Bigr)\\
\nn&=\frac{m_{r,\delta}e^{-m_{r,\delta}}}{m_{r,\delta}}\N_{x_0}\Bigl(\text{dim}\Bigl(\pmR\cap B(0,\frac{r}{\sqrt\delta}-\frac{1}{2})\Bigr)\ge d_f\Bigr),
\end{align}
where $x_0=(\frac{r}{\sqrt\delta})e_1$ and $e_1$ is the first unit basis vector. We also have used the fact that spherical symmetry shows we could have taken any $x_0$ on
the sphere of radius $r/\sqrt\delta$.  Now again use scaling to see that the right side of \eqref{Nonebndb} equals
\begin{align}\label{Ntwobnd}
\nn e^{-m_{r,\delta}}&\N_{x_0}\Bigl(\text{dim}\Bigl(\pmR\cap B(0,|x_0|-\frac{1}{2})\Bigr)\ge d_f\Bigr)\\
\nn&=e^{-m_{r,\delta}}\N_{e_1}(\text{dim}(\pmR\cap B(0,1-(2|x_0|)^{-1}))\ge d_f)\\
\nn&\ge e^{-\overline m}\N_{e_1}(\text{dim}(\pmR\cap B(0,1/2))\ge d_f)\quad(\text{recall that }|x_0|\ge 1)\\
&\ge e^{-\overline m}p_{\ref{prop:crudelbdim}}(B(-e_1,1/2))>0,
\end{align}
where the next to last inequality holds by translation invariance and Proposition~\ref{prop:crudelbdim}(b), and the first inequality uses \eqref{Nonebnd}. 
We have shown that the right-hand side of \eqref{Ntwobnd} is a lower bound for 
 \eqref{scaledim}, and so have proved the lemma with $q_{\ref{lem:spherelb}}=e^{-\overline m}p_{\ref{prop:crudelbdim}}(B(-e_1,1/2))$. 
\end{proof}

\section{Exit Measures and Continuous State Branching Processes}
\label{sec:6}
To finish the proofs of Propositions~\ref{exitdim}, \ref{candim} we need to establish some properties of the total exit measure process $X_{G_{r_0-r}}(1), 0\leq r< r_0$. We will show in 
Proposition~\ref{prop:yzinfo} that, for any $r_0>0$, the ``time changed"  process $Z_t=X_{G_{r_0e^{-t}}}(1)/(r_0e^{-t})^2, t\geq 0,$ is a time homogeneous continuous state branching process (CSBP) and thus has no negative jumps. 

\medskip

 
A non-negative function $\lambda\mapsto u(\lambda)$ on $[0,\infty)$ is completely concave iff for every $y>0$ 
$\lambda\to \exp(-yu(\lambda))$ is the Laplace transform of a probability measure on the half-line. (See (4.1) in \cite{Si68} for a L\'evy-Khintchine representation of such functions).
We recall the definition of a continuous state branching process from Section~4 of \cite{Si68}.

\medskip

\noindent{\bf Definition} A (finite) continuous state branching process (CSBP) $Z$  is a time-homogeneous $[0,\infty)$-valued Markov process with no fixed time discontinuities
(if $t_n\to t$, then $Z(t_n)\to Z(t)$ a.s.),  and such that  there is a family of completely concave functions $\{u(s,\cdot):s>0\}$ satisfying
\begin{equation}\label{csbp}
E(\exp(-\lambda Z(t_2))|Z(s),s\le t_1)=\exp(-Z(t_1)u(t_2-t_1,\lambda))\text{ a.s.  for all }t_2>t_1\ge 0.
\end{equation}
We call the associated family $\{u(t,\cdot):t>0\}$ the log Laplace transform of $Z$. 

Recall that $U^{\lambda,R}(x)=U^{\lambda,R}(|x|)$ is the unique continuous map on $\{|x|\ge R\}$ which is $C^2$ on $G_R$ and satisfies
\begin{equation}\label{UPDE}
\Delta U=U^2\text{ on }G_R\text{ and }U=\lambda\text{ on }\partial G_R.
\end{equation}
A simple application of the comparison principle (e.g., Chapter V, Lemma~7 of \cite{Leg99}), using the last part of \eqref{Uinftyprop}, gives
\begin{equation}\label{Ulbnd}
U^{\lambda,R}(x)\le \lambda\quad\forall |x|\ge R.
\end{equation}
Define
\begin{equation}\label{udef}
u(t,\lambda)=e^{2t}U^{\lambda,1}(e^t)\text{ for }t\ge 0.
\end{equation}
For the remainder of this section we assume that $r_0>0$ satisfies 
\begin{equation}\label{r0cond}
B_{2r_0}\subset \text{Supp}(X_0)^c.
\end{equation}

\no{\bf Notation.} For $0\le r<r_0$ we define $Y(r)=X_{G_{r_0-r}}$, $\cE_r=\cE_{G_{r_0-r}}\vee\{\N_{X_0}-\text{ null sets}\}$, and 
for $t\ge 0$ set
\[Z(t)=X_{G_{r_0e^{-t}}}(1)\frac{e^{2t}}{r_0^2}=Y(r_0(1-e^{-t}))(1)e^{2t}r_0^{-2}\ \text{ and }\ \cG_t=\cE_{r_0(1-e^{-t})}=\cE_{G_{r_0e^{-t}}}.\]

It is not hard to show that $\cE_r$ is non-decreasing in $r$ (the corresponding result for 
half-spaces is noted prior to (7.2) of \cite{MP17} and the observation  made there applies to balls as well.)  By Proposition 2.3 of \cite{Leg95}, $Y$ is $(\cE_r)$-adapted and $Z$ is $(\cG_t)$-adapted. Let $\cE^+_r=\cE_{r+}$ denote the associated right-continuous filtration.  
In addition to $\N_{X_0}$, we will also work under the probability $Q_{X_0}(\cdot)=\N_{X_0}(\cdot|Y_0(1)>0)$, where \eqref{r0cond} ensures that $\N_{X_0}(Y_0(1)>0)<\infty$.  Note that  
\begin{equation}\label{Qcond}
\text{for any r.v. }Z\ge 0,\text{ and any }r\ge 0,\ Q_{X_0}(Z|\cE_r)=\N_{X_0}(Z|\cE_r)\ Q_{X_0}-\text{a.s.}
\end{equation}
because $\{Y_0(1)>0\}\in \cE_0$.  When conditioning on $\cE_r$ under $Q_{X_0}$, we are adding the slightly larger class of $Q_{X_0}$-null sets to $\cE_r$, but will not record this distinction in our notation. Below we will apply the definition of (CSBP) under the $\sigma$-finite measure $\N_{X_0}$ as well as $Q_{X_0}$. 
We write $Q_{x_0}$ for $Q_{\delta_{x_0}}$ as usual. 

\begin{lemma}\label{lem:Zstuff}
(a) If $0\le t_1<t_2$ and $\lambda\ge 0$, then\\

(i) $Q_{X_0}\Bigl(e^{-\lambda Z_{t_2}}\Bigl|\cG_{t_1}\Bigr)=\N_{X_0}\Bigl(e^{-\lambda Z_{t_2}}\Bigl|\cG_{t_1}\Bigl)=\exp(-Z_{t_1}u(t_2-t_1,\lambda))$.\\

(ii) \begin{align}
\label{expinc}Q_{X_0}\Bigl(\Bigl(e^{-\lambda Z_{t_2}}-e^{-\lambda Z_{t_1}}\Bigr)^2\Bigr)
=Q&_{X_0}\Bigl(\exp(-Z_{t_1} u(t_2-t_1,2\lambda))\\
\nn&-2\exp(-\lambda Z_{t_1}-Z_{t_1}u(t_2-t_1,\lambda))+\exp(-2\lambda Z_{t_1})\Bigr),
\end{align}
and similarly for $\N_{X_0}$.\\

\noindent (b) For all $t>0$, $\lambda\mapsto u(t,\lambda)$ is completely concave.\\

\noindent(c) $(Z_t,t\ge 0)$ is a (time-homogeneous) $(\cG_t)$-Markov process under $Q_{X_0}$ or $\N_{X_0}$.
\end{lemma}
\begin{proof} (a) \eqref{Qcond} shows that for $\lambda\ge 0$, the left-hand side of (i) equals the middle expression, which by Proposition~\ref{spmarkov}(a)(ii) and then \eqref{LTexit} equals
\begin{align*}
\E_{X_{G_{r_0e^{-t_1}}}}(\exp(-\lambda e^{2t_2}r_0^{-2}X_{G_{r_0e^{-t_2}}}(1)))
&=\exp\Bigl(-\int U^{\lambda e^{2t_2}r_0^{-2},r_0e^{-t_2}}(x)X_{G_{r_0e^{-t_1}}}(dx)\Bigr)\\
&=\exp\Bigl(-U^{\lambda e^{2t_2}r_0^{-2},r_0e^{-t_2}}(r_0 e^{-t_1})X_{G_{r_0e^{-t_1}}}(1)\Bigr)\\
&=\exp(-u(t_2-t_1,\lambda)Z_{t_1}),
\end{align*}
where scaling (i.e., \eqref{scaling1}) is used in the last line.  This gives (i).  It is then easy to derive (ii) by expanding out the square, conditioning on $\cG_{t_1}$ and finally using (i).\\
(b) Let $y_0>0$ and $t>0$. Let $m_r$ be the uniform distribution on $\{|x|=r\}$ and set $W=e^{2t}r_0^{-2}X_{G_{r_0 e^{-t}}}(1)$. Apply \eqref{LTexit} and then scaling (\eqref{scaling1}) to see that for all $\lambda\ge 0$,
\begin{align*}\E_{y_0r_0^2m_{r_0}}(\exp(-\lambda W))&=\exp(-y_0r^2_0 U^{\lambda e^{2t}r_0^{-2},r_0e^{-t}}(r_0))\\
&=\exp(-y_0r_0^2r_0^{-2}e^{2t}U^{\lambda,1}(e^t))\\
&=\exp(-y_0 u(t,\lambda)).
\end{align*}

\noindent(c)  This is immediate from (a)(i), (b) (to define the family of laws $\{P_x:x\ge 0\}$), and a monotone class argument.
\end{proof}

\begin{proposition}\label{prop:yzinfo}
(a) $Y$ is an inhomogeneous $(\cE_r)$-Markov process under $\N_{X_0}$ or $Q_{X_0}$. That is, for $\psi:M_F(\R^d)\to[0,\infty)$ Borel measurable and $0\le r_1<r_2$, 
\[Q_{X_0}(\psi(Y(r_2))|\cE_{r_1})=\N_{X_0}(\psi(Y(r_2))|\cE_{r_1})=\E_{Y(r_1)}(\psi(Y(r_2)))\ \ \text{a.s.}\]
(b) If $0\le r_1<r_2<r_0$, then the total mass, $Y_r(1)$, of $Y_r$ satisfies
\begin{align}\label{Ymart}\N_{X_0}(Y_{r_2}(1)|\cE_{r_1})=\begin{cases}Y_{r_1}(1)&\text{ if }d=2\\
\frac{r_0-r_2}{r_0-r_1} Y_{r_1}(1)&\text{ if }d=3.
\end{cases}
\end{align}
Under $\N_{X_0}$ or $Q_{X_0}$, $Y_r(1)$ has a cadlag version on $[0,r_0)$ which is an $(\cE^+_{r})$-supermartingale (an $(\cE^+_{r})$-martingale if $d=2$), satisfies \eqref{Ymart} with $\cE^+_{r_1}$ in place of $\cE_{r_1}$,  and has only non-negative jumps a.e. \\
(c) Under $\N_{X_0}$ or $Q_{X_0}$, $Z(t),t\ge 0$ has a cadlag version which is a CSBP with log Laplace transform given by $\{u(t,\cdot):t>0\}$ in \eqref{udef}.
\end{proposition}
\begin{proof} (a) This is immediate from Proposition~\ref{spmarkov}(a)(ii) and \eqref{Qcond}.\\
(b,c) Let $B$ denote a $d$-dimensional Brownian motion starting at $x$ under $P^B_x$ and 
\[\tau_r=\inf\{t\ge 0:|B_t|\le r\}\quad(\inf\emptyset=\infty).\]
Recalling \eqref{r0cond}, Proposition~3 in Chapter V of \cite{Leg99} shows that for $0\le r<r_0$, 
\begin{equation}\label{meanmass}
\N_{X_0}(Y_r(1))=\int P_x^B(\tau_{r_0-r}<\infty)dX_0(x)=\begin{cases}X_0(1)&\text{ if }d=2\\
\int\frac{r_0-r}{|x|}dX_0(x)&\text{ if }d=3.
\end{cases}
\end{equation}
Return now to the probability $\P_{X_0}$, and use \eqref{exitdecomp} and the above to see that 
\begin{equation}\label{meanmass2}
\E_{X_0}(Y_r(1))=\N_{X_0}(Y_r(1))=\begin{cases}X_0(1)&\text{ if }d=2\\
\int\frac{r_0-r}{|x|}dX_0(x)&\text{ if }d=3.
\end{cases}
\end{equation}
Although we have assumed $B_{2r_0}\subset \text{Supp}(X_0)^c$, both \eqref{meanmass} and \eqref{meanmass2} will apply if $\text{Supp}(X_0)\subset G_{r_0-r}$.  This allows us to apply \eqref{meanmass2}, with $X_0=Y_{r_1}$ and $r=r_2$, and (a) to derive \eqref{Ymart}.

Turning to the second part of (b) and (c) we first work with $Z$.  Let $t_n\uparrow t>0$ ($t_n<t$) and set $r_n=r_0(1-e^{-t_n})\uparrow r_0(1-e^{-t})=r\in(0,r_0)$.  By \eqref{Ymart} and supermartingale convergence, $\{Y_{r_n}\}$ converges $\N_{X_0}$-a.e. to a limit we denote by $Y_{r-}(1)$ for now.  (The $\sigma$-finiteness of $\N_{X_0}$ is not an issue here, but the reader who prefers probabilities may work with $Q_{X_0}$ and note that on the complementary set, $\{Y_0(1)=0\}$,
$Y_{r_n}(1)=0$ $\N_{X_0}$-a.e. by \eqref{Ymart} with $r_1=0$ there.  Henceforth we will not make such arguments.) It follows that 
\begin{equation}\label{Zlag}
Z_{t_n}\to e^{2t} r_0^{-2}Y_{r-}(1):=Z_{t-}\quad \N_{X_0}-\text{a.e.}
\end{equation}
By \eqref{expinc},
\begin{align*}
Q&_{X_0}\Bigl(\Bigl(e^{-\lambda Z_t}-e^{-\lambda Z_{t_n}}\Bigr)^2\Bigr)\\
&=Q_{X_0}\Bigl(\exp(-Z_{t_n} u(t-t_n,2\lambda))-2\exp(-(\lambda+u(t-t_n,\lambda))Z_{t_n})\\
&\phantom{=Q_{X_0}(\exp(-Z_{t_n} u(t-t_n,2\lambda))\ }+\exp(-2\lambda Z_{t_n})\Bigr)\\
&\to Q_{X_0}(\exp(-2\lambda Z_{t-})-2\exp(-2\lambda Z_{t-})+\exp(-2\lambda Z_{t-}))\quad\text{as }n\to\infty\\
&=0,
\end{align*}
where Dominated Convergence is used in the above convergence. This and \eqref{Zlag} show that $Z_{t_n}\to Z_t$ $Q_{X_0}$-a.s.  The fact, noted above, that $Y_0(1)=0$ implies $Z_{t_n}=Z_t=0$ $\N_{X_0}$-a.e. allows us to upgrade this to
\begin{equation}\label{Zcag2}
Z_{t_n}\to Z_t\quad \N_{X_0}-a.e.\quad\text{if $t_n\uparrow t>0$. }
\end{equation}
A simpler argument, now using reverse supermartingale convergence, shows that 
\begin{equation}\label{Zcad}
Z_{t_n}\to Z_t\quad \N_{X_0}-a.e.\quad\text{if $t_n\downarrow t\ge 0$.}
\end{equation}
\eqref{Zcag2} and \eqref{Zcad} imply $Y_r(1)$ is continuous in measure on $[0,r_0)$.
Therefore by \eqref{Ymart} there is a cadlag version of $(Y_r(1),r\in[0,r_0))$ under $\N_{X_0}$  (we do not change the notation) which is an $(\cE^+_r)$-supermartingale (martingale if $d=2)$ satisfying \eqref{Ymart}  with $\cE^+_{r_1}$ in place of $\cE_{r_1}$. This gives a cadlag version of $Z$ which satisfies the $(\cG_{t+})$ version of Lemma~\ref{lem:Zstuff}(a)(i), and so is $(\cG_{t+})$-Markov under $\N_{X_0}$ or $Q_{X_0}$, just as for Lemma~\ref{lem:Zstuff}(c).  Clearly \eqref{Zcag2} and \eqref{Zcad} imply that $Z_{t-}=Z_t$ $\ \N_{X_0}$-a.e., and so $Z$ has no fixed time discontinuities.  It follows
from the above and Lemma~\ref{lem:Zstuff}(b) that $(Z_t,t\ge 0)$ is a (CSBP) with log Laplace transform $\{u(t,\cdot):t>0\}$ under $\N_{X_0}$ or $Q_{X_0}$.  A theorem of Lamperti (see. e.g. p. 1044 of \cite{Si68}) shows that $(Z_t,t\ge 0)$ has only non-negative jumps a.e. and so the same applies to $(Y_r(1),r\in[0,r_0])$.
\end{proof}

\begin{remark}\label{rem:CSBP} Although in this work we only use the above results, we briefly
discuss the processes $Z_\cdot$ and $Y_\cdot(1)$ in the general context of CSBP's.    By Proposition~\ref{prop:yzinfo}(c) above and Theorem~4 of \cite{Si68} there is a L\'evy measure $\tilde\pi$ on $[0,\infty)$ satisfying 
$\int \ell^2\wedge 1\,d\tilde\pi(\ell)<\infty$ and constants $\tilde a\in\R, b\ge 0$, such that if 
\begin{equation}\label{psi1}
\Psi(u)=\tilde a u-bu^2+
\int_0^\infty (1-e^{-u\ell}-u\ell e^{-\ell})\,d\tilde\pi(\ell),\ \ u\ge 0,
\end{equation}
then $t\mapsto u(t,\lambda)$ is the unique solution of 
\begin{equation}\label{ode1}
\frac{du(t,\lambda)}{dt}=\Psi(u(t,\lambda)),\quad u(0,\lambda)=\lambda.
\end{equation}
$Z$ is often called a $\Psi$-CSBP. \eqref{psi1} implies $\Psi$ is concave on $[0,\infty)$ and differentiable on $(0,\infty)$.  If $\lambda_d=2(4-d)$, then a short calculation using \eqref{udef} and \eqref{UPDE} gives (primes denote derivatives with respect to $t$)
\begin{equation}\label{ode2}
u''(t,\lambda)=(6-d)u'+u(u-\lambda_d),\ t\ge 0.
\end{equation}
Differentiating both sides of \eqref{ode1} and using \eqref{ode2} on the resulting left-hand side, leads to the first order ode for $\Psi$,
\[\Psi'\Psi(u(t,\lambda))=(6-d)\Psi(u(t,\lambda))+u(t,\lambda)(u(t,\lambda)-\lambda_d),\quad \Psi(0)=0.\]
Letting $t\to 0$ and varying $\lambda$ we conclude that $\Psi$ is a solution of the ode
\begin{equation}\label{ode3}
\Psi'\Psi(u)=(6-d)\Psi(u)+u(u-\lambda_d),\ u>0,\quad \Psi(0)=0.
\end{equation}

By using this equation to analyze the behaviour of $\Psi$ near $\infty$ it is easy to see that in \eqref{psi1}, $b=0$.  The concavity of $\Psi$ implies $\lim_{u\downarrow 0}\frac{\Psi(u)}{u}=\lim_{u\downarrow 0}\Psi'(u)\in(-\infty,+\infty]$. If we divide both sides of \eqref{ode3} by $u$ and let $u\downarrow 0$ we conclude this limit, $\Psi'(0)$ is in fact finite and satisfies
\[\Psi'(0)^2=(6-d)\Psi'(0)-\lambda_d,\]
that is, $\Psi'(0)=2$ if $d=2$, and $\Psi'(0)=1$ or $2$ if $d=3$.
It is not hard to see using~\eqref{Ymart} that, in fact, $\Psi'(0)=1$ if $d=3$. 
 The fact that this derivative is finite, already implies that $\int_0^\infty  \ell d\tilde\pi(\ell)<\infty$ and \eqref{psi1} can be rewritten as 
\begin{equation}\label{psi2}
\Psi(u)=a_d u+\int_0^\infty(1-e^{-u\ell}-u\ell)d\pi(\ell),\quad \int_0^\infty \ell\wedge\ell^2 d\pi(\ell)<\infty,
\end{equation}
where now $a_d=
\Psi'(0)=4-d$, by the above.  The ode \eqref{ode3} can be used to study the tail behaviour of $\Psi$, and hence $\pi$, via Tauberian theorems. For example it is not hard to show that for some explicit $c_{\ref{pilt}}>0$,
\begin{equation}\label{pilt}
\lim_{\veps\downarrow 0}\veps^{3/2}\pi([\veps,\infty))=c_{\ref{pilt}}.
\end{equation}

The process of total mass of the exit measure {\it from $B_r$} (as opposed to $G_{r_0-r}$) is studied  in \cite{Kyp14} as an inhomogeneous CSBP. The setting there is for general branching mechanisms, but the ideas used above and in defining $Z$  appear to be novel. It would be of interest to study the detailed behaviour of the measure-valued process $r\to X_{G_{r_0-r}}$.  

In \cite{MP17} we instead worked with the exit measure from half spaces $H_r=\{x:x_1<r\}$, where the total mass process is a $\Psi$-CSBP with $\Psi(u)=\frac{\sqrt 6}{3}u^{3/2}$ (see \cite{Kyp14} and Proposition~4.1 of \cite{MP17} and c.f. \eqref{pilt}). The CSBP analysis there was simpler due to this explicit $3/2$-stable $\Psi$, but half-planes were clumsier and led to less precise results. See the discussion at the end of  the Introduction.
\end{remark}


\medskip

\section{Proof of Propositions~\ref{exitdim}, \ref{candim}}\label{sec:maintheorem}

We  use the notation from Section~\ref{sec:6}. In particular $X_0$ and $r_0>0$ are 
as in \eqref{r0cond}, $Y_r=X_{G_{r_0-r}}$ for $0\le r<r_0$, and $Q_{X_0}(\cdot)=\N_{X_0}(\cdot|Y_0(1)>0)$.


In what follows we always will work with the cadlag versions of $Y_r(1)$, and hence $Z_t$, constructed in Proposition~\ref{prop:yzinfo}(b) above.  
We let $W$ denote a generic snake under $\N_{X_0}$ or $Q_{X_0}$ with the associated
``tip process" $\hat W(t)$ and excursion length $\sigma$. Define
\[T_0(W)=\inf\{r\in[0,r_0):Y_r(1)=0\}\in[0,r_0],\text{ where }\inf\emptyset=r_0,\]
and 
\begin{equation}\label{hatT0def}\hat T_0(W)=\inf\{|\hat W(t)|:0\le t\le \sigma\}=\inf\{|x|:x\in\mR\},
\end{equation}
the final equality holding $\N_{X_0}$-a.e. by \eqref{snakerange}. Clearly we have 
\[Q_{X_0}(\cdot)=\N_{X_0}(\cdot\, |T_0>0).\] 

\begin{lemma}\label{lem:T0s}
The sets $\{T_0>0\}$ and $\{\hat T_0<r_0\}$ coincide $\N_{X_0}$-a.e., and on this set,
$\hat T_0=r_0-T_0$ $\N_{X_0}$-a.e.
\end{lemma}
\begin{proof}  For every rational $q$ in $[0,T_0)$, $X_{G_{r_0-q}}(1)>0$ implies $\partial G_{r_0-q}\cap\mR$ is non-empty (by \eqref{exitsupport}) and so by \eqref{snakerange} $\hat T_0\le r_0-q$.  This proves that
\begin{equation}\label{hatub}
\hat T_0\le r_0-T_0\quad\N_{X_0}\text{-a.e. on }\{T_0>0\}.
\end{equation}
Conversely assume $r_0>T_0$ and choose rationals $q,q'$ so that $T_0<q'<q<r_0$.
Then $X_{G_{r_0-q'}}(1)=0$ and the special Markov property (Proposition~\ref{spmarkov}(b)) at $R_1=r_0-q'$ shows that $\N_{X_0}(\mR\cap B_{r_0-q}\neq \emptyset|\cE_{q'})=0$ a.e on  $\{T_0<q'\}$.  This proves 
that 
\begin{equation}\label{hatlb}
\hat T_0\ge r_0-T_0\quad\N_{X_0}\text{-a.e. on }\{T_0<r_0\}.
\end{equation}
The above is trivial if $T_0=r_0$ and so we have shown (by \eqref{hatub} and \eqref{hatlb})
\begin{equation*}
\hat T_0=r_0-T_0\quad \N_{X_0}\text{-a.e on }\{T_0>0\}.
\end{equation*}
Finally, note that \eqref{hatub} shows $T_0>0$ implies $\hat T_0<r_0$,  and 
\eqref{hatlb} shows $T_0=0$ implies $\hat T_0\ge r_0$, which in turn shows $\hat T_0< r_0$ 
implies $T_0>0$ (all up to $\N_{X_0}$ null sets).  This proves the a.e. equality of $\{T_0>0\}$ and $\{\hat T_0<r_0\}$, and completes the proof. 
\end{proof}

\begin{lemma}\label{lem:abscont}(a) For $0<r<r_0$,
\begin{align*}\N_{X_0}(0<T_0\le r)&=\N_{X_0}(r_0-r\le \hat T_0<r_0)\\
&=\N_{X_0}\Bigl(1(X_{G_{r_0}}(1)>0)\exp[-X_{G_{r_0}}(1)(r_0-r)^2U^{\infty,1}(r_0/(r_0-r))]\Bigr).
\end{align*}
(b) $\N_{X_0}(T_0\in dr)\ll dr$ on $\{0<r<r_0\}$ and $\N_{X_0}(\hat T_0\in dr)\ll dr$ on $\{0<r<2r_0\}$.
\end{lemma}
\begin{proof} (a) Using \eqref{LTexit} and  scaling (\eqref{scaling1} with $\lambda=\infty$), we have for $0<r<r_0$,
\begin{align}\label{infLT} 
\nn \P_{X_{G_{r_0}}}(X_{G_{r_0-r}}(1)=0)&
=\exp(-X_{G_{r_0}}(U^{\infty, r_0-r}))\\
&=\exp(-X_{G_{r_0}}(1)(r_0-r)^{-2}U^{\infty,1}(r_0/(r_0-r))).
\end{align}
The special Markov property (Proposition~\ref{spmarkov}(a)(ii)) now implies for $0<r<r_0$,
\begin{align*}
\N_{X_0}(0<T_0\le r)&=\N_{X_0}(1(T_0>0)\N_{X_0}(X_{G_{r_0-r}}(1)=0|\cE_0))\\
&=\N_{X_0}(1(T_0>0)\P_{X_{G_{r_0}}}(X_{G_{r_0-r}}(1)=0))\\
&=\N_{X_0}\Bigl(1(X_{G_{r_0}}(1)>0)\exp[-X_{G_{r_0}}(1)(r_0-r)^{-2}U^{\infty,1}(r_0/(r_0-r))]\Bigr),
\end{align*}
where \eqref{infLT} has been used in the last line.
This, together with Lemma~\ref{lem:T0s}, gives (a).

\noindent(b)  The right-hand side of (a) is continuously differentiable in $r\in(0,r_0)$ because $U^{\infty,1}$ is $C^2$ on $G_1$ (recall \eqref{Uinftyprop}). Here we note that it is easy to justify differentiation inside the integral since $\N_{X_0}(X_{G_{r_0}}(1)>0)<\infty$ (recall \eqref{r0cond}), $\N_{X_0}(X_{G_{r_0}}(1))<\infty$ (recall \eqref{meanmass}), and $(U^{\infty,1})'(r)$ is bounded on compacts away from $\{r\le 1\}$. This gives the first part of (b).  Lemma~\ref{lem:T0s} now implies the absolute continuity of $\N_{X_0}(\hat T_0\in dr)$ on $\{0<r<r_0\}$.  But \eqref{r0cond} allows us to replace $r_0$ with $\alpha r_0$ for any $1<\alpha<2$ in the above reasoning and so conclude that $\N_{X_0}(\hat T_0\in dr)$ is absolutely continuous on $\{0<r<2r_0\}$.
\end{proof}
\SUBMIT{
\no{\bf Proof of Proposition~\ref{exitdim} assuming Proposition~\ref{candim}.} 
By translation invariance we may assume $x_1=0$. Fix $r_0,r_1$ and $X_0$ as in our hypotheses.
 We must show that
\begin{equation}\label{goal}
X_{G_{r_1}}(1)=0\text{ and }X_{G_{r_0}}(1)>0\text{ imply dim}(B_{r_0}\cap\pmR)\ge d_f\ \P_{X_0}-a.s.
\end{equation}
Measurability issues are easily handled using Lemma~\ref{borel} and will henceforth be ignored.
We work under $\P_{X_0}$ in the standard set-up and so from \eqref{exitdecomp} have for $0<r\le r_0$ and
$J_0=\{j\in J:\hat T_0(W_j)\le r_0\}$,
\begin{equation}\label{XGrdec}
X_{G_r}=\sum_{j\in J}X_{G_r}(W_j)=\sum_{j\in J}X_{G_r}(W_j)1(\hat T_0(W_j)\le r)=\sum_{j\in J_0}X_{G_r}(W_j)1(\hat T_0(W_j)\le r).
\end{equation}
Here we used the fact that $\hat T_0(W_j)>r$ implies $X_{G_r}(W_j)=0$ (e.g. by  \eqref{exitsupport} and \eqref{snakerange}). 

It follows easily from Lemma~\ref{lem:T0s} that
\begin{equation}\label{T0hatT0}
T_0=r_0-\wedge_{j\in J_0}\hat T_0(W_j)\quad \text{ on }\{T_0>0\}=\{\wedge_{j\in J_0}\hat T_0(W_j)<r_0\}\ \P_{X_0}-\text{a.s.}
\end{equation}
In view of the absolute continuity properties of $\hat T_0$ under $\N_0$ from Lemma~\ref{lem:abscont} we see from the above that if $N_0=|J_0|$, a Poisson mean $N_{X_0}(\hat T_0\le r_0)$ random variable, then 
\begin{equation}\label{T0N0}
T_0>0\text{ iff }J_0\neq\emptyset\text{ iff }N_0>0\ \ \P_{X_0}-\text{a.s.}
\end{equation}
By enlarging our probability space and introducing appropriate randomization 
 we may
assume that there is an iid sequence $\{\widetilde W_j:j\in \N\}$, independent of the Poisson variable $N_0=|J_0|$ with mean $\N_{X_0}(\hat T_0\le r_0)$, and with common law 
\begin{equation}\label{tWlaw}
\N_{X_0}(\cdot\,|\hat T_0\le r_0)=\N_{X_0}(\cdot\,|\hat T_0<r_0),
\end{equation}
(the last equality by Lemma~\ref{lem:abscont}) and so that
\begin{equation}\label{pppid}
\sum_{j\in J_0}\delta_{W_j}=\sum_{j=1}^{N_0}\delta_{\widetilde W_j}.
\end{equation}
Now using Lemma~\ref{lem:abscont} and independence of excursions $\widetilde W_j$ one can show that given $
X_{G_{r_1}}(1)=0\text{ and }X_{G_{r_0}}(1)>0$, $\N_{X_0}$-a.e. there exists 
$$
\tilde T_0 =\textnormal{largest radius $r\leq r_0$ so that a single excursion $\widetilde W_j$ enters $B_r$}\in(\hat T_0,r_0]\subset(r_1,r_0]. 
$$
Thus it is easy to see that the lower bound on $\text{dim}(\pmR\cap B_{r_0})$ in \eqref{goal} can be given by the lower bound on corresponding dimension for a {\it single} excursion.  
To be more precise we claim it is enough to get 
\begin{equation}
\label{18_07_1}
\N_{X_0}(\text{dim}(\pmR\cap B_r)\ge d_f\ \forall r>\hat T_0|r_1\leq \hat T_0<r_0)=1.
\end{equation}
Clearly \eqref{18_07_1} follows from  Proposition~\ref{candim}.  
We omit the details of the careful derivation of~\eqref{goal}  from the above as it closely follows the derivation of Theorem~7.1 from Proposition~7.2 in Section 7.1 of \cite{MP17} (see \cite{HMP18A}).  
\qed}

 
\ARXIV{
\ARXIV{
\no{\bf Proof of Proposition~\ref{exitdim} assuming Proposition~\ref{candim}.}}
By translation invariance we may assume $x_1=0$. Fix $r_0,r_1$ and $X_0$ as in our hypotheses. We must show that
\begin{equation}\label{goal}
X_{G_{r_1}}(1)=0\text{ and }X_{G_{r_0}}(1)>0\text{ imply dim}(B_{r_0}\cap\pmR)\ge d_f\ \P_{X_0}-a.s.
\end{equation}
Measurability issues are easily handled using Lemma~\ref{borel} and will henceforth be ignored.
We work under $\P_{X_0}$ in the standard set-up and so from \eqref{exitdecomp} have for $0<r\le r_0$ and
$J_0=\{j\in J:\hat T_0(W_j)\le r_0\}$,
\begin{equation}\label{XGrdec}
X_{G_r}=\sum_{j\in J}X_{G_r}(W_j)=\sum_{j\in J}X_{G_r}(W_j)1(\hat T_0(W_j)\le r)=\sum_{j\in J_0}X_{G_r}(W_j)1(\hat T_0(W_j)\le r).
\end{equation}
Here we used the fact that $\hat T_0(W_j)>r$ implies $X_{G_r}(W_j)=0$ (e.g. by  \eqref{exitsupport} and \eqref{snakerange}). Recall from \eqref{snakerange} that the range of the $j$th excursion $W_j$ is
\[\mR_j:=\mR(W_j)=\{\hat W_j(t):t\le \sigma(W_j)\}.\]
It follows easily from \eqref{Xtdec} (see (2.19) in \cite{MP17}) that for $x\in \overline{B_{r_0}}$,
\[L^x=\sum_{j\in J_0}L^x(W_j),\]
and therefore,
\begin{equation}\label{rangedec}
\mR\cap B_{r_0}=\cup_{j\in J_{0}} (\mR_j\cap B_{r_0})\text{ and so }1_{\mR\cap B_{r_0}}(x)=1\Bigl(\sum_{j\in J_0}1_{\mR(W_j)\cap B_{r_0}}(x)>0\Bigr).
\end{equation}
We will frequently use the elementary topological result
\begin{equation}\label{top}
B_{r_0}\cap \partial F=B_{r_0}\cap\partial(B_{r_0}\cap F)=B_{r_0}\cap\partial(\overline{B_{r_0}}\cap F)\quad\text{for any closed set }F.
\end{equation}
It follows easily from Lemma~\ref{lem:T0s} that
\begin{equation}\label{T0hatT0}
T_0=r_0-\wedge_{j\in J_0}\hat T_0(W_j)\quad \text{ on }\{T_0>0\}=\{\wedge_{j\in J_0}\hat T_0(W_j)<r_0\}\ \P_{X_0}-\text{a.s.}
\end{equation}
In view of the absolute continuity properties of $\hat T_0$ under $\N_0$ from Lemma~\ref{lem:abscont} we see from the above that if $N_0=|J_0|$, a Poisson mean $\N_{X_0}(\hat T_0\le r_0)$ random variable, then 
\begin{equation}\label{T0N0}
T_0>0\text{ iff }J_0\neq\emptyset\text{ iff }N_0>0\ \ \P_{X_0}-\text{a.s.}
\end{equation}
By enlarging our probability space and randomizing the above Poisson points we may
assume that there is an iid sequence $\{\widetilde W_j:j\in \N\}$, independent of the Poisson variable $N_0=|J_0|$ with mean $\N_{X_0}(\hat T_0\le r_0)$, and with common law 
\begin{equation}\label{tWlaw}
\N_{X_0}(\cdot\,|\hat T_0\le r_0)=\N_{X_0}(\cdot\,|\hat T_0<r_0),
\end{equation}
(the last equality by Lemma~\ref{lem:abscont}) and so that
\begin{equation}\label{pppid}
\sum_{j\in J_0}\delta_{W_j}=\sum_{j=1}^{N_0}\delta_{\widetilde W_j}.
\end{equation}
Let $\widehat{\widetilde W}_j$ denote the tip of the $j$th excursion and define
\begin{align*} 
\hat T^j&=\hat T_0(\widetilde W_j)<r_0 \quad(\text{a.s.  by \eqref{tWlaw}}),\\
\tilde \mR_j&=\mR(\widetilde W_j)=\{\widehat {\widetilde W}_j(t):t\le \sigma(\widetilde W_j)\}.
\end{align*}

 Note that $X_{G_{r_1}}(1)=0$ implies $T_0<r_0-r_1$ a.s. and so, in view of \eqref{T0hatT0},  
\begin{equation}\label{hatTlb}
X_{G_{r_1}}(1)=0\text{ and }Y_0(1)=X_{G_{r_0}}(1)>0\text{ imply }\hat T^j\ge r_1\  \forall\ j\le N_0
\ \P_{X_0}-\text{a.s.}
\end{equation}
 The independence of the $\hat T^j$'s and 
fact they have no positive atoms by Lemma~\ref{lem:abscont} imply
\begin{equation}\label{noteqT}\P_{X_0}(\hat T^j=\hat T^{j'} \text{ for some }1\le j\neq j' \le N_0, X_{G_{r_1}}(1)=0,X_{G_{r_0}}(1)>0)=0.\end{equation}
So on $\{X_{G_{r_1}}(1)=0,X_{G_{r_0}}(1)>0\}$ there is an a.s. unique $\tilde j\le N_0$ s.t. \break $\hat T^{\tilde j}=\min\{\hat T^j:j\le N_0\}$. 
\eqref{noteqT} and \eqref{hatTlb} imply that (if an empty minimum is $r_0$) $\P_{X_0}-\text{a.s.}$,
\begin{equation}\label{tildetpos}
\tilde T:=\min\{\hat T^j:j\neq\tilde j,j\le N_0\}>\hat T^{\tilde j}\ge r_1\text{ on }\{X_{G_{r_1}}(1)=0, X_{G_{r_0}}(1)>0\}\subset\{N_0\ge 1\}.
\end{equation}
Hence $\tilde T$ is the largest radius $r\le r_0$ so that a single excursion $\widetilde W_j$ enters $B_r$ (it exists on $\{X_{G_{r_1}}(1)=0, X_{G_{r_0}}(1)>0\}$). 

By the definition of $\tilde j$ and $\tilde T$ we have from \eqref{tildetpos} and \eqref{rangedec},
\[ B_{\tilde T}\cap\mR= B_{\tilde T}\cap\tilde \mR_{\tilde j}\text{ on }\{X_{G_{r_1}}(1)=0,X_{G_{r_0}}(1)>0\}\ \text{a.s.}
\]
Therefore, using the above and \eqref{top} we obtain
\begin{align}\label{lb1}
\nn \P_{X_0}&(\text{dim}(B_{r_0}\cap\pmR)\ge d_f, X_{G_{r_1}}(1)=0,X_{G_{r_0}}(1)>0)\\
\nn&\ge \P_{X_0}(\text{dim}(B_{\tilde T}\cap\pmR)\ge d_f, X_{G_{r_1}}(1)=0,X_{G_{r_0}}(1)>0)\\
\nn&= \P_{X_0}(\text{dim}(B_{\tilde T}\cap\partial \tilde\mR_{\tilde j})\ge d_f, X_{G_{r_1}}(1)=0,X_{G_{r_0}}(1)>0)\\
&\ge \P_{X_0}(\{N_0\ge 1\}\cap(\cap_{j\le N_0}\{\text{dim}(B_r\cap\partial\tilde \mR_j)\ge d_f\ \forall r> \hat T^j\ge r_1\})).
\end{align}
In the last line we have used $N_0\ge 1$ iff $X_{G_{r_0}}(1)>0$ (by \eqref{T0N0}), and on this set, $\hat T^j\ge r_1$ for all $j\le N_0$ implies $T_0\le r_0-r_1$ (by \eqref{T0hatT0}) and so $X_{G_{r_1}}(1)=0$.  We also
 use the fact (from \eqref{tildetpos}) that if $\{0<T_0\le  r_0-r_1\}$ then $\tilde T> \hat T^{\tilde j}$.
The independence of the $\tilde W_j$'s and their joint independence from $N_0$ together with their common law in \eqref{tWlaw} imply that \eqref{lb1} equals
\begin{equation}\label{lb2}
\E_{X_0}(1(N_0\ge 1)\prod_{j=1}^{N_0}\N_{X_0}(\text{dim}(\pmR\cap B_r)\ge d_f\ \forall r>\hat T_0\ge r_1|\hat T_0<r_0)).\end{equation}
By Proposition~\ref{candim} and Lemma~\ref{lem:T0s} each of the terms in the above product equals
$\N_{X_0}(\hat T_0\ge r_1|\hat T_0<r_0)$ and so \eqref{lb2} equals
\begin{align}\label{lb3}
\nn \E_{X_0}&(1(N_0\ge 1)1(\wedge_{j=1}^{N_0}\hat T_0(\tilde W_j)\ge r_1))\\
&=\P_{X_0}(0<T_0\le r_0-r_1)=\P_{X_0}(X_{G_{r_1}}(1)=0, X_{G_{r_0}}(1)>0).
\end{align}
In the first equality we used \eqref{T0hatT0} and \eqref{T0N0}. We have proved the left-hand side of \eqref{lb1} exceeds the above, and we conclude that 
\[X_{G_{r_1}}(1)=0\text{ and } X_{G_{r_0}}(1)>0\text{ imply dim}(B_{r_0}\cap \pmR)\ge d_f\ \P_{X_0}-\text{a.s.},\]
thus proving \eqref{goal}.
\qed}

Recall again that we always work with the cadlag version of $Y_r(1)$ from Proposition~\ref{prop:yzinfo}(b) which only has non-negative jumps and is an $(\cE^+_r)$-supermartingale.  
Define a sequence of $(\cE^+_r)$-stopping times by
\[T_{n^{-1}}=\inf\{r\le r_0:Y_r(1)\le 1/n\}\quad(\inf\emptyset = r_0).\]
Then 
\begin{equation}\label{announce}
\text{on $\{0<T_0\}$ (and so $Q_{X_0}$-a.s.) $T_{n^{-1}}\uparrow T_0$ and $T_{n^{-1}}<T_0$, }
\end{equation}
where the last inequality holds since $Y_r(1)$ has no negative jumps.  So under $Q_{X_0}$, $T_0$ is a predictable stopping time which is announced by $\{T_{n^{-1}}\}$ and so (see (12.9)(ii) in Chapter VI of \cite{RW94})
\begin{equation}\label{E-}
\cE^+_{T_0-}=\vee_n\cE^+_{T_n}.
\end{equation}
Let $D_r=\{\text{dim}(B_r\cap\pmR)\ge d_f\}$ for $0<r\le r_0$. We assume $\cE^+_r$ is augmented by $Q_{X_0}$-null sets throughout this Section. 

To finish the proof of  Proposition~\ref{candim} we need: 
\begin{lemma}\label{Emeas} If $X_0=\delta_{x_0}$ where $|x_0|\ge 2r_0$, then
\begin{equation}
\label{17_08_2}
D_{r_0}\in\cE^+_{T_0-}.
\end{equation}
\end{lemma}
For the proof of Proposition~\ref{candim} below it would suffice to show that $D_{r_0}\cap\{T_0<r_0\}\in\cE^+_{T_0-}$, and this latter result
 should be intuitively obvious, as we now explain.  With Lemma~\ref{lem:T0s} in mind, we see that $\cE^+_{T_0-}$ includes information generated by the excursions of $W$ outside of its minimum radius. If this minimum radius is positive (as is the case on $\{T_0<r_0\}$) it is intuitively clear that  this includes all the information generated by $W$.  Even without intersecting with $\{T_0<r_0\}$, however, none of the mass that hits the origin will survive for any length of time and so again all of $W$ will have been observed. This last point  stems from the fact that points are polar for Brownian motion in more than one dimension and be more formally justified using a mean measure result for the integral of the snake (Proposition 2 in Ch. IV of \cite{Leg99} with $p=1$). \ARXIV{Before giving its proof below, we first show how Lemma~\ref{Emeas} implies Proposition~\ref{candim}.}\\
 
\SUBMIT{
\noindent{\bf Proof of Lemma~\ref{Emeas}.}
Following  the derivation of (7.11) in~\cite{MP17}, one sees that in order to get~\eqref{17_08_2}, it  is enough to show (cf. (7.18) of \cite{MP17})
\begin{align}\label{notimeatedge17}
\N&_{x_0}\Bigl(\int_0^\infty1(\inf_{v\le \zeta_u}|W_u(v)|=\hat T_0)du\Bigr)=0.
\end{align}
Next follow the proof of (7.18) in~\cite{MP17} using the historical process and its Palm measure formula, to bound the left-hand side of~\eqref{notimeatedge17} by (cf. (7.22) of \cite{MP17})
\[\int_0^\infty E^B_{x_0}\Bigl(\exp\Bigl(-\int_0^s\frac{2(4-d)}{|B_t-m_s|^2}dt\Bigr)\Bigr)\,ds .\]
Here  $B$ denotes a $d$-dimensional Brownian motion starting at $x_0$ under $P^B_{x_0}$ and $m_s=\inf_{s'\le s}|B_{s'}|$. 
A simple application of L\'evy's modulus for $B$ shows that $\int_0^s\frac{2(4-d)}{|B_t-m_s|^2}dt$  is infinite a.s. and
so proves~\eqref{notimeatedge17} as required.  More details may be found in  \cite{HMP18A}, where the actual definition of $\cE_r$ is even used.\qed}
\medskip

\noindent{\bf Proof of Proposition~\ref{candim}.} Clearly it suffices to fix $x_0\in\text{Supp}(X_0)$ and prove the result with $\N_{x_0}$  in place of $\N_{X_0}$. By translation invariance we may assume $x_1=0$, and so $|x_0|\ge 2r_0$.  
\SUBMIT{We will  use the following elementary topological result
\begin{equation}\label{top}
B_{r_0}\cap \partial F=B_{r_0}\cap\partial(B_{r_0}\cap F)=B_{r_0}\cap\partial(\overline{B_{r_0}}\cap F)\quad\text{for any closed set }F.
\end{equation}}
Fix $0<r_1<r_0$.   Assume $0\le r<r_0$ and $n\in\N$ is large enough so that $r+n^{-1}<r_0$.  By Lemma~\ref{borel}(b) there is a universally measurable map $\psi:\mK\to[0,1]$ such that 
\begin{align}\label{psiform}
\nn 1(D_{r_0-r-n^{-1}})&=1(\text{dim}(\partial(\overline{B}_{r_0-r-n^{-1}}\cap\mR)\cap B_{r_0-r-n^{-1}})\ge d_f)\quad(\text{by \eqref{top}})\\
&=\psi(\overline{B}_{r_0-r-n^{-1}}\cap\mR).
\end{align}

Recall that conditional expectations with respect to $\cE_r$, under $\N_{x_0}$ and $Q_{x_0}$, agree $Q_{x_0}$-a.s., and note that Proposition~\ref{spmarkov}(b) can be trivially extended to universally measurable maps.  Therefore up to $Q_{x_0}$-null sets, on 
$\{4n^{-2}\le Y_r(1)\le (r_0-r)^2\} (\in \cE_r)$ we have
\begin{align*} 
Q_{x_0}(D_{r_0}|\cE_r)&\ge Q_{x_0}(D_{r_0-r-n^{-1}}|\cE_r)\\
&=\P_{Y_r}(D_{r_0-r-n^{-1}})\quad(\text{by \eqref{psiform} and Proposition~\ref{spmarkov}(b)})\\
&\ge \P_{Y_r}\Bigl(\text{dim}\Bigl(\pmR\cap B_{r_0-r-(\sqrt{Y_r(1)}/2)}\Bigr)\ge d_f\Bigr)\\
&\ge q_{\ref{lem:spherelb}},
\end{align*}
where Lemma~\ref{lem:spherelb} and the assumed bounds on $Y_r(1)$ are used in the last inequality, and the assumed lower bound on $Y_r(1)$ is used in the next to last inequality. 
Let $n\to\infty$ and take limits from above in $r\in \Q_+$ (recall $Y_r(1)$ is cadlag) to conclude that 
\begin{equation}\label{Mlbnd}
M_r:=Q_{x_0}(D_{r_0}|\cE^+_{r})\ge q_{\ref{lem:spherelb}}\text{ on }\{0<Y_r(1)<(r_0-r)^2\}\ \forall r\in\Q\cap(0,r_0)\ Q_{x_0}-\text{a.s.}
\end{equation}
Here $M_r$ is a cadlag version of the bounded martingale on the left-hand side.  Using right-continuity one can strengthen \eqref{Mlbnd} to 
\begin{equation}\label{Mlbnd2}
M_r=Q_{x_0}(D_{r_0}|\cE^+_{r})\ge q_{\ref{lem:spherelb}}\text{ on }\{0<Y_r(1)<(r_0-r)^2\}\ \forall r\in(0,r_0)\ Q_{x_0}-\text{a.s.}
\end{equation}
On $\{0<T_0\le r_0-r_1\}$ we have from \eqref{announce} and the lack of negative jumps for $Y_r(1)$,
\begin{equation}\label{quadbound}
\text{for $n$ large, }T_{n^{-1}}\in(0,r_0-r_1)\text{ and }Y_{T_{n^{-1}}}(1)=n^{-1}<(r_0-T_{1/n})^2\ Q_{x_0}-\text{a.s.}
\end{equation}
By Corollary (17.10) in Chapter VI of \cite{RW94}, \eqref{Mlbnd2}, and \eqref{quadbound},  we have $Q_{x_0}$-a.s. on \\
$\{0<T_0\le r_0-r_1\}\in \cE^+_{T_0-}$,
\begin{equation}\label{lbcondexpD}
Q_{x_0}(D_{r_0}|\cE^+_{T_0-})=\lim_{n\to\infty}M(T_{n^{-1}})\ge q_{\ref{lem:spherelb}}.
\end{equation}
Multiplying the above by $1(\{0<T_0\le r_0-r_1\})$, we see from Lemma~\ref{Emeas} that
\[1(D_{r_0}\cap\{0<T_0\le r_0-r_1\})\ge q_{\ref{lem:spherelb}}1(\{0<T_0\le r_0-r_1\})\quad Q_{x_0}-\text{a.s.},\]
and therefore by Lemma~\ref{lem:T0s},
\[r_1\le \hat T_0<r_0\text{ implies } \text{dim}(B_{r_0}\cap\pmR)\ge d_f\quad Q_{x_0}-\text{a.s.}.\]
This remains true if we replace $r_0$ by any $r\in(r_1,r_0]$ since we still have $B_{2r}\subset \text{Supp}(X_0)^c$.  Therefore we may fix $\omega$ outside a $Q_{x_0}$-null set so that for any $r\in(r_1,r_0]\cap\Q$, $r_1\le \hat T_0<r$ implies dim$(B_r\cap\pmR)\ge d_f$. By monotonicity of the conclusion in $r$ this means that $\{r_1\le \hat T_0<r_0\}$ implies $\text{dim}(B_r\cap\pmR)\ge d_f$ for all $r>\hat T_0$. This gives Proposition~\ref{candim} under $Q_{x_0}$.  The result under $\N_{x_0}$ is now immediate from the definition of $Q_{x_0}$, and $\{Y_0(1)>0\}=\{\hat T_0<r_0\}$ $\N_{x_0}$-a.e. (by Lemma~\ref{lem:T0s}).  
\qed
\medskip

\ARXIV{
Turning now to Lemma~\ref{Emeas}, we work under $Q_{x_0}$ where $|x_0|\ge 2r_0$.  Recall the definitions of $\eta_s^G$ and $\cE_G$ from Section~\ref{sec:prel}.
For $0\le r<r_0$, introduce
\[A^r_t=\int_0^t1(\zeta_u\le S_{G_{r_0-r}}(W_u))\,du,\]
so that
\[\eta^r_s:=\eta_s^{G_{r_0-r}}=\inf\{t:A^r_t>s\}.\]
\begin{lemma}\label{Areg}
(a) $Q_{x_0}$-a.s. for all $t\ge 0$ we have
\[A^r_t=\int_0^t1(\inf_{v\le \zeta_u}|W_u(v)|>r_0-r)\,du\quad\forall r\in[0,r_0),\]
and
\[\text{$r\mapsto A^r_t$ is left-continuous on $[0,r_0)$.}\]
(b) $\lim_{r'\uparrow r}\eta^{r'}_s=\eta_s^r$ for all $r\in(0,r_0)$, $s\ge 0$\ \ $Q_{x_0}$-a.s.\\
(c) If $T$ is an $(\cE^+_r)$-stopping time, then $W_{\eta_s^T}$ is $\cE^+_T$-measurable.
\end{lemma}
\begin{proof} The proof is a straightforward modification of that of Lemma 7.4 in \cite{MP17}, where 
shrinking half spaces have now been replaced with shrinking balls. 
\end{proof}

\medskip
\noindent{\bf Proof of Lemma~\ref{Emeas}.} By \eqref{psiform} (with a different radii) and Lemma~\ref{borel}(a) there are  Borel maps $\tilde \psi$ on $\mK$ and $\psi$ on $C(\R_+,\cW)$ such that 
\[1_{D_{r_0}}=\tilde \psi(\mR)=\lim_{N\to\infty}\tilde\psi(\{\hat W(s):s\le N\})=\psi(W),\]
where we have used \eqref{snakerange} in the second equality.  In the last equality we have also called
on the continuity of $W\mapsto\{\hat W(s):s\le N\}$ from $C([0,\infty),\cW)$ to $\mK$. Therefore a monotone class argument shows it suffices to fix $s\ge 0$ and show that if $\phi:\cW\to\R$ is bounded Borel then 
\begin{equation}\label{fixedtime}
\phi(W_s)\text{ is }\cE^+_{T_0-}-\text{measurable}.
\end{equation}
Lemma~\ref{Areg}(b) implies that $W_{\eta_s^{T_0}}=\lim_{n\to\infty}W_{\eta_s^{T_{n^{-1}}}}$
 $Q_{x_0}$-as. and so by Lemma~\ref{Areg}(c) and \eqref{E-}, $W_{\eta_s^{T_0}}$ is $\cE^+_{T_0-}$-measurable.  So to prove \eqref{fixedtime} it suffices to show 
 \[W_s=W_{\eta_s^{T_0}}\quad Q_{x_0}-\text{a.s. }\]
 This, in turn, would follow from $A_t^{T_0}=t$ for all $t\ge 0$ $Q_{x_0}$-a.s., 
 or equivalently by Lemma~\ref{Areg}(a),
 \begin{equation}\label{notatmin}
 \int_0^\sigma 1(\inf_{v\le \zeta_u}|W_u(v)|\le r_0-T_0)\,du=0\quad Q_{x_0}-\text{a.s.}
 \end{equation}
 Here we have truncated the integral at $\sigma$ since $\zeta_u=0$ and $|W_u(0)|=|x_0|\ge 2r_0$ for $u\ge \sigma$.
 If $0\le u<\zeta_s$ and $s'<s$ is the last time before $s$ that $\zeta_{s'}=u$, then $\inf_{t\in[s',s]}\zeta_t=\zeta_{s'}=u$ and so (e.g., see p. 66 of \cite{Leg99}) $W_s(u)=\hat W(s')$ $Q_{x_0}$-a.s.  This and Lemma~\ref{lem:T0s} (recall also \eqref{hatT0def})  imply
 \begin{equation}\label{minrad}
 \inf_{u\le \sigma}\inf_{v\le \zeta_u}|W_u(v)|=\hat T_0=\inf\{|x|:x\in\mR\}=r_0-T_0\quad Q_{x_0}-\text{a.s.}
 \end{equation}
 Therefore \eqref{notatmin} is equivalent to
 \begin{equation}\label{notatmin2}
 \int_0^\sigma 1(\inf_{v\le \zeta_u}|W_u(v)|=\hat T_0) \,du=0\quad Q_{x_0}-\text{a.s.}
 \end{equation}

The historical process, $(H_t,t\ge 0)$ is an inhomogeneous Markov process under $\N_{x_0}$ taking values
in $M_F(C(\R_+,\R^d))$--see \cite{DP91} or p. 64 of \cite{Leg99} to see how it is easily
defined from the snake $W$.  The latter readily implies
\begin{equation}\label{hist}
\int_0^\infty H_t(\phi)dt=\int_0^\sigma \phi(W_u)\,du\quad\text{for all non-negative Borel }\phi,\end{equation}
where we have extended $W_u$ to $\R_+$ in the obvious manner.  Recalling \eqref{hatT0def} and letting $X$ be the SBM under $\N_{x_0}$ as usual, we have 
\begin{align}\label{notimeatedge1}
\N&_{x_0}\Bigl(\int_0^\infty1(\inf_{v\le \zeta_u}|W_u(v)|=\hat T_0)du\Bigr)\\
\nn&\le \N_{x_0}\Bigl(\int_0^\infty \int1(\inf_{t'}|y_{t'}|=\hat T_0)H_t(dy)dt\Bigr)\quad\text{(by \eqref{hist})}\\
\label{notimeatedge2}&\le\int_0^\infty \N_{x_0}\Bigl(\int 1\Bigl(\int_0^\infty X_s(\{x:|x|<\inf_{t'\le t}|y(t')|\})ds=0\Bigr)H_t(dy)\Bigr) dt,
\end{align}
where in the last line we use \eqref{minrad} and $y(\cdot)=y(\cdot\wedge t)\ H_t-\text{a.a.}\ y\ \ \forall t\ge 0$ $\N_{x_0}$-a.e. Below we will let $B$ denote a $d$-dimensional Brownian motion starting at $x_0$ under $P^B_{x_0}$, $m_t=\inf_{t'\le t}|B_{t'}|=|B_{\tau_t}|$ (for some $\tau_t<t$), and $L^x$ be the local time of the SBM $X$ (at time infinity).  Fix $t>0$ and use the Palm measure formula for $H_t$ (e.g. Proposition~4.1.5 of \cite{DP91}) to see that (cf. (7.22) in \cite{MP17})
\begin{align}\label{notimeatedge3}
 \nn\N&_{x_0}\Bigl(\int 1\Bigl(\int_0^\infty X_s(\{x:|x|<\inf_{t'\le t}|y(t')|\})ds=0\Bigr)H_t(dy)\Bigr)\\
\nn&=E^B_{x_0}\Bigl(\exp\Bigl(-\int_0^t\int 1\Bigl(\int_0^\infty X_s(\{x:|x|<m_t\})ds>0\Big)d\N_{B_u}du\Bigr)\Bigr)\\
&\le E_{x_0}^B\Bigl(\exp\Bigl(-\int_0^t\N_{B_u}(L^{B_{\tau_t}}>0)du\Bigr)\Bigr).
\end{align}
It follows from \eqref{VinfL}, \eqref{Vinf} and $\P_{\delta_x}(L^y=0)=\exp(-\N_x(L^y>0))$ (see, e.g., (2.12) in
 \cite{MP17}) that
\[\N_x(L^y>0)=2(4-d)|x-y|^{-2}.\]  
Use this to bound  \eqref{notimeatedge3} by 
\[E^B_{x_0}\Bigl(\exp\Bigl(-\int_0^t\frac{2(4-d)}{|B_s-B_{\tau_t}|^2}ds\Bigr)\Bigr).\]
A simple application of L\'evy's modulus for $B$ shows the above integral is infinite a.s. and
so proves that \eqref{notimeatedge1} equals zero. This implies \eqref{notatmin2}, as required.
\qed
}
\medskip

\appendix

\section{Appendix}
\subsection{Proof of Lemma~\ref{borel}}\label{secmeas}
 (a) Let $K^{(\veps)}=\{x:d(x,K)<\veps\}$.  If $K_0\in\mK$ and $0<r\le 1$ are fixed it suffices to show that 
$\{K\in\mK:d(K\cap\overline B_R,K_0)<r\}$ is Borel.  If $r_n\uparrow r$, this set is equal to
\begin{align*} 
&\{K:K\cap \overline B_R\subset K_0^{(r)}\}\cap\{K:K_0\subset (K\cap\overline B_R)^{(r)}\}\\
=&\{K:K\cap \overline B_R\subset K_0^{(r)}\}\cap\Bigl(\cup_{n=1}^\infty\{K:K_0\subset (K\cap\overline B_R)^{r_n}\}\Bigr):=S_1\cap\Bigl(\cup_{n=1}^\infty S^n_2\Bigr).
\end{align*}
It is then not hard to show that $S_1$ is open in $\mK$ and $S_2^n$ is closed in $\mK$.\\
\no 
(b) This easily reduces to showing that for any fixed rationals $q\in(0,\alpha)$ and $r\in(0,R)$, the following describes an universally measurable subset of $K's$ in $\mK$:
\begin{align*}&\text{For any natural number }N \text{ there is a finite number of open balls $B^1,\dots, B^M$}\\
&\text{centered at points in }\Q^d\text{ and with rational radii }r_1,\dots,r_M>0\text{ satisfying }\sum_{i=1}^M r_i^q<N^{-1}\\
&\text{ so that }\partial K\cap\overline B_{r}\subset\cup_{i=1}^M B^i.
\end{align*}
So fixing $B_i$ and $r$ as above, it suffices to show
\begin{equation*}
A_1=\{K\in\mK:\partial K\cap\overline B_r\subset \cup_{i=1}^MB^i\}^c\text{ is an analytic set in }\mK,
\end{equation*}
because this implies $A_1$,  and hence $A_1^c$, is a universally measurable set in $\mK$.
Let \hfil\break $K_0=\Bigl(\cup_{i=1}^M B^i\Bigr)^c\cap \overline B_r\in \mK$, $H_0=\{(x,K)\in \R^d\times\mK:x\in K\}$, and for $n\ge 1$, set $H_n=\{(x,K)\in \R^d\times\mK:B_{n^{-1}}(x)\not\subset K\}$.  Then 
\begin{align*}A_1&=\{K\in\mK:K_0\cap\partial K\neq\emptyset\}\\
&=\{K\in\mK:\exists x\in K_0\text{ s.t. } x\in K\text{ and }\forall n\in\N\ B_{n^{-1}}(x)\cap K^c\neq\emptyset\}\\
&=\{K\in\mK:\exists x\in \R^d\text{ s.t. }(x,K)\in (K_0\times\mK)\cap\Bigl(\cap_{n=0}^\infty H_n\Bigr)\}.
\end{align*}
Using the well-known fact that the projection of a Borel subset of $K_0\times \mK$ onto $\mK$ is an analytic subset of $\mK$ (see eg. Theorem 13 in Ch. III of \cite{DM78} and note the argument goes through with $\R^d$ in place of $\R$), it then suffices to show each $H_n$ is Borel. $H_0$ is the countable intersection of the open sets
$H^M_0=\{(x,K):x\in K^{(1/M)}\}$. Moreover it is not hard to see that $H_n$ is open for $n\ge 1$, and we are done.\qed

\subsection{Proof of Lemma~\ref{lem:22_7_1}}
\label{sec22_7_1}
For the proof we will use the following lemma of Marc Yor (see Proposition 2.5 of \cite{MP17}). Recall that for  $\gamma\in\R$, $(\rho_t)$ denotes a $\gamma$-dimensional Bessel process starting from $r>0$ under $P_{r}^{(\gamma)}$, $(\cF_t)$ is the filtration generated by $\rho$ and $\tau_R=\inf\{t\geq 0: \rho_t\leq R\}$ for $R>0$. 

\begin{lemma}\label{l1.0}
Let $\lambda\geq 0$, $\mu\in \R, r>0$ and $\nu=\sqrt{\lambda^2+\mu^2}$. If $\Phi_t\geq 0$ is $\cF_t$-adapted, then for all $R<r$, we have
\[E_r^{(2+2\mu)}\Big(\Phi_{t\wedge \tau_R} \exp\Big(-\frac{\lambda^2}{2} \int_0^{t\wedge \tau_R} \frac{1}{\rho_s^2} ds\Big)\Big)=r^{\nu-\mu} E_{r}^{(2+2\nu)} \Big((\rho_{t\wedge \tau_R})^{-\nu+\mu} \Phi_{t\wedge \tau_R}\Big).\]
\end{lemma}
Now we are ready to give 
\paragraph{Proof of Lemma~\ref{lem:22_7_1}}
We use Fatou's lemma and then Lemma~\ref{l1.0} to get that for $a\geq 0$, 
\begin{align*}
E_x&\Big(1(\tau _R<\infty)\exp\Bigl(\int_{0}^{\tau_R} \frac{a}{|B_s|^q}ds\Bigr)  \exp\Bigl(-\int_{0}^{\tau_R} \frac{2(4-d)-\zeta/2}{|B_s|^2}ds\Bigr) \Big)
 \\
&\leq \liminf_{t\to \infty} E_{|x|}^{(2+2\mu)}\Big(1_{(\tau_R\leq \tau_R\wedge t )} \exp\Bigl(\int_{0}^{\tau_R\wedge t } \frac{a}{\rho_s^q}ds\Bigr) \exp\Bigl(-\int_{0}^{\tau_R\wedge t} \frac{2(4-d)-(\zeta/2)}{\rho_s^2}ds\Bigr) \Big) \nonumber\\
&  = |x|^{\nu_\zeta-\mu}\liminf_{t\to \infty} E_{|x|}^{(2+2\nu_\zeta)}\Big(1_{(\tau_R\leq \tau_R\wedge t )}\exp\Bigl(\int_{0}^{\tau_R\wedge t } \frac{a}{\rho_s^q}ds\Bigr)  \rho_{t\wedge \tau_R}^{\mu-\nu_\zeta} \Big) \nonumber
\\
&
= (R/|x|)^{\mu-\nu_\zeta} E_{|x|}^{(2+2\nu_\zeta)}\Big(1_{(\tau_R<\infty)}\exp\Bigl(\int_{0}^{\tau_R} \frac{a}{\rho_s^q}ds\Big)\Big)\quad(\text{since }\rho_{t\wedge \tau_R}=R\text{ on }\{\tau_R\le t\})
\\
&
= (R/|x|)^{p_\zeta} E_{|x|}^{(2+2\nu_\zeta)}\Big(\exp\Bigl(\int_{0}^{\tau_R} \frac{a}{\rho_s^q}ds\Bigr)  \Bigl|\tau_R<\infty\Big), 
\end{align*}
where in next to  the last line we use monotone convergence for $a\geq 0$, and in the last line the hitting probabilities for Bessel processes (e.g. (48.3) and (48.5) in Ch. V of \cite{RW94}) as well as $p_\zeta=\mu+\nu_\zeta$. Note that for $a<0$,
 by bounded convergence, we get equality in the second line above (with $\liminf_{t\rightarrow \infty}$ replaced by $\lim_{t\rightarrow \infty}$) and thus proceeding as above we get, by using  bounded convergence again in the next to the last line, that \eqref{eq:23_7_1} holds for $a<0$.

It remains to verify the lower bound in \eqref{eq:23_7_1}, for $a\geq 0$. Fix $T>0$. Then we have 
\begin{align*}
E_x&\Big(1(\tau _R<\infty)\exp\Bigl(\int_{0}^{\tau_R} \frac{a}{|B_s|^q}ds\Bigr)  \exp(-\int_{0}^{\tau_R} \frac{2(4-d)-\zeta/2}{|B_s|^2}ds) \Big)
 \\
&\geq  E_{|x|}^{(2+2\mu)}\Big(1_{(\tau_R<\infty)} \exp\Bigl(\int_{0}^{\tau_R\wedge T } \frac{a}{\rho_s^q}ds\Bigr) \exp\Bigl(-\int_{0}^{\tau_R} \frac{2(4-d)-(\zeta/2)}{\rho_s^2}ds\Bigr) \Big) \nonumber\\
&  = |x|^{\nu_\zeta-\mu}\lim_{t\to \infty} E_{|x|}^{(2+2\nu_\zeta)}\Big(1_{(\tau_R\leq \tau_R\wedge t )}\exp\Bigl(\int_{0}^{\tau_R\wedge t \wedge T} \frac{a}{\rho_s^q}ds\Bigr)  R^{\mu-\nu_\zeta} \Big) \nonumber
\\
&
= (R/|x|)^{\mu-\nu_\zeta} E_{|x|}^{(2+2\nu_\zeta)}\Big(1_{(\tau_R<\infty)}\exp\Bigl(\int_{0}^{\tau_R\wedge T} \frac{a}{\rho_s^q}ds\Big)\Big)
\\
&
= (R/|x|)^{p_\zeta} E_{|x|}^{(2+2\nu_\zeta)}\Big(\exp\Bigl(\int_{0}^{\tau_R\wedge T} \frac{a}{\rho_s^q}ds\Bigr)  \Bigl|\tau_R<\infty\Big),
\end{align*}
where in the first equality we used bounded convergence and  Lemma~\ref{l1.0}, in the second equality bounded convergence again, and in the last equality the hitting probabilities 
for Bessel processes. Now let $T\rightarrow \infty$, to get the required lower bound, and we are done. 
\qed
\ARXIV{\subsection{Proof of Proposition \ref{p3.1} }
\label{sec:5}
For $|x_i| \geq \varepsilon_0, i=1,2,$ and $\varepsilon \in (0,\varepsilon_0)$, if $|x_1-x_2|\leq 5 \varepsilon$, then use $xe^{-x} \leq e^{-1}, \forall x\geq 0$ to get
\begin{align*}
&\E_{\delta_0}\Big(\prod_{i=1}^2 \lambda \frac{X_{G_\varepsilon^{x_i}}(1)}{\varepsilon^2} \exp\big(-\lambda \frac{X_{G_\varepsilon^{x_i}}(1)}{\varepsilon^2}\big)\Big) \leq e^{-1} \E_{\delta_0}\Big(\lambda \frac{X_{G_\varepsilon^{x_1}}(1)}{\varepsilon^2} \exp\big(-\lambda \frac{X_{G_\varepsilon^{x_1}}(1)}{\varepsilon^2}\big)\Big).
\end{align*}
Recall the definition of $F=F_{\varepsilon,x_1}$ in \eqref{e1.8.1}. For all $\lambda>0$, an integration by parts gives 
\begin{align*}
&\E_{\delta_0}\Big(\lambda \frac{X_{G_\varepsilon^{x_1}}(1)}{\varepsilon^2} \exp\big(-\lambda \frac{X_{G_\varepsilon^{x_1}}(1)}{\varepsilon^2}\big)\Big)= \int_0^\infty \lambda xe^{-\lambda x} dF(x)\\
=& \int_0^\infty \lambda (\lambda x-1)e^{-\lambda x}F(x) dx
=\int_0^\infty  (y-1)e^{-y}F(\frac{y}{\lambda}) dy\leq F(2)+\int_{2\lambda}^\infty ye^{-y} F(\frac{y}{\lambda})dy\\
\leq &c_{\ref{t1.2}} 2^{p-2} \varepsilon^{p-2}+\int_{2\lambda}^\infty ye^{-y}  c_{\ref{t1.2}} (\frac{y}{\lambda})^{p-2} \varepsilon^{p-2} dy= C(\varepsilon_0, \lambda) \varepsilon^{p-2},
\end{align*}
the last line by Proposition  \ref{t1.2}.
Therefore 
\[\E_{\delta_0}\Big(\prod_{i=1}^2 \lambda \frac{X_{G_\varepsilon^{x_i}}(1)}{\varepsilon^2} \exp\big(-\lambda \frac{X_{G_\varepsilon^{x_i}}(1)}{\varepsilon^2}\big)\Big) \leq e^{-1} C(\varepsilon_0, \lambda) \varepsilon^{p-2}\leq e^{-1} 5^{p-2} C(\varepsilon_0, \lambda) |x_1-x_2|^{2-p} \varepsilon^{2(p-2)},\] provided $|x_1-x_2|\leq 5 \varepsilon$. As a result, \\

\noindent $\mathbf{throughout\ the\ rest\ of\ this\ Section\ we\ may\ fix }$ $\varepsilon_0>0$, $|x_i|\geq \varepsilon_0$ {\bf and} $\varepsilon \in (0,\varepsilon_0)$ {\bf with} $|x_1-x_2|>5 \varepsilon$. In this case, we have $B({x_1},2\varepsilon) \cap B({x_2},2\varepsilon) =\emptyset$.\\

Let  $\vec{x}=(x_1,x_2)$, $G=G_\varepsilon^{x_1} \cap G_\varepsilon^{x_2}$, and $\vec{\lambda}=(\lambda_1,\lambda_2) \in [0, \infty)^2\backslash \{(0,0)\}$. For $X_0\in M_F(\R^d)$ such that $d(\text{Supp}(X_0),G^c)>0$, the decomposition \eqref{exitdecomp} with $G=G_\veps^{x_i}$, $i=1,2$, gives
\begin{align}\label{e4.01}
\E_{X_0}\Big( \exp\big(-\sum_{i=1}^2 \lambda_i \frac{X_{G_\varepsilon^{x_i}}(1)}{\varepsilon^2}\big)\Big) =\exp\Big(-\int U^{\vec{\lambda},\vec{x},\varepsilon} (x) X_0(dx)\Big),
\end{align} 
where $U^{\vec{\lambda},\vec{x},\varepsilon}\geq 0$ is defined as
\begin{align}\label{e4.3}
U^{\vec{\lambda},\vec{x},\varepsilon}(x) \equiv \N_x\Big(1-\exp\big(-\sum_{i=1}^2 \lambda_i \frac{X_{G_\varepsilon^{x_i}}(1)}{\varepsilon^2}\big) \Big).
\end{align} 
 We use results from Chapter V of \cite{Leg99} to get the following lemma.
\begin{lemma}\label{l4.0}
$U^{\vec{\lambda},\vec{x},\varepsilon}$ is a $C^2$ function on $G$ and solves
\begin{align}\label{e4.4}
\Delta U^{\vec{\lambda},\vec{x},\varepsilon}=(U^{\vec{\lambda},\vec{x},\varepsilon})^2 \text{ on } G.
\end{align}
Moreover, \[ U^{\vec{\lambda},\vec{x},\varepsilon}(x) \leq (\lambda_1+\lambda_2) \varepsilon^{-2},\ \forall x\in G.\]
\end{lemma}
\begin{proof}

Let \[u(x)\equiv U^{\vec{\lambda},\vec{x},\varepsilon}(x)= \N_x\Big( 1-\exp\big(-\sum_{i=1}^2 \lambda_i \frac{X_{G_\varepsilon^{x_i}}(1)}{\varepsilon^2} \big)\Big).\]
Then use $1-e^{-x}\leq x$ to get
\begin{align}\label{e4.02}
u(x)&\leq \N_x\Big( \sum_{i=1}^2 \lambda_i \frac{X_{G_\varepsilon^{x_i}}(1)}{\varepsilon^2}\Big)=\sum_{i=1}^2 \lambda_i \varepsilon^{-2} P_x(\tau_i<\infty ) \leq (\lambda_1+\lambda_2) \varepsilon^{-2},
\end{align}
the equality by Proposition V.3 of \cite{Leg99}, where $(B_t)$ is $d$-dimensional Brownian motion starting from $x$ under $P_x$ and $\tau_i=\inf\{t\geq 0: B_t \notin G_\varepsilon^{x_i}\}$.\\

Next, for any $x'\in G$, let $D$ be an open ball that contains $x'$, whose closure is in $G$. Use \eqref{e4.01} with $X_0=\delta_x$ and then Proposition \ref{p0.1}(b)(i) to see that for $x\in D$,
\begin{align*}
e^{-u(x)}
=&\E_{\delta_x}\Big( \exp\big(-\sum_{i=1}^2 \lambda_i \frac{X_{G_\varepsilon^{x_i}}(1)}{\varepsilon^2} \big) \Big)=\E_{\delta_x}\Big( \E_{X_D} \Big(\exp\big(-\sum_{i=1}^2 \lambda_i \frac{X_{G_\varepsilon^{x_i}}(1)}{\varepsilon^2} \big)\Big)\Big)\\
=& \E_{\delta_x}\Big(  \exp\big(-\int  u(x) X_D(dx)  \big)\Big)=\exp\Big(- \N_x\big( 1-\exp\big(-\int  u(y) X_D(dy)\big)\big) \Big),
\end{align*}
 the third  equality by \eqref{e4.01} with $X_0=X_D$, and the last by the decomposition \eqref{exitdecomp}. Therefore 
 \[u(x)=\N_x\big( 1-\exp\big(-\int  u(y) X_D(dy)\big)\big)\quad\forall x\in D.\] 
 Note $u$ is bounded in $G$ by \eqref{e4.02}, and hence on $\partial D$. Use Theorem V.6 of \cite{Leg99} to conclude \[\Delta u(x)=(u(x))^2, \ \forall x\in D,\text{ and, in particular, for }x=x'.\] 
 Since $x'$ is arbitrary, it holds for all $x\in G$.
\end{proof}

Let $X_0=\delta_x$ in \eqref{e4.01} for $x\in G$ to get 
\begin{align}\label{e4.1}
\E_{\delta_x}\Big( \exp\big(-\sum_{i=1}^2 \lambda_i \frac{X_{G_\varepsilon^{x_i}}(1)}{\varepsilon^2}\big)\Big) =\exp(-U^{\vec{\lambda},\vec{x},\varepsilon}(x)).
\end{align} 
Monotone convergence and the convexity of $e^{-ax}$ for $a,x>0$ allow us to differentiate the left-hand side of \eqref{e4.1} with respect to $\lambda_i>0$ through the expectation and so conclude that for $i=1,2,$ $U_i^{\vec{\lambda},\vec{x},\varepsilon}(x)=\frac{\partial}{\partial \lambda_i} U^{\vec{\lambda},\vec{x},\varepsilon}(x)$ exists and
\[\E_{\delta_x}\Big( \frac{X_{G_\varepsilon^{x_i}}(1)}{\varepsilon^2}\exp\big(-\sum_{i=1}^2 \lambda_i \frac{X_{G_\varepsilon^{x_i}}(1)}{\varepsilon^2}\big)\Big)=e^{-U^{\vec{\lambda},\vec{x},\varepsilon}(x)} U_i^{\vec{\lambda},\vec{x},\varepsilon}(x) \text{ for } \lambda_i>0, \lambda_{3-i}\geq 0.\]
Repeat the above to see that $U^{\vec{\lambda},\vec{x},\varepsilon}(x)$ is $C^2$ in $\lambda_1, \lambda_2>0$ and if $U_{1,2}^{\vec{\lambda},\vec{x},\varepsilon}(x)=\frac{\partial^2}{\partial \lambda_1 \partial \lambda_2} U^{\vec{\lambda},\vec{x},\varepsilon}(x)$, then
\begin{align}\label{e4.2}
\E_{\delta_x}\Big(& \frac{X_{G_\varepsilon^{x_1}}(1)}{\varepsilon^2} \frac{X_{G_\varepsilon^{x_2}}(1)}{\varepsilon^2}\exp\big(-\sum_{i=1}^2 \lambda_i \frac{X_{G_\varepsilon^{x_i}}(1)}{\varepsilon^2}\big)\Big)\\
\nn&=e^{-U^{\vec{\lambda},\vec{x},\varepsilon}(x)} \Big[U_1^{\vec{\lambda},\vec{x},\varepsilon}(x)U_2^{\vec{\lambda},\vec{x},\varepsilon}(x)-U_{1,2}^{\vec{\lambda},\vec{x},\varepsilon}(x)\Big],\text{ for $\lambda_1, \lambda_2>0$}. 
\end{align}

The  next monotonicity result follows just as in the proof of Lemma~9.2 of \cite{MP17}.
\begin{lemma}\label{l4.1}
{\ }
\begin{enumerate}[(a)]
\item $U_i^{\vec{\lambda},\vec{x},\varepsilon}(x)>0$ is strictly decreasing in $\vec{\lambda} \in \{(\lambda_1,\lambda_2):\lambda_i>0,\lambda_{3-i}\geq 0\},$ for $i=1,2$.
\item $-U_{1,2}^{\vec{\lambda},\vec{x},\varepsilon}(x)>0$ is strictly decreasing in $\vec{\lambda} \in (0,\infty)^2.$
\end{enumerate}
\end{lemma}

\pf{\begin{proof}
(a) Differentiate \eqref{e4.3} with respect to $\lambda_i>0$ to conclude
\begin{align}\label{e4.21}
U_i^{\vec{\lambda},\vec{x},\varepsilon}(x)=\int \frac{X_{G_\varepsilon^{x_i}}(1)}{\varepsilon^2} \exp\big(-\sum_{i=1}^2 \lambda_i \frac{X_{G_\varepsilon^{x_i}}(1)}{\varepsilon^2}\big) \N_x(dW), \ \forall \lambda_i>0,
\end{align}
where differentiation through the integral follows by Monotone Convergence Theorem. This shows the strict monotonicity of $U_i^{\vec{\lambda},\vec{x},\varepsilon}(x)>0$ in $\vec{\lambda}$ since $\N_x(X_{G_\varepsilon^{x_i}}(1)>0)>0$.\\

(b) For $\lambda_1, \lambda_2>0$ we can take $i=1$ and differentiate \eqref{e4.21} with respect to $\lambda_2$ and use Monotone
Convergence to differentiate through the integral, and so conclude
\[-U_{1,2}^{\vec{\lambda},\vec{x},\varepsilon}(x)=\int \frac{X_{G_\varepsilon^{x_1}}(1)}{\varepsilon^2} \frac{X_{G_\varepsilon^{x_2}}(1)}{\varepsilon^2} \exp\big(-\sum_{i=1}^2 \lambda_i \frac{X_{G_\varepsilon^{x_i}}(1)}{\varepsilon^2}\big) \N_x(dW)>0.\] This shows the strict monotonicity of $-U_{1,2}^{\vec{\lambda},\vec{x},\varepsilon}(x)>0$ in $\vec{\lambda}$ since $\N_x(X_{G_\varepsilon^{x_1}}(1) X_{G_\varepsilon^{x_2}}(1)>0)>0$.
\end{proof}}

Note that 
\begin{align}\label{e4.20}
U^{\vec{\lambda},\vec{x},\varepsilon}(x)=U^{\lambda_{i}\varepsilon^{-2},\varepsilon}(x-x_i), \ \text{ for } \lambda_{i}>0 \text{ and } \lambda_{3-i}=0.
\end{align}
The above monotonicity results easily give the following, just as for Lemma~9.3 of \cite{MP17}.
\begin{lemma}\label{l4.2}
\begin{enumerate}[(a)]
\item For all $\lambda_i>12$ and $\lambda_{3-i}\geq 0$,
\[U_i^{\vec{\lambda},\vec{x},\varepsilon}(x)\leq \frac{2}{\lambda_i}(U^{\lambda_i \varepsilon^{-2},\varepsilon}(x_i-x)-U^{(\lambda_i/2) \varepsilon^{-2},\varepsilon}(x_i-x)) \leq \frac{2}{\lambda_i} \frac{2^p}{|x_i-x|^p} D^{\lambda_i/2}(2)\varepsilon^{p-2},\ \forall |x_i-x|\geq 2\varepsilon.\] 
\item For all $\lambda_1, \lambda_2>12$,
\begin{align*}
-U_{1,2}^{\vec{\lambda},\vec{x},\varepsilon}(x)&\leq \frac{4}{\lambda_1 \lambda_2}\min_{i=1,2} (U^{\lambda_i \varepsilon^{-2},\varepsilon}(x_i-x)-U^{(\lambda_i/2) \varepsilon^{-2},\varepsilon}(x_i-x))\\
&\leq \frac{4}{\lambda_1 \lambda_2}{2^p} ([D^{\lambda_1/2}(2)|x_1-x|^{-p}]\wedge [D^{\lambda_2/2}(2)|x_2-x|^{-p}]) \varepsilon^{p-2},\ \forall |x_i-x|\geq 2\varepsilon,\ i=1,2.
\end{align*}

\end{enumerate}
\end{lemma}
\pf{\begin{proof} 
By Corollary \ref{c1.4}(a) with $\lambda_i/2>6$ and $R=2$, it suffices to show the first inequalities in $(a)$ and $(b)$.\\

(a) By symmetry take $i= 1$. The monotonicity in Lemma \ref{l4.1}(a) and the Fundamental Theorem of Calculus imply
\begin{align*}
U_1^{\vec{\lambda},\vec{x},\varepsilon}(x)\leq \frac{2}{\lambda_1} \int_{\lambda_1/2}^{\lambda_1} U_1^{(\lambda_1^{'},\lambda_2),\vec{x},\varepsilon}(x) d\lambda_1^{'}&\leq \frac{2}{\lambda_1} \int_{\lambda_1/2}^{\lambda_1} U_1^{(\lambda_1^{'},0),\vec{x},\varepsilon}(x) d\lambda_1^{'}\\
&= \frac{2}{\lambda_1}(U^{\lambda_1 \varepsilon^{-2},\varepsilon}(x_1-x)-U^{(\lambda_1/2) \varepsilon^{-2},\varepsilon}(x_1-x)) .
\end{align*}
(b) Argue as above using the monotonicity in Lemma \ref{l4.1}(b) to see that
\begin{align*}
-U_{1,2}^{\vec{\lambda},\vec{x},\varepsilon}(x)&\leq \frac{2}{\lambda_1} \int_{\lambda_1/2}^{\lambda_1} -\frac{\partial}{\partial \lambda_1^{'}} U_2^{(\lambda_1^{'},\lambda_2),\vec{x},\varepsilon}(x) d\lambda_1^{'}\leq \frac{2}{\lambda_1}  U_2^{(\lambda_1/2, \lambda_2),\vec{x},\varepsilon}(x) \\
&\leq \frac{4}{\lambda_1 \lambda_2}(U^{\lambda_2 \varepsilon^{-2},\varepsilon}(x_2-x)-U^{(\lambda_2/2) \varepsilon^{-2},\varepsilon}(x_2-x)),
\end{align*}
the last line by part (a) with $i=2$. The first inequality now follows by symmetry.
\end{proof}}

Let $r_{\varepsilon}=2\varepsilon$ and assume $0<r_{\varepsilon}<\min\{|x_i-x|: i=1,2\}$. Set $T_{r_{\varepsilon}}^i=\inf\{t\geq 0: |B_t-x_i|\leq r_{\varepsilon} \}$ and $T_{r_{\varepsilon}}=T_{r_{\varepsilon}}^1 \wedge T_{r_{\varepsilon}}^2$, and let $(\cF_t)$ denote the right-continuous filtration generated by the Brownian motion $B$, which starts at $x$ under $P_x$.

\begin{lemma}\label{l4.3}
Let $\lambda_1, \lambda_2>0$.
\begin{enumerate}[(a)]
\item $U_1^{\vec{\lambda},\vec{x},\varepsilon}(B(t\wedge T_{r_{\varepsilon}}))-\int_0^{t\wedge T_{r_{\varepsilon}}} U^{\vec{\lambda},\vec{x},\varepsilon}(B(s))U_1^{\vec{\lambda},\vec{x},\varepsilon}(B(s)) ds$ is an $(\cF_t)$-martingale. 
\item For any $t>0$, \[U_1^{\vec{\lambda},\vec{x},\varepsilon}(x)=E_x\Big( U_1^{\vec{\lambda},\vec{x},\varepsilon}(B({t\wedge T_{r_{\varepsilon}}})\exp\big(-\int_0^{t\wedge T_{r_{\varepsilon}}} U^{\vec{\lambda},\vec{x},\varepsilon}(B(s)) ds\big)\Big).\]
\end{enumerate}
\end{lemma}
This result follows from Lemmas~\ref{l4.0}, \ref{l4.2} and It\^o's Lemma,  exactly as for  Lemma~9.4 in \cite{MP17},
and so the proof is omitted.
\pf{\begin{proof}
(a) By Lemma \ref{l4.0}, for $\delta>0$,
\begin{align}\label{e4.22}
\Big(\frac{\Delta U^{\vec{\lambda}+(\delta,0),\vec{x},\varepsilon}}{2}(x)-\frac{\Delta U^{\vec{\lambda},\vec{x},\varepsilon}}{2}(x)\Big) \delta^{-1}&=\Big(\frac{\big(U^{\vec{\lambda}+(\delta,0),\vec{x},\varepsilon}(x)\big)^2}{2}-\frac{\big(U^{\vec{\lambda},\vec{x},\varepsilon}(x)\big)^2}{2}\Big) \delta^{-1}\nonumber \\
&\to U_1^{\vec{\lambda},\vec{x},\varepsilon}(x) U^{\vec{\lambda},\vec{x},\varepsilon}(x),
\end{align}
 as $\delta \to 0$.
Moreover the bounds on $U_1^{\vec{\lambda},\vec{x},\varepsilon}(x)$ in Lemma \ref{l4.2}(a) and on $U^{\vec{\lambda},\vec{x},\varepsilon}(x)$ in Lemma \ref{l4.0} show that the
above pointwise convergence is also uniformly bounded for $x$ satisfying $|x -x_i| > r_{\varepsilon}$. It\^o's lemma
shows that $U^{\vec{\lambda},\vec{x},\varepsilon}(B(t\wedge T_{r_{\varepsilon}}))-\int_0^{t\wedge T_{r_{\varepsilon}}} 
\frac{\Delta U^{\vec{\lambda},\vec{x},\varepsilon}(B(s))}{2} ds$ is an $\cF_t$-martingale. Therefore if $s < t$ and $\delta > 0$,
\begin{align*}
E_x\Big(&\frac{U^{\vec{\lambda}+(\delta,0),\vec{x},\varepsilon}(B(t\wedge T_{r_{\varepsilon}}))-U^{\vec{\lambda},\vec{x},\varepsilon}(B(t\wedge T_{r_{\varepsilon}}))}{\delta} \Big| \cF_s\Big)\\
&=E_x\Big(\int_0^{t\wedge T_{r_{\varepsilon}}} \frac{\Delta U^{\vec{\lambda}+(\delta,0),\vec{x},\varepsilon}(B(r))-\Delta U^{\vec{\lambda},\vec{x},\varepsilon}(B(r))}{2\delta} dr \Big| \cF_s\Big).
\end{align*}
The left-hand side of the above approaches $E_x\Big(U_1^{\vec{\lambda},\vec{x},\varepsilon}(B(t\wedge T_{r_{\varepsilon}}))\Big| \cF_s\Big)$as $\delta \to 0$ (the uniform boundedness
of $U_1^{\vec{\lambda},\vec{x},\varepsilon}(x)$ noted above allows us to take the limit through the conditional expectation).
The result now follows by letting $\delta \to 0$ in the above and applying \eqref{e4.22} and the boundedness
established above to take the limits through the conditional expectation and Lebesgue integral on
the right-hand side.\\
(b) By (a), It\^o's lemma, and boundedness of $U_1^{\vec{\lambda},\vec{x},\varepsilon}(x)$ on $\{|x-x_1|>2\varepsilon\}$ (from Lemma \ref{l4.2} (a)), $U_1^{\vec{\lambda},\vec{x},\varepsilon}(B({t\wedge T_{r_{\varepsilon}}}))\exp\big(-\int_0^{t\wedge T_{r_{\varepsilon}}} U^{\vec{\lambda},\vec{x},\varepsilon}(B(s)) ds\big)$
is an $\cF_t$-martingale and so the result follows.
\end{proof}}

\begin{lemma}\label{l4.4}
For all $\lambda_1, \lambda_2>0$,
\begin{align*}
 -U_{1,2}^{\vec{\lambda},\vec{x},\varepsilon}(x)
 =&E_x\Big( \int_0^{T_{r_\varepsilon}} \prod_{i=1}^2 U_i^{\vec{\lambda},\vec{x},\varepsilon}(B({t}))\exp\big(-\int_0^{t} U^{\vec{\lambda},\vec{x},\varepsilon}(B(s)) ds\big)dt\Big)\\
 &+E_x\Big( \exp\big(-\int_0^{ T_{r_{\varepsilon}}} U^{\vec{\lambda},\vec{x},\varepsilon}(B(s)) ds\big)1(T_{r_{\varepsilon}}<\infty) (-U_{1,2}^{\vec{\lambda},\vec{x},\varepsilon}(B( T_{r_{\varepsilon}}))\Big)
 \end{align*}
\end{lemma}
This follows from Lemmas~\ref{l4.2} and \ref{l4.3}, as in the proof of Lemma~9.5 of \cite{MP17}.
\pf{\begin{proof}
Using the bounds in Lemma \ref{l4.2} one may easily differentiate the representation for $U_{1}^{\vec{\lambda},\vec{x},\varepsilon}(x)$
in Lemma \ref{l4.3}(b) with respect to $\lambda_2 >0$ through the expectation and obtain
\begin{align}\label{e4.23}
 -U_{1,2}^{\vec{\lambda},\vec{x},\varepsilon}(x)
 =&E_x\Big( U_1^{\vec{\lambda},\vec{x},\varepsilon}(B(t \wedge T_{r_\varepsilon}))\exp\big(-\int_0^{t \wedge T_{r_\varepsilon} } U^{\vec{\lambda},\vec{x},\varepsilon}(B(s)) ds\big) \int_0^{t \wedge T_{r_\varepsilon} } U_2^{\vec{\lambda},\vec{x},\varepsilon}(B({s}))ds\Big)\nonumber \\
 &-E_x\Big(  U_{1,2}^{\vec{\lambda},\vec{x},\varepsilon}(B( t \wedge T_{r_{\varepsilon}})) \exp\big(-\int_0^{ t \wedge T_{r_\varepsilon} } U^{\vec{\lambda},\vec{x},\varepsilon}(B(s)) ds\big)\Big)\nonumber\\
 \equiv & I_1(t)+I_2(t).
 \end{align}
Use the Markov property, then Lemma \ref{l4.3}(b), and then Monotone Convergence to see that
\begin{align}\label{e4.24}
I_1(t)=&E_x\Big( \int_0^{t \wedge T_{r_\varepsilon}} U_2^{\vec{\lambda},\vec{x},\varepsilon}(B({s})) \exp\big(-\int_0^{s } U^{\vec{\lambda},\vec{x},\varepsilon}(B(r)) dr\big) \nonumber\\
& \ \ \ \ \ \times E_{B_s} \Big(U_1^{\vec{\lambda},\vec{x},\varepsilon}(B((t-s) \wedge T_{r_\varepsilon})) \exp\big(-\int_0^{(t-s) \wedge T_{r_\varepsilon} } U^{\vec{\lambda},\vec{x},\varepsilon}(B(r)) dr\big)\Big) ds\Big)\nonumber\\
=&E_x\Big( \int_0^{t \wedge T_{r_\varepsilon}} U_2^{\vec{\lambda},\vec{x},\varepsilon}(B({s})) U_1^{\vec{\lambda},\vec{x},\varepsilon}(B({s})) \exp\big(-\int_0^{s } U^{\vec{\lambda},\vec{x},\varepsilon}(B(r)) dr\big) ds\Big)\nonumber\\
\to &E_x\Big( \int_0^{ T_{r_\varepsilon}} U_2^{\vec{\lambda},\vec{x},\varepsilon}(B({s})) U_1^{\vec{\lambda},\vec{x},\varepsilon}(B({s})) \exp\big(-\int_0^{s } U^{\vec{\lambda},\vec{x},\varepsilon}(B(r)) dr\big) ds\Big), 
 \end{align}
as $t \to \infty.$
 Lemma \ref{l4.2}(b)  shows that $-U_{1,2}^{\vec{\lambda},\vec{x},\varepsilon}(x)$ is uniformly bounded on $\{x : |x- x_i| \geq \varepsilon, i = 1, 2\}$ and $\lim_{|x| \to \infty} -U_{1,2}^{\vec{\lambda},\vec{x},\varepsilon}(x)=0$. Therefore by Dominated Convergence
 \[I_2(t) \to E_x\Big( \exp\big(-\int_0^{ T_{r_{\varepsilon}}} U^{\vec{\lambda},\vec{x},\varepsilon}(B(s)) ds\big)1(T_{r_{\varepsilon}}<\infty) (-U_{1,2}^{\vec{\lambda},\vec{x},\varepsilon}(B( T_{r_{\varepsilon}}))\Big). \]
 as $t\to \infty$, and this, together with \eqref{e4.23} and \eqref{e4.24}, completes the proof.
\end{proof}}

\begin{proof}[Proof of Proposition \ref{p3.1}]
Recall $r_\varepsilon=2\varepsilon$. For the case $\varepsilon \in [\varepsilon_0/2,\varepsilon_0)$, the result follows immediately by  letting $c_{\ref{p3.1}}\geq e^{-2} 2^{2(p-2)} \varepsilon_0^{-2(p-2)}$ and by using $xe^{-x}\leq e^{-1}, \text{ for } x\geq 0$, so we assume 
\begin{align}\label{e4.81}
r_\varepsilon=2\varepsilon<\varepsilon_0.
\end{align}
 Recall that $T_{r_{\varepsilon}}^i=\inf\{t\geq 0: |B_t-x_i|\leq r_{\varepsilon} \}$ and $T_{r_{\varepsilon}}=T_{r_{\varepsilon}}^1 \wedge T_{r_{\varepsilon}}^2$. Since $|x_i|\geq \varepsilon_0$, we have $T_{r_{\varepsilon}}>0, P_0$-a.s.. We set $\vec{\lambda}=(\lambda,\lambda)$, $\vec{x}=(x_1,x_2)$, and $\Delta=|x_1-x_2|$, where the constant $\lambda>0$ will be chosen large below.

Apply  \eqref{e4.2} and Lemma \ref{l4.2}(a) to see that for $\lambda>12$,
\begin{align}\label{e4.8}
&\E_{\delta_0}\Big( \lambda^2 \frac{X_{G_\varepsilon^{x_1}}(1)}{\varepsilon^2} \frac{X_{G_\varepsilon^{x_2}}(1)}{\varepsilon^2}\exp\big(-\lambda \sum_{i=1}^2 \frac{X_{G_\varepsilon^{x_i}}(1)}{\varepsilon^2}\big)\Big)=\lambda^2 e^{-U^{\vec{\lambda},\vec{x},\varepsilon}(x)} \Big[U_1^{\vec{\lambda},\vec{x},\varepsilon}(0)U_2^{\vec{\lambda},\vec{x},\varepsilon}(0)-U_{1,2}^{\vec{\lambda},\vec{x},\varepsilon}(0)\Big]\nonumber \\
\leq& 2^{2p+2} (D^{\lambda/2}(2))^2 |x_1|^{-p}|x_2|^{-p} \varepsilon^{2(p-2)} -\lambda^2 U_{1,2}^{\vec{\lambda},\vec{x},\varepsilon}(0)\leq c\varepsilon_0^{-2p} \varepsilon^{2(p-2)}+\lambda^2(-U_{1,2}^{\vec{\lambda},\vec{x},\varepsilon}(0)).
\end{align}
To bound the last term, use Lemma \ref{l4.4} to get
\begin{align}\label{e4.9}
\lambda^2(-U_{1,2}^{\vec{\lambda},\vec{x},\varepsilon}(0)) =&\lambda^2 E_0\Big( \int_0^{T_{r_\varepsilon}} \prod_{i=1}^2 U_i^{\vec{\lambda},\vec{x},\varepsilon}(B({t}))\exp\big(-\int_0^{t} U^{\vec{\lambda},\vec{x},\varepsilon}(B(s)) ds\big)dt\Big)\nonumber\\
 &+\lambda^2 E_0\Big( \exp\big(-\int_0^{ T_{r_{\varepsilon}}} U^{\vec{\lambda},\vec{x},\varepsilon}(B(s)) ds\big)1(T_{r_{\varepsilon}}<\infty) (-U_{1,2}^{\vec{\lambda},\vec{x},\varepsilon}(B( T_{r_{\varepsilon}}))\Big)\nonumber\\
 \equiv&K_1+K_2.
\end{align}

We first consider $K_2$. On $\{T_{r_\varepsilon}<\infty\}$ we may set $x_\varepsilon(\omega)=B(T_{r_\varepsilon})$ and choose $i(\omega)$ so that $|x_i-x_\varepsilon|\geq \Delta/2$. By the definition of $T_{r_\varepsilon}, |x_i-x_\varepsilon|\geq r_\varepsilon=2\varepsilon, $ and so $|x_i-x_\varepsilon|\geq \frac{1}{2}(\Delta \vee r_\varepsilon).$ Lemma \ref{l4.2}(b) and the above imply
\[
\lambda^2(-U_{1,2}^{\vec{\lambda},\vec{x},\varepsilon}(B( T_{r_{\varepsilon}})))\leq 4\cdot {2^p} ( D^{\lambda/2}(2)(\Delta \vee r_\varepsilon)^{-p} 2^p) \varepsilon^{p-2}\leq c (\Delta \vee r_\varepsilon)^{-p}\varepsilon^{p-2}.
\]
 This shows that 
 \begin{align}\label{e4.5}
 K_2\leq c(\Delta \vee r_\varepsilon)^{-p}\varepsilon^{p-2} \sum_{i=1}^2 E_0\Big( 1(T_{r_{\varepsilon}}^i<\infty)\exp\big(-\int_0^{ T_{r_{\varepsilon}}^i} U^{\vec{\lambda},\vec{x},\varepsilon}(B(s)) ds\big) \Big).
 \end{align}
 Use \eqref{e4.20} and Corollary \ref{c1.4}(a) with $|B(s)-x_i|\geq r_\varepsilon=2\varepsilon$ and $R=2$ to see that
 \begin{align}\label{e4.11}
 U^{\vec{\lambda},\vec{x},\varepsilon}(B(s))\geq U^{\lambda\varepsilon^{-2},\varepsilon}(B(s)-x_i)\geq& U^{\infty,\varepsilon}(B(s)-x_i)- 2^p |B(s)-x_i|^{-p} D^{\lambda}(2)\varepsilon^{p-2}\nonumber \\
 \geq& V^{\infty}(B(s)-x_i)- 2^p |B(s)-x_i|^{-p} D^{\lambda}(2)\varepsilon^{p-2},
 \end{align}
 where the last follows by using \eqref{Uinflb} and scaling to see that $U^{\infty,\varepsilon}(x)=\varepsilon^{-2} U^{\infty,1}(x/\varepsilon) \geq \varepsilon^{-2} V^\infty(x/\varepsilon)=V^\infty(x)$ for all $|x|/\varepsilon>1$.
 Let $\tau_{r_\varepsilon}=\inf\{t: |B_t|\leq r_\varepsilon\}$ and let $\mu,\nu$ be as in \eqref{e1.5}. Use the above in \eqref{e4.5} and then use Brownian scaling to see that for $i=1,2$,
  \begin{align}\label{e4.6}
 &E_0\Big( 1(T_{r_{\varepsilon}}^i<\infty)\exp\Big(-\int_0^{ T_{r_{\varepsilon}}^i} U^{\vec{\lambda},\vec{x},\varepsilon}(B(s)) ds\Big) \Big)\nonumber\\
\leq&E_{-x_i}\Big( 1(\tau_{r_{\varepsilon}}<\infty)\exp\Big(\int_0^{ \tau_{r_{\varepsilon}}} \frac{2^p D^{\lambda}(2)\varepsilon^{p-2}}{|B(s)|^{p}}  ds\Big) \exp\Big(-\int_0^{ \tau_{r_{\varepsilon}}} \frac{2(4-d)}{|B(s)|^{2}}  ds\Big) \Big)\nonumber\\
\leq&E_{-x_i/r_{\varepsilon}}\Big( 1(\tau_{1}<\infty)\exp\Big(\int_0^{ \tau_{1}} \frac{2^p D^{\lambda}(2)\varepsilon^{p-2} r_\varepsilon^{2-p}}{|B(s)|^{p}}  ds\Big) \exp\Big(-\int_0^{ \tau_{1}} \frac{2(4-d)}{|B(s)|^{2}}  ds\Big) \Big)\nonumber\\
=& E_{|x_i|/r_{\varepsilon}}^{(2+2\nu)}\Big( \exp\Big(\int_0^{ \tau_{1}} \frac{4 D^{\lambda}(2)}{\rho_s^{p}}  ds\Big)\Big|\tau_{1}<\infty \Big) (|x_i|/r_{\varepsilon})^{-p}
  \end{align}
where we have used 
Lemma~\ref{lem:22_7_1} in the last line, and recalled that $p=\nu+\mu$. Choose $\lambda>12$ large such that 
\[2\gamma \equiv 2\cdot 4 D^{\lambda}(2) \leq 2(4-d)< \nu^2,\]  and then apply Lemma \ref{expbound2} 
 to conclude that \eqref{e4.6} is bounded by
\[c_{\ref{expbound2}} (p,\nu) (|x_i|/r_{\varepsilon})^{-p}\leq c_{\ref{expbound2} }(p,\nu) \varepsilon_0^{-p} r_{\varepsilon}^{p}.\]
So \eqref{e4.5} becomes 
 \begin{align}\label{e4.7}
K_2\leq c(\Delta \vee r_\varepsilon)^{-p}\varepsilon^{p-2} 2 c_{\ref{expbound2} }(p,\nu) \varepsilon_0^{-p} r_{\varepsilon}^{p} \leq c(\varepsilon_0)\Delta^{2-p} \varepsilon^{p-2} r_\varepsilon^{p-2}=2^{p-2} c(\varepsilon_0)\Delta^{2-p} \varepsilon^{2(p-2)}.
 \end{align}
 
In view of \eqref{e4.8}, \eqref{e4.9} and \eqref{e4.7}, it remains to prove
\begin{align}\label{e4.10}
K_1\leq C(\varepsilon_0) \Delta^{2-p} \varepsilon^{2(p-2)}.
 \end{align}
Let $\Delta_i=x_{3-i}-x_i$, so that $|\Delta_i|=\Delta$. Let $T_{r_\varepsilon}^{',i}=\inf\{t: |B_t|\leq r_\varepsilon \text{ or } |B_t-\Delta_i|\leq r_\varepsilon\}$. Lemma \ref{l4.2}(a) and then \eqref{e4.11} imply that
\begin{align}\label{e4.12}
K_1\leq \lambda^2 \frac{1}{\lambda^2} (2^{p+1}\varepsilon^{p-2} &D^{\lambda/2}(2))^2 E_0\Big( \int_0^{T_{r_\varepsilon}} \prod_{i=1}^2 |B_t-x_i|^{-p}\exp\Big(-\int_0^{t} U^{\vec{\lambda},\vec{x},\varepsilon}(B(s)) ds\Big)dt\Big)\nonumber\\
\leq c \varepsilon^{2(p-2)}\sum_{i=1}^2 &E_{-x_i} \Big( \int_0^{T_{r_\varepsilon}^{',i}}  |B_t|^{-p} |B_t-\Delta_i|^{-p}1( |B_t|\leq |B_t-\Delta_i| )\nonumber\\
&\times \exp\Big(\int_0^{ t} \frac{2^p D^{\lambda}(2)\varepsilon^{p-2}}{|B(s)|^{p}}  ds\Big) \exp\Big(-\int_0^{ t} \frac{2(4-d)}{|B(s)|^{2}}  ds\Big) dt\Big).
\end{align}
On $\{ |B_t|\leq |B_t-\Delta_i| \}$, we have \[  \Delta=|\Delta_i|\leq |B_t-\Delta_i|+|B_t|\leq 2|B_t-\Delta_i|, \] and hence
\[|B_t-\Delta_i|^{-p} \leq \Bigl(\frac{1}{2}\Delta \vee |B_t|\Bigr)^{-p} \leq 2^p (\Delta^{-p}\wedge |B_t|^{-p}).\]
Use $T_{r_\varepsilon}^{',i}\leq \tau_{r_\varepsilon}$ and Brownian scaling to see that
\begin{align*}
K_1\leq c 2^{p} \varepsilon^{2(p-2)}\sum_{i=1}^2 E_{-x_i} &\Big( \int_0^{\tau_{r_\varepsilon}}  |B_t|^{-p} (|B_t|^{-p}\wedge \Delta^{-p})\\
&\times \exp\Big(\int_0^{ t} \frac{2^p D^{\lambda}(2)\varepsilon^{p-2}}{|B(s)|^{p}}  ds\Big) \exp\Big(-\int_0^{ t} \frac{2(4-d)}{|B(s)|^{2}}  ds\Big) dt\Big)\\
\leq c \varepsilon^{2(p-2)}\sum_{i=1}^2 E_{-x_i/r_\varepsilon} &\Big( \int_0^{\tau_{1}} r_\varepsilon^{2-2p} |B_t|^{-p} (|B_t|^{-p}\wedge (\Delta/r_\varepsilon)^{-p}) \\
&\times \exp\Big(\int_0^{ t} \frac{2^p D^{\lambda}(2)\varepsilon^{p-2} r_\varepsilon^{2-p}}{|B(s)|^{p}}  ds\Big) \exp\Big(-\int_0^{ t} \frac{2(4-d)}{|B(s)|^{2}}  ds\Big)dt \Big)\\
= c \varepsilon^{-2}\sum_{i=1}^2 \int_0^\infty E_{-x_i/r_\varepsilon} &\Big( 1(t<\tau_1) |B(t\wedge \tau_1)|^{-p} (|B(t\wedge \tau_1)|^{-p}\wedge (\Delta/r_\varepsilon)^{-p}) \\
&\times \exp\Big(\int_0^{t\wedge \tau_1} \frac{4D^{\lambda}(2)}{|B(s)|^{p}}  ds\Big) \exp\Big(-\int_0^{t\wedge \tau_1} \frac{2(4-d)}{|B(s)|^{2}}  ds\Big)\Big) dt
\end{align*}
Now let $\delta=4D^{\lambda}(2)$,  $\mu, \nu$ be as in \eqref{e1.5}, and use Lemma~\ref{lem:22_7_1} as in 
\eqref{e4.6}, to get
\begin{align*}
K_1\leq c \varepsilon^{-2}\sum_{i=1}^2 \int_0^\infty (|x_i|/r_\varepsilon)^{\nu-\mu}E_{|x_i|/r_\varepsilon}^{(2+2\nu)} &\Big( 1(t<\tau_1) {\rho(t\wedge \tau_1)}^{-p} ({\rho(t\wedge \tau_1)}^{-p}\wedge (\Delta/r_\varepsilon)^{-p}) \\
&\times \exp\big(\int_0^{t\wedge \tau_1} \delta \rho_s^{-p}  ds\big) {\rho(t\wedge \tau_1)}^{-\nu+\mu} \Big)dt\\
= c  \varepsilon^{\mu-\nu-2} \sum_{i=1}^2 |x_i|^{\nu-\mu}E_{|x_i|/r_\varepsilon}^{(2+2\nu)} \Big(  \int_0^{\tau_1} &{\rho_t}^{-p-\nu+\mu} (\rho_t^{-p}\wedge (\Delta/r_\varepsilon)^{-p}) \exp\big(\int_0^{t} \delta \rho_s^{-p}  ds\big) dt\Big).
\end{align*}
We interrupt the proof of the proposition for another auxiliary result from \cite{MP17}.
\begin{lemma} \label{l3.1}
There is some universal constant $c_{\ref{l3.1}}>0$ such that for any $r>0$ with $r<(|x_i| \wedge \Delta)$ and $0<\delta<(p-2)(2-\mu)$, we have
\begin{align*}
E_{|x_i|/r}^{(2+2\nu)} \Big(  \int_0^{\tau_1} {\rho_t}^{-p-\nu+\mu}& (\rho_t^{-p}\wedge (\Delta/r)^{-p}) \exp\big(\int_0^{t} \delta \rho_s^{-p}  ds\big) dt\Big) \leq c_{\ref{l3.1}}  r^{-2+2p+\nu-\mu} |x_i|^{-2\nu} \Delta^{2-p}.
\end{align*}
\end{lemma}
\begin{proof}
This is included in the proof of Proposition~6.1 of \cite{MP17} with $r=r_{\lambda}$. In particular, the above expectation appears in (9.23) of \cite{MP17} and is bounded by $eJ_i$ in (9.27) of that paper.  Following the inequalities in that work, noting we only 
need consider Case 1 or Case 3 (the latter with $r\le |x_i|\le \Delta$) at the end of the proof, we arrive at the above bound. \pf{[Note to self: Case 3 ends with a bound of 
$$cr^{p+2\nu-2}|x_i|^{2(1-\nu)}|\Delta^{-p}=cr^{-2+2p+\nu-\mu}|x_i|^{-2\nu}|x_i|^2\Delta^{-p}\le cr^{-2+2p+\nu-\mu}|x_i|^{-2\nu}\Delta^{2-p}.]$$}
\end{proof}
Returning now to the proof of Proposition \ref{p3.1}. Pick $\lambda>12$ large such that $\delta<(p-2)(2-\mu)$.  Note we assumed $|x_i|\geq \varepsilon_0>r_\varepsilon$ by \eqref{e4.81} and  $\Delta=|x_1-x_2|> 5\varepsilon>r_\varepsilon$ at the very beginning of this section. So use Lemma \ref{l3.1} applied with $r=r_\varepsilon$ to see that
\[K_1\leq c \varepsilon^{\mu-\nu-2} \sum_{i=1}^2 |x_i|^{{\nu}-\mu} c_{\ref{l3.1}}  r_\varepsilon^{-2+2p+\nu-\mu} |x_i|^{-2\nu} \Delta^{2-p}\leq C \varepsilon_0^{-\mu-\nu}\Delta^{2-p} \varepsilon^{2p-4}.\]  This gives \eqref{e4.10}, and so the proof is complete.
\end{proof}}

\bibliographystyle{plain}
\def\cprime{$'$}

\end{document}